\documentclass[11pt]{article}

\usepackage{latexsym}
\usepackage{amssymb}
\usepackage{amsthm}
\usepackage{amscd}
\usepackage{amsmath}
\usepackage{verbatim}
\usepackage{xcolor}
\usepackage{enumerate}
\usepackage{graphicx}

\newtheorem{theorem}{Theorem}[section]

\newtheorem{lemma}[theorem]{Lemma}
\newtheorem*{lemma*}{Lemma}
\newtheorem*{lemmacoords}{Lemma \ref{coords}}
\newtheorem*{lemmaC2}{Lemma \ref{C2}}
\newtheorem*{definition*}{Definition}
\newtheorem{definition}[theorem]{Definition}

\newtheorem{corollary}[theorem]{Corollary}

\theoremstyle{definition}
\newtheorem*{example}{Example}
\newtheorem*{remark}{Remark}

\numberwithin{equation}{section}

\let\inf\relax \DeclareMathOperator*\inf{\vphantom{p}inf}

\newcommand{\leqn}{\begin{equation}\label}
\def\endeqn{\end{equation}}

      \def\RR{\mathbb{R}}
\newcommand{\Ss}{\mathbb{S}}
\newcommand{\NN}{\mathbb{N}}
  
\def\tr{\mathop\mathrm{tr}\nolimits}
\def\dim{\mathrm{dim}} 
 \def\ZZ{\mathbb{Z}} 
           
\def\vphi{\varphi}

\newcommand{\cB}{\mathcal{B}}

\newcommand{\cC}{\mathcal{C}}

\newcommand{\cE}{\mathcal{E}}
\newcommand{\cQ}{\mathcal{Q}}

\newcommand{\cL}{\mathcal{L}}

\newcommand{\cH}{\mathcal{H}}

\newcommand{\Id}{\mathrm{Id}}

\newcommand{\supp}{\mathrm{supp\,}}

\newcommand{\D}[3]{D^{#1}\hspace{-3 pt}\left(#2,#3\right)}
\newcommand{\Da}[3]{\twid{D}^{#1}\hspace{-3 pt}\left(#2,#3\right)}
\newcommand{\hd}[3]{d^{#1}\hspace{-3 pt}\left(#2,#3\right)}
\newcommand{\hda}[3]{\twid{d}^{#1}\hspace{-3 pt}\left(#2,#3\right)}
\newcommand{\frexp}[2]{^{\frac{#1}{#2}}}
\newcommand{\inv}{^{-1}}
\newcommand{\sleq}{\leqslant}
\newcommand{\sgeq}{\geqslant}
\newcommand{\twid}[1]{\widetilde{#1}}
\newcommand{\meang}{\measuredangle}
\newcommand{\Tan}{\mathrm{Tan}}
\newcommand{\res}{\hbox{{\vrule height .22cm}{\leaders\hrule\hskip .2cm}}}

\makeatletter
\renewcommand{\@makefnmark}{\mbox{\textsuperscript{}}}
\makeatother

\def\adots{\mathinner{\mkern2mu\raise0pt\hbox{.}  
\mkern2mu\raise4pt\hbox{.}\mkern1mu
\raise7pt\vbox{\kern7pt\hbox{.}}\mkern1mu}}

\allowdisplaybreaks[1]

\begin{document}

\title{Singular points of H\"older asymptotically optimally doubling measures}
\author{Stephen Lewis
\footnote{The author was paritally supported by NSF DMS 0838212 during this research}
\\ Department of Mathematics\\ University of Washington\\  \textsf{stedalew@u.washington.edu}
}

\date{}

\maketitle

\abstract{We consider the question of how the doubling characteristic of a measure determines the regularity of its support. The question was considered in \cite{DKT} for codimension 1 under a crucial assumption of flatness, and later in \cite{PTT} in higher codimension. However, the studies leave open the geometry of the support of such measures in a neighborhood about a non-flat point of the support. We here answer the question (in an almost classical sense) for codimension-1 H\"older doubling measures in $\RR^4$.}

\vspace{10pt}

\noindent AMS-Subject Classification: 28E99, 28A75


\section{Introduction}

In this paper we study the relationship between the optimal doubling properties of a measure and the regularity and geometry of its support. This question had been considered in \cite{DKT} and \cite{PTT} where one of their crucial hypothesis was a baseline assumption of flatness. Roughly speaking they showed that if a Radon measure doubles asymptotically like Lesbegue measure of the appropriate dimension and the support of the measure is sufficiently flat then it can be locally parameterized as the image of an open set of the plane. Their study leaves open the question of what happens in the presence of non-flat points. In this paper we address that question.

An $(n-1)$-uniform measure on $\RR^n$ is one for which the measure of any ball of radius $r$ centered in the support is the same as a ball of $m$-dimensional Lebesgue measure, $\omega_{n-1} r^{n-1}$. Kowalski and Preiss showed in \cite{KP} that an $(n-1)$-uniform measure on $\RR^n$ is, up to translation and rotation, surface measure on either a hyperplane or the cone $\cC = \{x_1^2 + x_2^2 + x_3^2 = x_4^2\}$ (hereforward called a KP cone). An $(n-1)$-asymptotically optimally doubling measure is one whose asymptotic doubling properties coincide with $(n-1)$-dimensional Lebesgue measure (see Definition \ref{mureg}). Our first main result in Section 3 is to show that the support of an $(n-1)$-asymptotically optimally doubling measure at a nonflat point is well approximated by a (translated and rotated) KP cone (see Definition \ref{thetas}). In Sections 4 and 5 we suppose that the Lebesgue doubling holds up to a H\"older error term (see Definition \ref{mureg}), and show that in some neighborhood of a nonflat point, the support admits a $C^{1,\beta}$ parametrization by a KP cone. The KP cone is of course singular, so adequate care is taken to make this precise.  In \cite{KT}, the authors showed that the tangent measures of an $(n-1)$-asymptotically optimally doubling measure are $(n-1)$-uniform. Coupling this with the classification of \cite{KP}, the appearance of KP cones should not be surprising in the context of this paper.

 Section 5 fits into a larger picture of using set approximations to construct parametrizations. The simplest case of this is Reifenberg's topological disk theorem which roughly speaking says that if $\Sigma$ is a closed set containing 0 such that every point in $B(0,1)\cap \Sigma$ is well approximated by a plane at all scales $0<r\sleq 1$, then $\Sigma\cap B(0,1)$ admits a $C^{0,\alpha}$ parametrization by a disk (see Theorem \ref{DKTdisk}). Similar situations and generalizations include \cite{Ta} where $C^{1,\beta}$ parametrizations are constructed for approximately minimal sets in $\RR^3$, \cite{DT} where $C^{0,\alpha}$ parametrizations for sets which have holes are constructed, and \cite{DDT} where $C^{0,\alpha}$ parametrizations for sets which are very close to the minimal cones of \cite{Ta} are constructed.

We begin by giving some precise definitions. We take $B(x,r)$ to be the closed ball of center $x$ and radius $r$ in $\RR^n$.
Let $A, B\subseteq \RR^n$ be nonempty sets. We define
\begin{equation}\label{ddef}
d(A, B) = \sup_{a\in A}\inf_{b\in B} |a-b|.
\end{equation}
Note that $d$ is neither a symmetric quantity nor a metric, but does satisfy the triangle inequality. Further, if $A$ and $B$ are closed sets, then $A\subseteq B$ if and only if $\hd{} A B = 0$.
We then define the \emph{Hausdorff distance} between $A$ and $B$ as
\begin{equation}\label{Ddef}
D(A, B) = \max( \hd{}A B, \hd{} B A).
\end{equation}
Note that $D$ forms a metric on the nonempty compact subsets of $\RR^n$.

For $x\in \RR^n$, $r>0$, and sets $A, B\subseteq \RR^n$ such that $A\cap B(x,r)\neq \emptyset$ and $B\cap B(x,r)\neq \emptyset$, we define
\begin{equation}\label{dxrdef}
\hd{x, r}A B = \frac 1 r \hd{}{A\cap B(x,r)}{B\cap B(x,r)}.
\end{equation}
Note that $d^{x,r}$ is neither a symmetric quantity nor a metric, but does satisfy the triangle inequality. Further, if $A$ and $B$ are closed sets intersecting $B(x,r)$, then $A\cap B(x,r)\subseteq B\cap B(x,r)$ if and only if $\hd{x,r} A B = 0$. We then define the \emph{relative Hausdorff distance at point $x$ and scale $r$} to be 
\begin{equation}\label{Dxrdef}
\D{x,r}A B = \max( \hd{x,r} A B, \hd{x,r} B A).
\end{equation}
Note that $D^{x,r}$ forms a pseudometric on closed subsets of $\RR^n$ intersecting $B(x,r)$ (that is, it satisfies the triangle inequality). Further, it forms a metric on the set of closed subsets of $\RR^n$ modulo $B(x,r)^c$ (that is, $A\sim B$ if $A\cap B(x,r) = B\cap B(x,r)$). We note also that $0\sleq \D{x,r}AB\sleq 2$.

Let a \emph{KP cone based at $y\in\RR^n$} be a set $\cC$ such that in some orthonormal coordinates $(x_1,\ldots,x_n)$ centered at the origin,
$$\cC - y = \{x_4^2 = x_1^2+ x_2^2+x_3^2\}.$$

\begin{definition}\label{thetas}\emph{For a set $\Sigma\subseteq \RR^n$, an integer $0< m \sleq n$, a point $x\in\Sigma$ and a scale $r>0$, we define the following three quantities:
\begin{itemize}
  \item$\theta_\Sigma^{P(m)}(x,r) = \inf\{ \D{x,r}\Sigma L\mid L\mbox{ is an affine $m$-plane containing }x\}$. When $m$ is clear (usually $m=n-1$) or unimportant to specify, then we will simplify the notation by setting $\theta_\Sigma^P(x,r) = \theta_\Sigma^{P(m)}(x,r)$.
  \item $\theta_\Sigma^C(x,r) =  \inf\{ \D{x,r}\Sigma \cC\mid \cC\mbox{ is a KP cone based at }x\}$.
  \item $\vartheta_\Sigma^C(x,r) = \inf\{ \D{x,r}\Sigma \cC\mid \cC\mbox{ is a KP cone containing }x\}$.
\end{itemize}
Roughly speaking, these quantities estimate how well a set is approximated in $B(x,r)$ by a plane containing $x$, a KP cone based at $x$, or KP cone containing $x$ respectively (see Figure 1).
We call a point $x\in \Sigma$ \emph{flat} if $\theta^P_\Sigma(x,r)\to 0$ as $r\downarrow 0$. A point which is not flat is \emph{nonflat}.  The set $\Sigma$ is said to be \emph{$\delta$-Reifenberg flat} if for all compact sets $K\subseteq \Sigma$, there exists a radius $r_K>0$ such that for all $x\in K$, $0<r\sleq r_K$, $\theta^P_\Sigma(x,r)\sleq \delta$. The set $\Sigma$ is said to be \emph{Reifenberg flat with vanishing constant} if it is $\delta$-Reifenberg flat for every $\delta>0$. Equivalently, $\Sigma$ is Reifenberg flat with vanishing constant if and only if the quantity $\theta^P_\Sigma(x,r)\to 0$ uniformly on compact sets as $r\downarrow0$.
}\end{definition}

\begin{figure}
\begin{center}
\includegraphics[width=300pt]{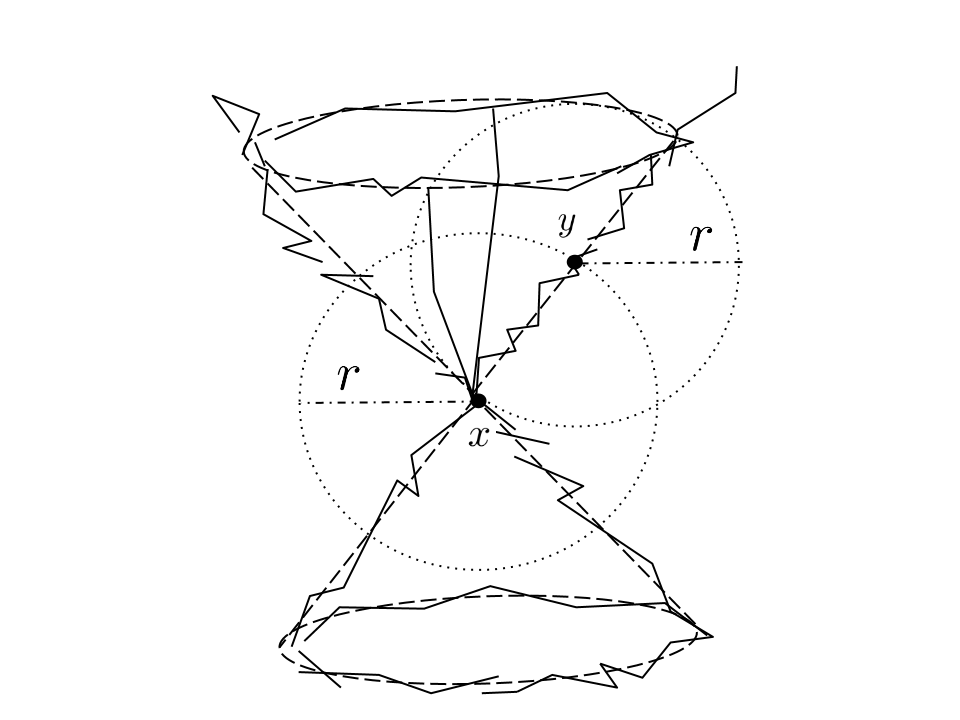}
\caption{A set $\Sigma$ such that $\theta_\Sigma^C(x,r)$ and $\vartheta_\Sigma^C(y,r)$ are small, $\theta_\Sigma^C(y,r)$ big.}
\end{center}
\end{figure}

\begin{remark}  Note that the quantity $\vartheta^{P(m)}(x,r)$, appropriately defined, would be the same as the quantity $\theta_\Sigma^{P(m)}(x,r)$.\end{remark}

We now give increasingly strong conditions on the regularity of a measure $\mu$. First, we define the \emph{support of $\mu$} as 
\begin{equation}\label{suppdef}
\supp(\mu) = \{x : \mu( B(x,r) )> 0 \mbox{ for all } r>0\}.
\end{equation} 
The support may be alternatively viewed as the minimal closed set of comeasure 0.

\begin{definition}\label{mureg} Let $\mu$ be a nonzero Radon measure on $\RR^n$ and $\Sigma = \supp \mu$.
 \begin{enumerate}[(1)]
 \item We say $\mu$ is \emph{locally doubling at $x\in \Sigma$} if there exists a neighborhood $U$ of $x$ and a constant $C$ such that for all $y\in\Sigma\cap U$ and all $r>0$ such that $B(y,r)\subseteq U$,
     $$\frac {\mu(B(y,r))}{\mu(B(y,r/2))} \sleq C.$$
 \item We say that $\mu$ is \emph{locally doubling} if it is locally doubling at all $x\in\Sigma$.
 \item For $x\in \Sigma$, $r>0$, and and integer $0< m\sleq n$ (understood) we define
 \leqn{Rdef1}
 R(\mu, x, r) = \sup\left\{ \left| \frac{\mu(B(x,\tau r'))}{\mu(B(x,r'))} - \tau^m \right| :0<r'\sleq r, \tau\in[1/2,1] \right\}.
 \endeqn
 For $K\subseteq\Sigma$, we define
 \leqn{Rdef2}
 R(\mu, K, r) = \sup_{x\in K} R(\mu, x, r).
 \endeqn
 \item For an integer $0<m\sleq n$, we say that $\mu$ is \emph{$m$-asymptotically optimally doubling} if for all compact sets $K\subseteq \Sigma$ and $\delta>0$, there exists a radius $r_0>0$ such that
 \leqn{AOD1}
 R(\mu, K, r_0) < \delta.
 \endeqn
 That is, for all $x\in K$, $0<r\sleq r_0$ and $\tau\in [1/2,1]$,
     \begin{equation}\label{AOD}
     \left|\frac{\mu(B(x,\tau\,r))}{\mu(B(x,r))} - \tau^m\right|<\delta.
     \end{equation}
     Equivalently, we may say that the quantity $\frac{\mu(B(x,\tau r))}{\mu(B(x,r))}-\tau^m$ converges to $0$ uniformly on compact sets as $r\downarrow 0$ (independent of $\tau\in[1/2,1]$). In the case where $m$ is understood, we will drop it from the beginning.
 \item For $\alpha>0$ and an integer $0<m\sleq n$, we say that $\mu$ is \emph{$(\alpha,m)$-H\"older asymptotically optimally doubling} if for all compact sets $K\subseteq \Sigma$, there exist a constant $C_K$ and a radius $r_0>0$ such that for $0<r\sleq r_0$,
\begin{equation}
R(\mu, K, r)\sleq C_K r^\alpha.
\endeqn
That is, for all $x\in K$, $0<r\sleq r_0$, $\tau\in[1/2,1]$,
     \begin{equation}
     \left|\frac{\mu(B(x,\tau\,r))}{\mu(B(x,r))} - \tau^m\right|\sleq C_K r^\alpha.
     \end{equation}
     In the case where $\alpha$ and $m$ are understood, we will drop them from the beginning.
 \item For $\alpha >0$ and an integer $0<m\sleq n$, we say $\mu$ is \emph{$(\alpha,m)$-H\"older asymptotically uniform} if for each compact set $K\subseteq \Sigma$, there exist a constant $C_K$ and a radius $r_0>0$ such that for all $x\in K$ and $0<r\sleq r_0$,
     \begin{equation}
     \left|\frac{\mu(B(x,r))}{\omega_m r^m}-1\right|\sleq C_K r^\alpha,
     \end{equation}
where $\omega_m = \cL^m(B^m(0,1))$.
 \item For an integer $0<m\sleq n$, we say that $\mu$ is \emph{$m$-uniform} if for all $x\in \Sigma$, $r>0$,
\begin{equation}
\mu(B(x,r)) = \omega_m r^m,
\end{equation}
where $\omega_m = \cL^m(B^m(0,1))$.

 \end{enumerate}
 \end{definition}

Although Definition \ref{mureg}(6) gives a stronger property than Definition \ref{mureg}(5), we observe the following result. We use the notation $\mu\res_g$ to be the measure $\mu\res_g(A) = \int_A g \,d\mu$.

\begin{lemma}[\cite{DKT},\cite{PTT}]\label{HAL}
 Let $\mu$ be an $(\alpha,m)-$H\"older asymptotically optimally doubling measure on $\RR^n$. Then the density $f(x) = \lim_{r\downarrow0} \theta(x,r)$ exists and is finite and nonzero at $\mu$-almost every $x\in \RR^n$ and $\nu = \mu\res_{1/f}$ is $(\frac{\alpha}{\alpha+1},m)$-H\"older asymptotically uniform.
\end{lemma}

As a consequence of Lemma \ref{HAL}, we note that to study the support of H\"older asymptotically optimally doubling measures, we can study the support of H\"older asymptotically uniform measures. We note, however, that this is not true in general for asymptotically optimally doubling measures. That is, there are asymptotically optimally doubling measures whose measure is not given by the density of the set.

We now give precise statements of the theorems mentioned earlier. We begin with a theorem from \cite{DKT} which says that 

\begin{theorem}[\cite{DKT}]\label{DKTaod}
Suppose that $\mu$ is an $(n-1)$-asymptotically optimally doubling measure on $\RR^n$, and $\Sigma = \supp \mu$. If $n>3$, suppose also that $\Sigma$ is $1/(4\sqrt 2)$-Reifenberg flat. Then $\Sigma$ is Reifenberg flat with vanishing constant.
\end{theorem}

In \cite{DKT}, the authors showed a similar statement for arbitrary $m$ (i.e., arbitrary codimension). However, since $m=n-1$ (codimension 1) will be our focus, we omit the generalization and refer the curious reader to \cite{DKT}. We expand their study from the specialized setting of flat points to all $(n-1)-$asymptotically optimally doubling measures and show a global statement akin to Reifenberg flatness, as well as showing that at a nonflat point $x$ the support is well approximated by a KP cone based at $x$.

\begin{theorem}\label{introAOD}
Suppose that $\mu$ is an $(n-1)$-asymptotically optimally doulbing measure on $\RR^n$ and $\Sigma = \supp\mu$. Then 
$$\min(\vartheta_\Sigma^C(x,r),\theta_\Sigma^P(x,r))\to 0\quad\mbox{ as }\quad r\downarrow 0$$
uniformly on compact subsets. Further, if $x$ is a nonflat point, then
$$\theta_\Sigma^C(x,r)\to 0\quad\mbox{ as } r\downarrow 0.$$
\end{theorem}

\cite{DKT} also gave strong regularity results about measures which are $(\alpha,n-1)$-H\"older asymptotically uniform in the special setting of flatness.

\begin{theorem}[\cite{DKT}]\label{DKThaod}
For all $\alpha>0$, there exists $\beta = \beta(\alpha)>0$ with the following property. Suppose that $\mu$ is $(\alpha,n-1)$-H\"older asymptotically uniform. If $n>3$, suppose also that $\supp\mu$ is $\frac 1{4\sqrt 2}$-Reifenberg flat. Then $\Sigma = \supp\mu$ is a $C^{1,\beta}$-manifold of dimension $n-1$.
\end{theorem}

In \cite{DKT}, the authors showed a quantitative version of this statement, but we omit the extra complication until it will prove useful to us later in the paper (see Theorem \ref{DKTprop}). We complete the study of $(n-1,\alpha)-$H\"older asymptotically optimally doubling measures on $\RR^4$ by giving a parametrization of the support in the neighborhood of a nonflat point by a KP cone. Theorem \ref{introHAOD} is our main result.

\begin{theorem}\label{introHAOD}
For all $\alpha>0$, there exists $\beta = \beta(\alpha)>0$ with the following property. Suppose that $\mu$ is an $(\alpha,3)-$H\"older asymptotically uniform measure on $\RR^4$ and $x\in\Sigma = \supp \mu$ is a nonflat point. Then there exists a neighborhood of $x$ which is $C^{1,\beta}$ diffeomorphic to an open piece of the KP cone containing the singular point $0$.
\end{theorem}

We note that the lowest dimension in which the KP cone appears is $\RR^4$, and this is the only dimension to which Theorem \ref{introHAOD} applies. In this case, the singular set of a KP cone is a single point, and this makes $n=4$ the simplest case to construct a parametrization. Future work includes the question of local parametrization of the support about a nonflat point in dimension $n\sgeq 5$. Further, in Section 4 we profit implicitly several times from the following fact; let $\cC$ be a rotation by $O$ of the KP cone $\{x_4^2 = x_1^2 + x_2^2 + x_3^2\}$ in $\RR^4$. If we know $O(x_4)$, then we know what $\cC$ is.

The structure of this paper is as follows: In Section 2, we investigate the geometry of sets approximated by cones (i.e., dilation invariant sets, of which the KP cone and planes are examples). In particular, we investigate the behavior of a set $\Sigma$ under different assumptions on $\theta_\Sigma^C$, $\vartheta_\Sigma^C$, and $\theta_\Sigma^P$ at differening locations and scales, often exploiting their interplay. In Section 3, we investigate $(n-1)-$asymptotically optimally doubling measures, culminating in Theorem \ref{introAOD}. Section 4 begins our study of $(n-1)$-H\"older asymptotically optimally doubling measures on $\RR^4$. We prove Theorem \ref{introHAOD} in two main parts. In Section 4, we prove H\"older estimates on the quanities $\theta_\Sigma^C$ and $\theta_\Sigma^P$ at different scales (see Theorem \ref{theorem}). In Section 5, we use the demonstrated estimates to construct a local $C^{1,\beta}$ parametrization of the set by a KP cone (see Theorem \ref{parametrization}).

\section{The Geometry of Sets Approximated by Planes and KP Cones}

In this section, our goal is to study sets which are well approximated by planes and KP cones. In Section 2.1, we show some nice properties of a modified relative Hausdorff distance $\twid D^{x,r}$ on dilation invariant sets. In Section 2.2, we give some simple geometry of planes. In Section 2.3, we study sets well approximated by KP cones in different ways, often with an eye for the interactions between $\theta_\Sigma^C$, $\vartheta_\Sigma^C$, and $\theta_\Sigma^P$ at different points and scales.

\subsection{Behavior of Hausdorff Distance and a Modified Hausdorff Distance}

Our goal in this section is to study the geometry of sets which are well approximated by cones; that is, dilation invariant sets. 

\begin{definition}\emph{
A set $C\subseteq\RR^n$ is a \emph{cone based at $y$} if $y\in C$ and $C-y$ satisfies that $s(C-y) = C-y$  for all $s>0$, .}
\end{definition}

\begin{example}
 Recall that a KP cone based at $y$ is a set $\cC$ which in some orthonormal coordinates $(x_i)$ centered at the origin satisfies
$\cC -y = \{x_4^2 = x_1^2 + x_2^2 + x_3^2\}.$
Note that a KP cone based at $y$ is a cone based at $y$ and that a plane including $y$ is a cone based at $y$.
\end{example}

A useful tool, especially for dilation invariant sets, will be the following. For $x\in \RR^n$, $r>0$, and sets $A, B\subseteq \RR^n$ such that $A\cap B(x,r)\neq \emptyset$ and $B\neq \emptyset$, we define
\begin{equation}\label{daxrdef}
\hda{x, r}A B = \frac 1 r \hd{}{A\cap B(x,r)}{B}
\end{equation}
(see (\ref{dxrdef}) for comparison with $d^{x,r}$). Note that $\twid d^{x,r}$ is neither a symmetric quantity nor a metric and, unlike $d$ and $d^{x,r}$, does not satisfy the triangle inequality. For sets $A, B\subseteq \RR^n$ such that $A\cap B(x,r)\neq \emptyset$ and $B\cap B(x,r)\neq \emptyset$, we define 
\begin{equation}\label{Daxrdef}
\Da{x,r} A B = \max( \hda{x,r}A B , \hda{x,r} B A).
\end{equation}
Note that $\twid D^{x,r}$ is symmetric, and $A\cap B(x,r) = B\cap B(x,r)$ if and only if $\Da{x,r} A B = 0$ , but that $\twid D^{x,r}$ does not satisfy the triangle inequality.

We note that $D^{x,r}$ is a pseudometric, but $\twid D^{x,r}$ is not. However, $\twid D^{x,r}$ enjoys much more stability as $x$ and $r$ vary (for $A$ and $B$ fixed) than $D^{x,r}$ does, and so each one may be argued to be more natural. Further, in practice $\twid D^{x,r}$ is often less finicky to work with than is $D^{x,r}$.  We note also that $\hda{x,r}AB\sleq \hd{x,r}AB$ (and hence $\Da{x,r}AB\sleq\D{x,r}AB$) for all sets $A,B\subseteq\RR^{n}$ interesecting $B(x,r)$.

\begin{lemma}\label{ConeG1}
Let $C$ be a cone based at $y\in\RR^n$ and $\Sigma\subseteq\RR^n$ be any set. Then for any $r>0$ such that $B(y,r)\cap \Sigma\neq \emptyset$, we have that
\begin{equation}\label{ConeG1.-1}
\hd{y,r}\Sigma C= \hda{y,r}\Sigma C\quad\mbox{and}\quad \hda{y,r}C\Sigma\sleq\hd{y,r}C\Sigma\sleq 2\hda{y,r}C\Sigma.
\end{equation}
In particular,
\begin{equation}\label{ConeG1.1}
\Da{y,r}C\Sigma\sleq\D{y,r}C\Sigma\sleq 2\Da{y,r}C\Sigma.
\end{equation}
\end{lemma}
\begin{proof}
Let $C$ be a cone based at $y$ and $\Sigma\subseteq\RR^n$. Without loss of generality, take $y = 0$. Let $B(0,r)\cap \Sigma\neq \emptyset$.  We begin by showing that
\begin{equation}\label{ConeG1.2}
\hd{0,r}\Sigma C= \hda{0,r}\Sigma C.\end{equation}
Because $\hda{0,r}AB\sleq\hd{0,r}AB$ for any $A$ and $B$, it suffices to show that $\hd{0,r}\Sigma\cC\sleq\hda{0,r}\Sigma\cC$. Applying the definition, we must show that for any $x\in \Sigma\cap B(0,r)$, $d(x,C\cap B(0,r)) \sleq d(x,C)$. To do so, we show that if $z\in C$, there is another point $z'\in C\cap B(0,r)$ such that $|x-z'|\sleq |x-z|$. Let $x\in \Sigma\cap B(0,r)$ and $z\in C\setminus B(0,r)$. Let $\ell = \{sz:s\in\RR\}$ be the line through $z$ passing through 0; note that $\ell$ is a cone through the origin and that $\ell\subseteq C$ because $C$ is a cone and $z\in C$. Let $\pi:\RR^n\to\ell$ be the orthogonal projection onto $\ell$. Then the nearest point to $x$ in $\ell$ is $z' = \pi(x)$. But $ | \pi (x) |\sleq |x| \sleq r$, and so $\pi(x)\in C\cap B(0,r)$. Hence, we have shown (\ref{ConeG1.2}).

We now seek to show that
\begin{equation}\label{ConeG1.2'}
\hda{0,r}C\Sigma\sleq \hd{0,r}C\Sigma \sleq 2 \hda{0,r}C\Sigma .\end{equation}
Again, we note that the left inequality is automatic. So, by applying the definition, we must show that for any $z\in C\cap B(0,r)$, $d(z,\Sigma\cap B(0,r)) \sleq 2 \hda{0,r}C\Sigma$. Let $\twid d = \hda{0,r}C\Sigma$. Note that since $\Sigma\cap B(0,r)\neq\emptyset$ and $0\in C$, we have that $\twid d\sleq 1$. Let $z\in C\cap B(0,r)$. Then the point $z' = (1-\twid d)z\in C\cap B(0,r-\twid dr)$, and $|z-z'| = \twid d|z|\leq\twid  d\,r$. By assumption, there exists some point $x\in \Sigma$ with $|x-z'|\sleq \twid d\,r$. Thus, we have that
$|x-z|\sleq |x-z'|+|z'-z| \sleq 2\twid dr.$
Further, since $z'\in B(0,r-\twid dr)$, we have that
$|x|\sleq |x-z'| + |z'|\sleq \twid dr + (1-\twid d)r = r.$
So $x\in B(0,r)$. Hence,
\begin{equation}\label{ConeG1.100}
d(z,\Sigma\cap B(0,r)\sleq 2 \twid dr.
\end{equation}
Because (\ref{ConeG1.100}) holds for all $z\in B(0,r)$, we have shown (\ref{ConeG1.2'}). Thus, we have established (\ref{ConeG1.-1}). We conclude by noting that (\ref{ConeG1.1}) follows immediately from (\ref{ConeG1.-1}) and the definitions of $D^{0,r}$ and $\twid D^{0,r}$.

\end{proof}

\begin{lemma}\label{ConeG2}
Let $C$ be a cone based at $y$ and $\Sigma\subseteq\RR^{n}$. Let $x\in \RR^{n}$ and $r,s>0$ satisfy $B(y,s)\subseteq B(x,r)$ and $B(y,s)\cap \Sigma\neq \emptyset$. Then
\begin{equation}\label{ConeG2.1}
\D{y,s}C\Sigma\sleq 2 \frac rs \Da{x,r}C\Sigma.
\end{equation}
In particular,
\begin{equation}\label{ConeG2.2}
\D{y,s}C\Sigma\sleq 2 \frac rs \D{x,r}C\Sigma.
\end{equation}
\end{lemma}

\begin{proof}
Let all notation and suppositions hold. It follows immediately from the definitions that
\begin{equation}\label{ConeG2.3}
\Da{y,s} C\Sigma \sleq \frac rs \Da{x,r} C\Sigma.
\end{equation}
We then have that (\ref{ConeG2.1}) follows immediately by Lemma \ref{ConeG1}.
We note that (\ref{ConeG2.2}) follows immediately from (\ref{ConeG2.1}).

\end{proof}

\subsection{Geometry of Planes}

In this section we give some basic definitions and lemmas about planes and their geometry. Let $V_1$ and $V_2$ be vector spaces with $\dim V_1\sleq \dim V_2$. We define
\begin{equation}
\Gamma(V_1, V_2) = \hd{0,1}{V_1}{ V_2}.
\end{equation}
It is not hard to see that for any $r>0$, we have that
$
\Gamma(V_1, V_2) = \hd{0,r}{V_1}{V_2}.
$
We extend $\Gamma$ to a pseudometric on the set of all affine planes. Let $P_1$ and $P_2$ be planes $p_1\in P_1, p_2\in P_2$ and $\dim P_1\sleq \dim P_2$. Define
\begin{equation}
\Gamma(P_1, P_2) = \Gamma(P_1-p_1, P_2-p_2).
\end{equation}
That is, we extend $\Gamma$ to arbitrary affine planes by first translating them to pass through the origin. Note that this of course does not depend on the $p_i$. Note that $0\sleq \Gamma(P_1, P_2)\sleq 1$. Further, $\Gamma(P_1, P_2) = 0$ if and only if $P_1 || P_2$, and $\Gamma(P_1, P_2) = 1$ if and only if $P_1$ contains a vector perpendicular to $P_2$. It is not hard to see that if $P_i^\perp$ is an affine orthogonal complement to $P_i$, then
\begin{equation}\label{normals}
\Gamma(P_2^\perp, P_1^\perp) = \Gamma(P_1, P_2).
\end{equation}
It will sometimes be convenient to make reference to the angle between two affine planes. We define the angle between $P_1$ and $P_2$ to be
\begin{equation}
\meang(P_1, P_2) = \arcsin\Gamma(P_1, P_2).
\end{equation}

Alternatively, we may define the angle between two planes as follows. Let $\Ss^{n-1} = \{x\in \RR^n\mid |x| = 1\}$ be the unit sphere in $\RR^n$. Let $d_{\Ss}$ denote the path metric on $\Ss^{n-1}$ defined by $d_{\Ss}(a , b) = \inf\{{\rm len}(\gamma)\mid \gamma:[0,1]\to \Ss^{n-1}, \gamma(0) = a, \gamma(1) = b\}$. For nonempty sets $A,B\subseteq \Ss^{n-1}$, we define
\begin{equation}\label{dsn-1}
d_\Ss(A, B) = \sup_{a\in A} \inf_{b\in B} d_\Ss(a, b).
\end{equation}
(See for comparison (\ref{ddef}) and (\ref{dxrdef}).) For two vector spaces $V_1$ and $V_2$ with $\dim V_1 \sleq \dim V_2$, we then have that
\begin{equation}\label{anglesdistance}
\meang(V_1, V_2) = d_{\Ss}(V_1\cap \Ss^{n-1}, V_2\cap \Ss^{n-1}).
\end{equation}
It follows that for planes $P_1$ and $P_2$ with $\dim P_1\sleq \dim P_2$ and $p_i\in P_i$, that
\begin{equation}\label{planeanglesdistance}
\meang(P_1, P_2) = d_\Ss((P_1-p_1)\cap \Ss^{n-1}, (P_2-p_2)\cap \Ss^{n-1})
\end{equation}
Further, from (\ref{planeanglesdistance}), subadditivity of angles follows. That is, if $P_1, P_2, $ and $P_3$ are planes with $\dim(P_1)\sleq\dim(P_2)\sleq \dim(P_3)$, we get that
\begin{equation}\label{subadditivityofangles}
\meang(P_1, P_3)\sleq \meang(P_1, P_2) + \meang(P_2, P_3).
\endeqn
Further, if $y\in P_1$, $x\in P_2$, then
\begin{equation}\label{distlessthanangle}
\hd{y,r}{P_1}{P_2}\sleq 2\left(\meang(P_1,P_2) + \frac{|x-y|}{r}\right).
\endeqn

\begin{lemma}\label{planes} Let $P_1$ and $P_2$, planes in $\RR^n$, $\dim P_1\sleq \dim P_2$.
\begin{enumerate}[(1)]
\item Let  $y\in P_1$, $r>0$, and $P_2\cap B(y,r)\neq \emptyset$. Then $\meang(P_1, P_2)\sleq \frac {3\pi}{2} \hd{y,r}{P_1}{P_2}$.
\item If $\dim P_1 = \dim P_2 = n-1$ and $\nu_i$ is a normal vector to $P_i$ with $\nu_1\cdot \nu_2\sgeq 0$, then $|\nu_1-\nu_2|\sleq \meang(P_1, P_2)$.
\end{enumerate}
\end{lemma}
\begin{proof}
Without loss of generality, take $y=0$, $r=1$. We first prove that if both $P_1$ and $P_2$ go through 0, then
\begin{equation}\label{planes.1}
\meang(P_1, P_2)\sleq \frac \pi 2 \hd{0,1}{P_1}{P_2}.
\end{equation}
We note that $\arcsin z\sleq (\pi/2) z$ for any $z\sgeq 0$. Hence,
\begin{equation}\label{planes.3}
\meang(P_1,P_2) = \arcsin(\Gamma(P_1, P_2))\sleq \frac \pi2 \hd{0,1}{P_1}{P_2},
\end{equation}
which is (\ref{planes.1}).

Now, suppose that $0\in P_1$, but that $0\not\in P_2$.  Then there exists $p\in P_2$ such that
\begin{equation}\label{planes.4}
|p|\sleq \hd{0,1}{P_1}{P_2}.
\end{equation}
We compute that
\begin{equation}\label{planes.5}
\hd{0,1}{P_1}{P_2-p}\sleq \hd{0,1}{P_1}{ P_2} + \hd{0,1}{P_2-p}{P_2}.
\end{equation}
Because $P_2-p$ is a cone through the origin, Lemma \ref{ConeG1} tells us that
\begin{equation}\label{planes.6}
\hd{0,1}{P_2-p}{P_2}\sleq 2\hda{0,1}{P_2-p}{P_2} = 2|p|\sleq2 \hd{0,1}{P_1}{P_2}.
\end{equation}
Combining (\ref{planes.5}) and (\ref{planes.6}), we get that
\begin{equation}\label{planes.7}
\hd{0,1}{P_1}{ P_2-p}\sleq 3 \hd{0,1}{P_1}{P_2}.
\end{equation}
Because $P_2-p$ goes through the origin, we combine (\ref{planes.1}) and (\ref{planes.7}) to get
\begin{equation}\label{planes.8}
\meang(P_1,P_2) \sleq \frac\pi2\hd{0,1}{P_1}{P_2-p}\sleq \frac{3\pi}2\hd{0,1}{P_1}{P_2}.
\end{equation}

We now prove (2). Without loss of generality, suppose that $P_1$ and $P_2$ are codimension 1 planes through the origin. Suppose that for $i=1,2$,  $\nu_i\perp P_i$, $|\nu_i| = 1$, and that $\nu_1\cdot \nu_2 \sgeq 0$. It follows from (\ref{normals}) and (\ref{planeanglesdistance}) that
\begin{equation}\label{plane.9}
\meang(P_1, P_2) = \meang(P_2^\perp, P_1^\perp) = d_\Ss(\{\pm \nu_2\}, \{\pm\nu_1\}).
\end{equation}
From the fact that $\nu_1\cdot \nu_2\sgeq 0$, it follows that
\begin{equation}\label{plane.10}
d_\Ss(\{\pm \nu_2\}, \{\pm\nu_1\}) = d_\Ss(\nu_2, \nu_1)\sgeq |\nu_2-\nu_1|.
\end{equation}
Putting together (\ref{plane.9}) and (\ref{plane.10}), we prove (2).
\end{proof}

\subsection{Geometry of KP Cone Approximated Sets}

The goal of this section is to analyze the geometry of sets which are well approximated by KP cones.  For a set $\Sigma$ with a tangent plane at $a$,  let $T_a\Sigma$ be the tangent plane to $\Sigma$ at $a$. We use the convention that $a\in T_a\Sigma$.

\begin{lemma}\label{KPflat1}
Let $\cC$ be a KP cone in $\RR^4$ based at $0$. There exists $C_0$ such that for any $a\in \cC$, we have that exactly one of the following holds:
\begin{enumerate}
         \item[(1)] $a=0$, in which case $\theta_\cC^P(a,r) = 1/\sqrt 2$ for all $r>0$,
         \item[(2)] or $a\neq 0$, in which case for all $0<r\sleq |a|/2$, $\D{a,r}\cC{T_a\cC}\sleq C_0 r/|a|$ 
          (and thus $\theta_\cC^P(a,r)\sleq C_0 r/|a|$).
       \end{enumerate}
\end{lemma}
\begin{proof}
Claim (1) follows from an elementary computation.  Let $a\in \cC, |a|=1$.  Note that $\cC\setminus B(0, 1/2)$ is a smooth manifold, and hence there exists a constant $C_0$ such that for all $0<r\sleq 1/2$, $\D{a,r}{\cC}{T_a\cC}\sleq C_0 r$. For any other point $b\in\cC$, $|b|=1$, there is an isometry fixing $0$, taking $a$ to $b$, and taking $\cC$ to $\cC$. Hence, for all $|a| = 1, a\in\cC$, $0<r\sleq 1/2$, $\D{a,r}{\cC}{T_a\cC}\sleq C_0r$. For $a\neq 0$, set $b = a/|a|\in\cC$. Then
$\D{b,r}{\cC}{T_b\cC}\sleq C_0 r$ for $0<r\sleq 1/2$. Because $\cC$ is a cone based at $0$ and $|a|b = a$, we get that $|a|\cC = \cC$ and $|a|T_b \cC = T_a\cC$.  This gives that $ \D{a,|a|r}{\cC}{T_a\cC} = \D{b,r}{\cC}{T_b\cC} \sleq C_0r$ for all $0<r\sleq 1/2$. Plugging in $r/|a|$ for $r$, we get that
$$\D{a,r}{\Sigma}{\cC}\sleq C_0\frac r{|a|}$$
for all $0<r\sleq |a|/2$, and hence have proven (2).
\end{proof}

\begin{remark}
Using Lemma 4.1 of \cite{Bad}, one can obtain that in Lemma \ref{KPflat1}, $C_0 = 1$. This will not be of crucial importance to us, and so we leave the details to the interested reader.
\end{remark}

We now state a lemma of \cite{Bad} of which we employ a variation. First, we must give some definitions from \cite{Bad}. For a fixed $n$ understood, let $\cH_d$ denote the set of nonconstant harmonic polynomials in $\RR^n$ of degree at most $d$. Let $\cQ$ be any set of polynomials from $\RR^n$ to $\RR$. For a set $\Sigma\subseteq \RR^n$, $x\in \Sigma$, and $r>0$, we define
$$\theta_\Sigma^{\cQ}(x,r) = \inf \D{x,r}\Sigma{\Sigma_q}$$
where the infimum is taken over all polynomials in $q\in\cQ$ with $q(x) = 0$, and $\Sigma_q = q\inv(0)$.

\begin{lemma}[\cite{Bad}]\label{Badflat}
For all $n\sgeq 2$, $d\sgeq 1$, and $\delta>0$, there exist $\epsilon>0$ and $\eta>0$ with the following property. Let $\Sigma\subseteq \RR^n$, $x\in \Sigma$, $r>0$, and assume that
$$\sup_{0<r'\sleq r} \theta_\Sigma^{\cH_d}(x,r') < \epsilon.$$
If $\theta_\Sigma^P(x,r)< \eta$, then $\sup_{0<r'\sleq r}\theta_\Sigma^P(x,r')<\delta$.
\end{lemma}

This lemma is stated for sets approximated by zero sets of harmonic polynomials. Although the polynomial $p_{KP}(x) = x_4^2-x_1^2 - x_2^2-x_3^2$ satisfying $\cC = p_{KP}\inv(0)$ is not harmonic, the only fact used in the proof about zero sets of harmonic polynomials is the following.

\begin{theorem}[\cite{Bad}]\label{Badflat2}
For all $n\sgeq 2$ and $d\sgeq 1$, there exists $\delta_{n,d}>0$ and $C_{n,d}$ such that for any harmonic polynomial $h:\RR^n\to \RR$ and any $x\in \Sigma_h = h\inv(0)$,
\begin{enumerate}[(1)]
\item either $\theta_{\Sigma_h}^P(x,r)\sgeq \delta_{n,d}$ for all $r>0$, or
\item $\theta_{\Sigma_h}^P(x,r_0)<\delta_{n,d}$ for some $r_0>0$, in which case $\theta_\Sigma^P(x,r) < C_{n,d} r/r_0$ for all $0<r\sleq r_0$.
\end{enumerate}
\end{theorem}

We thus note that \cite{Bad} actually proved the following more general lemma.

\begin{lemma}[\cite{Bad}]\label{Badflat3}
Let $\cQ$ be a family of polynomials from $\RR^n$ to $\RR$ such that there exist constants $\delta_{\cQ}>0$ and $C_{\cQ}$ such that for all polynomials $q\in \cQ$ and any $x\in\Sigma_q = q\inv(0)$,
\begin{enumerate}[(1)]
\item either $\theta_{\Sigma_h}^P(x,r)\sgeq \delta_{\cQ}$ for all $r>0$, or
\item $\theta_{\Sigma_h}^P(x,r_x)<\delta_\cQ$ for some $r_x>0$, in which case $\theta_{\Sigma_h}^P(x,r) < C_{\cQ} r/{r_x}$ for all $0<r\sleq r_x$.
\end{enumerate}
Then $\cQ$ has the following property. For all $\delta>0$, there exist $\epsilon = \epsilon(\delta_\cQ,C_{\cQ}, \delta)>0$ and $\eta = \eta(\delta_\cQ,C_\cQ,\delta)>0$ with the following property. Let $\Sigma\subseteq \RR^n$, $x\in \Sigma$, $r>0$, and assume that
$$\sup_{0<r'\sleq r} \theta_{\Sigma}^{\cQ}(x,r') < \epsilon.$$
If $\theta_\Sigma^P(x,r)< \eta$, then $\sup_{0<r'\sleq r}\theta_\Sigma^P(x,r')<\delta$.
\end{lemma}

We now obtain the following corollary.

\begin{corollary}\label{KPflat2} For all $\delta>0$ there exist $\epsilon>0$ and $\eta >0$ with the following property. Let $\Sigma\subseteq\RR^4$, $x\in \Sigma$, and $r>0$ satisfy
\begin{equation}\label{5.6}
\sup_{0<r'\sleq r} \vartheta_\Sigma^C(x,r')<\epsilon.
\end{equation}
If further $\theta_\Sigma^P(x,r)<\eta$, then $\sup_{0<r'\sleq r}\theta_\Sigma^P(x,r')<\delta$.
\end{corollary}
\begin{proof}
Take $\cQ$ to be the family
$$\cQ = \{q(x) = p_{KP}(x-y): y\in\RR^n\}.$$
Next, we note that
\begin{equation}\label{KPflat2.1}
\theta_\Sigma^{\cQ}(x,r) = \vartheta_\Sigma^\cC(x,r)
\end{equation}
for all $x$ and $r$. By Lemma \ref{KPflat1}, conditions (1) and (2) of Lemma \ref{Badflat3} are satisfied for $\delta_\cQ = 1/\sqrt2$, $C_\cQ = C_0$ (the constant from Lemma \ref{KPflat1}), and $r_0 =|x|/2$ for any $x\neq 0$. Thus, the hypotheses of Lemma \ref{Badflat3} are satisfied, and we are done.
\end{proof}

\begin{lemma}\label{KPflat3}
For all $\sigma>0$ and $0<s\sleq 1/2$, there exists $\eta = \eta(\sigma,s)$ with the following property. Let $\Sigma\subseteq \RR^4$ be a closed set, $x\in\Sigma$, and $\cC$ be a KP cone based at $x$. If $y\in \Sigma$ satisfies
\begin{equation}\label{KPflat3.1}
\D{x,3|x-y|/2}\Sigma\cC<\eta,
\end{equation}
then
\begin{equation}\label{KPflat3.2}
\D{y,s|x-y|}\Sigma\cC<\sigma.
\end{equation}
\end{lemma}

\begin{proof}
Suppose the contrary. Then there exist $\sigma>0$, $0<s\sleq 1/2$, sequences $x_i, y_i\in \RR^4$, KP cones $\cC_i$ based at $x_i$, and closed sets $\Sigma_i$ with $x_i, y_i\in\Sigma_i$ satisfying
\begin{equation}\label{KPflat3.3}
\D{x_i, 3|x_i-y_i|/2}{\Sigma_i}{\cC_i} < \frac 1i\quad\mbox{ but }\quad
\D{y_i, s|x_i-y_i|}{\Sigma_i}{\cC_i} \sgeq \sigma.
\end{equation}
Let
\begin{equation}\label{KPflat3.5}
\twid \Sigma_i = \frac{\Sigma_i - x_i}{|y_i-x_i|}, \twid \cC_i = \frac{\cC_i - x_i}{|y_i-x_i|}, \twid y_i = \frac{y_i - x_i}{|y_i-x_i|}.
\end{equation}
Because $|\twid y_i|=1$ we may assume (by applying a rotation) that there exists $y\in \RR^4$ with $\twid y_i = y$ for all $i$. From   (\ref{KPflat3.3}) and (\ref{KPflat3.5}), we have that
\begin{equation}\label{KPflat3.6}
\D{0,3/2}{\twid \Sigma_i}{\twid \cC_i} < \frac 1i
\quad\mbox{ but }\quad
\D{y,s}{\twid\Sigma_i}{\twid \cC_i}\sgeq \sigma.
\end{equation}
By passing to a subsequence, we may assume that $\twid\Sigma_i\to \Sigma$ and $\twid \cC_i\to \cC$ in the topology of Hausdorff distance on compact balls, where $\Sigma$ is some closed set and $\cC$ is a KP cone based at 0. Then (\ref{KPflat3.6}) implies that
\begin{equation}\label{KPflat3.8}
\D{0,3/2}{\Sigma}{\cC} = 0
\quad\mbox{ but }\quad
\D{y,s}{\Sigma}{\cC} \sgeq \sigma > 0.
\end{equation}
Because $s\sleq 1/2$ and $|y| = 1$, $B(y,s)\subseteq B(0,3/2)$. Hence, (\ref{KPflat3.8}) implies that $\Sigma \cap B(0,3/2)\neq \cC\cap B(0,3/2)$ but $\D{0,3/2}{\Sigma}{\cC} = 0,$ yielding a contradiction.

\end{proof}

From Corollary \ref{KPflat2} and Lemma \ref{KPflat3}, we obtain a statement about how flat the ``sides'' of a set approximated by KP cones are.

\begin{corollary}\label{FlatSidesSet} For all $\delta>0$, there exists $A = A(\delta)$, $\epsilon = \epsilon (\delta)>0$ and $\eta = \eta(\delta)>0$ with the following property. Let $x_0\in\Sigma\subseteq\RR^4$ and $r_0>0$. Suppose that
\begin{equation}\label{FlatSidesSet1}
\sup_{0<r\sleq r_0, x\in \Sigma\cap B(x_0,r_0)} \vartheta_\Sigma^C(x,r')<\epsilon
\end{equation}
and $\cC$ is a KP cone based at $x_0$ such that
\begin{equation}\label{FlatSidesSet2}
\D{x_0,r_0}\Sigma \cC < \eta.
\end{equation}
If $x\in \Sigma$ satisfies $r_0/3\sleq |x-x_0| \sleq 2 r_0/ 3$, then for $0<r\sleq {|x-x_0|}/{A}$,
\begin{equation}\label{FlatSidesSet3}
\theta^P_\Sigma(x,r)<\delta.
\end{equation}
\end{corollary}

\begin{proof}
Let $\delta>0$ be given. Fix parameters $A$, $\epsilon$, $\eta >0$ to be specified later. Let $x\in\Sigma$ satisfy ${r_0}/3\sleq |x-x_0| \sleq {2r_0}/3$. We apply Lemma \ref{ConeG2} to (\ref{FlatSidesSet2}) and get
\begin{equation}\label{FlatSidesSet3.5}
\D{x_0, 2|x-x_0|}{\Sigma}{\cC} \sleq  2\frac {r_0}{2|x-x_0|} \eta \sleq \frac{r_0}{r_0/3}\eta = 3 \eta.
\end{equation}
 Let $\sigma>0$ and $s = 1/A$.
Then by (\ref{FlatSidesSet3.5}), we apply Lemma \ref{KPflat3} to get that there exists $\eta$ small enough so that 
\begin{equation}
\D{x, |x-x_0|/A}\Sigma \cC < \sigma/2.
\end{equation}
 By Lemma \ref{KPflat1}, we have that for $A$ large enough, 
\begin{equation}\label{FlatSidesSet4}
\D{x,|x-x_0|/A}\cC {T_x\cC}<\sigma/2.
\end{equation}
 Hence, 
\begin{equation}\label{FlatSidesSet5}
\D{x,|x-x_0|/A}\Sigma {T_x\cC}\sleq \sigma.
\end{equation}
 By Corollary \ref{KPflat2}, (\ref{FlatSidesSet1}) and (\ref{FlatSidesSet5}) imply that for $\sigma$ and $\epsilon$ small enough, $\theta_\Sigma^P(x,r)<\delta$ for all $0<r\sleq |x-x_0|/A$.

\end{proof}

We now quote the Reifenberg Topological Disk Theorem (see for example \cite{DKT}).

\begin{theorem}[Reifenberg Topological Disk Theorem]\label{DKTdisk} There exists $\xi_0 >0$ and $C_0$ with the following property. Let $\Sigma\subseteq\RR^n$ be a closed set and $y_0\in \Sigma$. Assume that $r_0>0$ and $0<\xi\sleq \xi_0$ satisfy
$$\theta_\Sigma^P(y,r)\sleq \xi\quad\mbox{ for all }y\in\Sigma\cap B(y_0,4r_0), 0<r\sleq 10r_0.$$
Let $P_0$ be a plane through $y_0$ such that
$$\D{y_0,10r_0} {P_0}\Sigma\sleq \xi.$$
Then there exists a continuous injective map
$$\tau:B(y_0,3r_0)\cap P_0\to \Sigma\cap B(0,4r_0)$$
which satisfies
\begin{equation}\label{DKTdisk1}
|\tau(y)-y|\sleq C_0 \,\xi\, r_0 \quad\mbox{ for all }y\in P_0\cap B(y_0, 3r_0).
\end{equation}
\end{theorem}

\begin{theorem}\label{Top} For all $\delta>0$, there exist $\epsilon = \epsilon(\delta)>0$ and $\eta = \eta(\delta)>0$ with the following property. Suppose that $\Sigma\subseteq \RR^4$ is a closed set,  and
\begin{equation}\label{Top1} 
\vartheta_\Sigma^C(z,r)<\epsilon\quad\mbox{for all }z\in\Sigma\cap  B(0,1), 0<r\sleq 1.
\end{equation}
Suppose also that $C$ is a KP cone based at the origin such that
\begin{equation}\label{Top3} \D{0,1} C\Sigma < \eta.\end{equation}
Let $x\in \cC$, $|x| = 1/2$, and $v$ be a unit vector such that
\begin{equation} \meang(v,T_xC)\sgeq \frac\pi4.\end{equation}
Then there exists $t\in \RR$ with $|t|\sleq \delta$ and $x+tv\in \Sigma$.
\end{theorem}

\begin{proof}
Let $\delta>0$ be given. Fix parameters $\epsilon>0, \eta>0$ to be specified later.  By (\ref{Top1}), (\ref{Top3}), and Corollary \ref{FlatSidesSet}, for any $\xi>0$ there exist $\epsilon>0$, $\eta>0$ small enough, and $A>0$ such that
\begin{equation}\label{Top6}
\theta_\Sigma^P(z,r)<\xi\quad\mbox{ for all }z\in \Sigma, \frac 13\sleq |z| \sleq \frac 23, 0<r\sleq \frac{|z|}A.
\end{equation}
Fix $\xi>0$ to be determined and stipulate that $\epsilon>0$ is small enough that (\ref{Top6}) holds. By (\ref{Top3}), there exists $y\in\Sigma$ with $|x-y|<\eta$. Stipulate that $\eta\sleq 1/12$. Then for all $y'\in B(y,1/12)\cap \Sigma$, we have that
$|y'|\sgeq |x|-|x-y|-|y'-y|\sgeq 1/2-1/12-1/12 = 1/3$. Thus by (\ref{Top6}),
\begin{equation}\label{Top7}
\theta_\Sigma^P(y',r)<\xi\quad\mbox{ for all }y'\in \Sigma\cap B\left(y,\frac1{12}\right), 0<r\sleq \frac1{3A}.
\end{equation}

We now require that $\xi\sleq \xi_0$ from Theorem \ref{DKTdisk}, $10r_0\sleq 1/(3A)$, and $4 r_0\sleq 1/{12}$. Then (\ref{Top7}) implies
\begin{equation}\label{Top8}
\theta_\Sigma^P(y',r)<\xi\sleq \xi_0\quad\mbox{ for all }y'\in\Sigma\cap B(y,4r_0), 0<r\sleq 10r_0.
\end{equation}
Statement (\ref{Top8}) tells us that the hypotheses of Theorem \ref{DKTdisk} are satisfied. Let $P_0$ be a plane such that
\begin{equation}\label{Top9.1} \D{y,10r_0} {P_0}\Sigma<\xi\quad\mbox{ and }\quad y\in P_0.\end{equation}
Then by Theorem \ref{DKTdisk}, there exists a continuous injective map $\tau:P_0\cap B(y,3r_0)\to\Sigma\cap B(y,4r_0)$ such that
\begin{equation}\label{Top9.2} |\tau(y')-y'|< C_0\,\xi \,r_0\quad\mbox{ for all } y'\in P_0\cap B(y,3r_0).\end{equation}

We would now like to know that the angle between $v$ and $P_0$ is not too small. We seek to establish that
\begin{equation}\label{Top10}
\meang(v,P_0)\sgeq  \frac\pi 6.\end{equation}
To show this, we will first establish that for small enough $r_0$, $\eta$, and $\xi$,
\begin{equation}\label{Top10.1}
\meang({P_0},{ T_x\cC})\sleq \frac\pi{12}.
\end{equation}
Recall that $y\in P_0$ by (\ref{Top9.1}). So by Lemma \ref{planes}, it is sufficient to show that
\begin{equation}\label{Top10.2}
\hd{y,3r_0}{P_0}{ T_x\cC}\sleq \frac 1{18}.
\end{equation}
Let $p\in P_0\cap B(y, 3r_0)$. Then by (\ref{Top9.1}), there exists $z\in \Sigma$ such that
\begin{equation}\label{Top10.3}
|p-z|\sleq 10 r_0 \xi.
\end{equation}
We require that $\xi\sleq 1/10$, so that $z\in B(y, 4r_0)$. Recall that $|x-y|\sleq \eta$. We require that $\eta\sleq r_0$ so that $z\in B(x, 5r_0)$ (actually, we will require later that $\eta$ be significantly smaller than $r_0$). Fix $\sigma>0$ to be specified later.
By (\ref{Top3}) and Lemma \ref{KPflat3}, we have that
\begin{equation}\label{Top10.4}
\D{x,5r_0}\Sigma\cC \sleq \sigma\end{equation}
For as $\eta$ small enough (depending on $r_0$). Thus, there exists $c\in \cC\cap B(x,5r_0)$ such that
\begin{equation}\label{Top10.4.1}
|z-c|\sleq 5 r_0 \sigma.
\end{equation}
By Lemma \ref{KPflat1}, we have that
\begin{equation}\label{Top10.5}
\D{x,5r_0}\cC{T_x\cC}\sleq Cr_0,\end{equation}
where we use $C$ in this proof to denote a constant which may depend on $\delta$.
Thus, there exists a $q\in T_x\cC$ such that
\begin{equation}\label{Top10.5.1}
|c-q|\sleq Cr_0^2.
\end{equation}
Combining (\ref{Top10.3}), (\ref{Top10.4.1}), and (\ref{Top10.5.1}), we get
\begin{equation}\label{Top10.6}
|p-q|\sleq 10 r_0 \xi + 5 r_0 \sigma + C r_0^2.
\end{equation}
Because for every $p\in P_0\cap B(y, 3r_0)$, there exists a $q\in T_x\cC$ satisfying (\ref{Top10.6}), we get that
\begin{equation}\label{Top10.7}
\hda{y,3r_0}{P_0}{T_x\cC}\sleq \frac{10}3  \xi + \frac53 \sigma + C r_0.
\end{equation}
We now require that
\begin{equation}\label{Top10.8}
 \frac{10}3  \xi,\frac53 \sigma,C r_0\sleq \frac1{108}.
\end{equation}
(That is, we first choose $r_0$ small enough so that $Cr_0\sleq 1/{108}$, then by Lemma \ref{KPflat3}, we choose $\eta$ small enough so that $5/3 \sigma \sleq 1/{108}$. We also require that $\xi\sleq {3}/{1080}$.)
From (\ref{Top10.7}) and (\ref{Top10.8}), we get
\begin{equation}\label{Top10.9}
\hda{y,3r_0}{P_0}{T_x\cC} \sleq \frac 3{108} = \frac 1{36}.
\end{equation}
Because $P_0$ is a cone centered at $y$, Lemma \ref{ConeG1} tells us that
\begin{equation}
\hd{y,3r_0}{P_0}{T_x\cC}\sleq \frac1{18}.
\end{equation}
Thus, by previous remarks we have established (\ref{Top10.1}) and proven that
\begin{equation}\label{Top10.10}
\meang(P_0, T_x\cC)\sleq \frac \pi{12}.
\end{equation}
We conclude that
\begin{equation}\label{Top10.12}
\meang(v, P_0)\sgeq \meang(v, T_x\cC) - \meang(T_x\cC, P_0)\sgeq \frac\pi 4 - \frac \pi{12} = \frac \pi 6.
\end{equation}

For $C_0$ (still the same constant from (\ref{Top9.2})), we now define the cylinders
\begin{equation}\label{T}
\begin{split}
T &= \{x+su+tv:u\in \langle v\rangle^\perp, |u|=1, |s|\sleq 3\sin(\pi/6) r_0+C_0\xi r_0, |t|\sleq \delta\}\\
T' &= \{x+su+tv:u\in\langle v\rangle^\perp, |u|=1, |s|\sleq 3\sin(\pi/6) r_0, |t|\sleq \delta -C_0\xi r_0 \}.
\end{split}
\end{equation}
We recall that $r_0\sleq 1/{12}$, and we require that $\xi$ be small enough that
\begin{equation}
C_0\xi r_0\sleq C_0 \xi \frac 1{12}\sleq \frac \delta 2.
\end{equation}
Recall that $\eta < r_0$ so that in particular, $y$ is in the interior of $T'$. We now require additionally that
\begin{equation}\label{Top11}
\eta, 3\cos(\pi/6)r_0 < \frac\delta 4
\end{equation}
 We observe three key facts about the geometry of these cylinders. First,
\begin{equation}\label{Top12}
\tau(y')\in \Sigma\cap T \quad\mbox{ for all }y'\in P_0\cap T'.
\end{equation}
Second, we observe that by (\ref{Top10.12}), (\ref{Top11}) and the definition of $T'$ (\ref{T}), that
$\partial T'\cap P_0$ is an ellipse with minimal axis length at least $3 \sin(\pi/6) r_0$. Third, $T'\cap P_0\subseteq B(y,3r_0)$. In particular, the map $\tau$ is defined on $T'\cap P_0$.

Define $\pi:\RR^4\to P_0$ to be the projection in the $v$ direction onto $P_0$. Note that to prove the existence of $|t|\sleq \delta$ such that $x+tv\in \Sigma$, it suffices to show that $\pi(x)$ has a $\pi$-preimage in $T\cap \Sigma$. Suppose the contrary; that is, that for all $y'\in\Sigma\cap T$, $\pi(y')\neq \pi(x)$. Then in particular, consider the continuous map $\pi\circ\tau:T'\cap P_0\to T\cap P_0$. Because $\pi$ is projection onto $P_0$ in the direction of $v$, we have that $|\pi(z) - z| = \sec(\meang(v, P_0^\perp)) d(z, P_0)$ for $z\in \RR^4$. Hence, applying and (\ref{Top10.3}) and (\ref{Top10.12}), for all $y'\in \Sigma\cap B(y,3r_0)\supseteq \Sigma\cap T$, we have that 
\begin{equation}
|\pi(y')-y'| = \sec(\meang(v, P_0^\perp)) d(y', P_0) \sleq 3r_0\xi \,\sec\left(\frac\pi2-\frac\pi6\right)\sleq C r_0 \xi.
\end{equation}
 Coupling this with (\ref{Top9.2}), we get that
\begin{equation}\label{Top14}
|\pi\circ\tau(y')-y'|\sleq  C\xi r_0. \end{equation}
By assumption, $\pi\circ\tau$ misses $\pi(x)$, so we may define a continuous retract $h: P_0\setminus\{\pi(x)\}\cap T'\to P_0\cap \partial T'$ which fixes $P_0\cap \partial T'$ (e.g. radial projection). Thus, we create a continuous map $\vphi = h\circ\pi\circ\tau:P_0\cap T'\to P_0\cap \partial T'$ such that
\begin{equation}\label{Top15}
|\vphi(y')-y'|\sleq C\xi r_0 \quad\mbox{ for all }y'\in P_0\cap \partial T'.\end{equation}
But because $P_0\cap T'$ is an ellipse with minimal axis of length at least $3 \sin(\pi/6) r_0$, $\vphi$ restricted to $P_0\cap \partial T'$ has degree 1 for small enough $\xi$. However, by degree theory for the sphere, a continuous map on the sphere extends continuously over the ball if and only if the degree of the map is $0$. So for $\xi$ small enough, we get a contradiction. Hence for small enough $\xi$, $\pi(x)$ has a preimage in $T$ and the lemma is proven.
\end{proof}

\section{The Local Structure of $(n-1)$-Asymptotically Optimally Doubling Measures}

Define the map $T_{x,r} (y) = ry + x$. For a measure $\mu$ on $\RR^{n}$, define $\mu_{x,r} = \frac 1{\mu(B(x,r))}T_{x,r\#}\mu$ to be the (rescaled) push forward measure under the map $T_{x,r}$. That is, 
$$\mu_{x,r}(A) = \frac{\mu\left(rA + x\right)}{\mu(B(x,r))} .$$
We write $\mu_i\rightharpoonup \mu$ for a sequence of measures $\mu_i$ converging weakly to $\mu$ in the sense of Radon measures.

\begin{definition}
Let $\mu$ and $\nu$ be nonzero Radon measures on $\RR^n$. We say that $\nu$ is a \emph{pseudo-tangent measure} of $\mu$ at $x$ if $x\in \supp\mu$ and there exist a sequence $x_i\in\supp\mu$ such that $x_i\to x$, a sequence of positive numbers $r_i\to 0$, and a sequence of positive numbers $c_i$ such that $c_iT_{x_i, r_i\#}\mu\rightharpoonup \nu$. We say that $\nu$ is a \emph{tangent measure} if it is a pseudo-tangent measure with $x_i = x$ for all $i$, and we denote the set of tangent measures to $\mu$ at $x$ by $\Tan(\mu,x)$.
\end{definition}

As is suggested by the names, the idea of a tangent measure came first and is due to Preiss, with the idea of a pseudo-tangent measure appearing later as a generalization. The following theorem says roughly that the pseudo-tangents of a locally doubling measure behave as we would expect. The first part gives a normalization on $c_i$ and the second part says that blow ups of the support converge to the support of the pseudo-tangent measure.

\begin{theorem}\label{localdoubling} Let $\mu$ be a locally doubling measure on $\RR^n$, $\Sigma = \supp\mu$,and $\nu$ be a pseudo-tangent measure of $\mu$ with $c_iT_{x_i, r_i\#}\mu\rightharpoonup \nu$. Then the following hold.
\begin{itemize}
\item[\emph{\cite{Mattila}}] There exists a constant $c>0$ such that
\begin{equation}
c\,\mu_{x_i,r_i} = \frac c{\mu(B(x_i,r_i))} T_{x_i, r_i\#} \mu\rightharpoonup\nu.
\end{equation}
\item [\emph{\cite{KT}}] Let $\Sigma_i = (\Sigma-x_i)/r_i = \supp \mu_{x_i,r_i}$. Then $\Sigma_i\to \Sigma_\infty = \supp \nu$ as $i\to\infty$, where the convergence is in the topology of Hausdorff distance restricted to compact balls. In particular, $0\in \supp\nu$.
\end{itemize}
\end{theorem}

When $\mu$ is an $m$-asymptotically optimally doubling measure, the pseudo-tangent measures are (up to multiplication by a constant) $m$-uniform measures.

\begin{theorem}[{\cite{KT}}]\label{KTaod}
Suppose that $\mu$ is an $m$-asymptotically optimally doubling measure on $\RR^n$, and $\nu$ is a pseudo-tangent measure of $\mu$. Then up to multiplying $\nu$ by a constant, Theorem \ref{localdoubling} says that $\omega_m \mu_{x_i, r_i}\rightharpoonup \nu$. In this case, $\nu$ is an $m$-uniform measure on $\RR^n$. If $m = n-1$, the classification of \cite{KP} says that $\nu$ is $(n-1)$-dimensional Hausdorff measure restricted to either an $(n-1)$-plane containing 0 or a KP cone containing 0.
\end{theorem}

\begin{remark}
In light of Theorem \ref{KTaod}, it is helpful to recall that KP cones are defined in $\RR^n$ only for $n\sgeq 4$.
\end{remark}

Preiss \cite{Pre} showed that the cone of tangent measures satisfies a strong form of connectedness in the topology of weak convergence of Radon measures. This general feature of tangent measures, together with deep computations on the geometry of uniform measures by Preiss, establish the following result (which is a particular case of Preiss' Theorem). We follow the language of Preiss, saying that a measure $\mu$ is $m$-\emph{flat} if $\mu$ is $m$-dimensional Hausdorff measure restricted to some $m$-plane. Note that a flat measure is $m$-uniform but not every $m$-uniform measure is flat.

\begin{corollary}[\cite{Pre}]\label{KPTstrongconnectedness}
Let $\mu$ be an asymptotically optimally doubling measure, and $x\in \supp\mu$. Then if one tangent measure to $\mu$ at $x$ is flat, all tangent measures to $\mu$ at $x$ are flat.
\end{corollary}

\begin{corollary}\label{strongconnectedness}
Let $\mu$ be an $m$-asymptotically optimally doubling measure, and $x\in \supp \mu$. If one tangent measure to $\mu$ at $x$ is flat, then $x$ is a flat point of $\supp\mu$.\end{corollary}

\begin{proof} Let all notation and suppositions hold. Define
\begin{equation}\label{sc1}
\ell = \limsup_{r\downarrow 0} \theta_\Sigma^P(x,r).
\end{equation}
Let $r_i\to 0$ be a sequence such that $\theta_\Sigma^P(x,r_i)\to \ell$. Then by weak compactness of Radon measures, there exists a subsequence such that $\mu_{x,r_i}\rightharpoonup \mu_\infty$. By Corollary \ref{KPTstrongconnectedness} and the assumption that one tangent measure is flat, $\mu_\infty$ is flat. Hence its support is some $m$-plane $P$. By Theorem \ref{localdoubling}, $\D{x,r}\Sigma{P+x}\to 0$, and hence $\ell = 0$.
\end{proof}

We begin by showing that the only tangent measures to $(n-1)$-asymptotically optimally doubling measures are Hausdorff measure on planes or KP cones based at the origin. To do so, we first quote a Lemma about tangent measures to tangent measures.

\begin{lemma}[\cite{Bad2}]\label{tangenttotangent}
If $\mu$ is a measure on $\RR^n$, $x\in \supp(\mu)$, and $\nu\in \Tan(\mu,x)$ such that $0\in \supp(\nu)$, then $\Tan(\nu,0)\subseteq \Tan(\mu,x)$.
\end{lemma}

\begin{corollary}\label{tanmeasures}
Let $\mu$ be an $(n-1)$-asymptotically optimally doubling measure on $\RR^{n}$ and $\nu$ a tangent measure to $\mu$ at $x\in\supp \mu$. Then up to rescaling by a constant, $\nu$ is either $\cH^{n-1}|_\cC$ for a KP cone $\cC$ based at 0 or $\cH^{n-1}|_P$ for an $(n-1)$-plane $P$ containing 0.
\end{corollary}

\begin{remark}
Before giving its proof, we stop to note that Corollary \ref{tanmeasures} differs from Theorem \ref{KTaod} by telling us that if $\nu$ is a tangent measure of $\mu$ (and not just a pseudo-tangent measure), then it is either flat or Hausdorff measure on a KP-cone based at the origin (not just containing the origin).
\end{remark}

\begin{proof}
Let all notation and suppositions hold. Suppose that $\supp\nu$ is neither a KP cone centered at the origin nor a plane containing the origin. By Theorem \ref{KTaod}, the only other option is that $\supp\nu$ is a KP cone centered somewhere besides the origin. However, by Lemma \ref{tangenttotangent}, a tangent measure to $\nu$ at the origin is a tangent measure to $\mu$ at $x$. However, if $\nu = \cH^{n-1}|_\cC$ for some KP cone $\cC$ not centered at the origin, the (unique) tangent measure to $\nu$ at the origin is $\cH^{n-1}|_P$ for some plane $P$. However, this violates Corollary \ref{KPTstrongconnectedness}, yielding a contradiction and finishing the proof.
\end{proof}

\begin{theorem}\label{AODstructure}
Suppose that $\mu$ is an $(n-1)-$asymptotically optimally doubling measure on $\RR^n$ with $\Sigma = \supp\mu$. Then
$\vartheta_\Sigma^C(x,r)\to 0$ uniformly on compact sets. Further, 
suppose that $x\in\Sigma$ a nonflat point. Then $\theta_\Sigma^C(x,r) \to 0$ as $r\downarrow 0$.
\end{theorem}

\begin{proof}
Let $\mu$ be an $(n-1)-$asymptotically optimally doubling measure on $\RR^n$ and $\Sigma = \supp\mu$. Fix a compact set $K\subseteq \Sigma$, and define
\begin{equation}\label{AODstructure.1}
\ell = \lim_{r\to 0}\sup_{x\in K} \vartheta_\Sigma^C(x,r).
\end{equation}
Let $x_i$ and $r_i$ be sequences such that $r_i\to 0$, $x_i\in K$, and 
\begin{equation}\label{AODstructure.2}
\ell = \lim_{i\to\infty} \vartheta_\Sigma^C(x_i,r_i).
\end{equation}
Because $K$ is compact, we have that there is an $x\in \Sigma$ and a subsequence (which we relabel) such that $x_i\to x$. By weak compactness of Radon measures, we may extract a subsequence (which we also relabel) such that
\begin{equation}\label{AODstructure.3}
\mu_{x_i,r_i}\rightharpoonup\mu_{\infty}.
\end{equation}
By Theorem \ref{KTaod}, we have that $\Sigma_\infty:=\supp \mu_\infty$ is either a plane or a KP cone containing 0. By Theorem \ref{localdoubling}, we have that
\begin{equation}\label{AODstructure.4}
\D{0,1}{\frac{\Sigma-x_i}{r_i}}{\Sigma_\infty}\to 0.
\end{equation}
Applying scale invariance, we get that
\begin{equation}\label{AODstructure.5}
\D{x_i,r_i}{\Sigma}{r_i \Sigma_\infty + x_i}\to 0.
\end{equation}
If $\Sigma_\infty$ is a plane or a KP cone including $0$, then $r_i \Sigma_\infty + x_i$ is a plane or KP cone including $x_i$ respectively. If $\Sigma_\infty$ is a KP cone, then (\ref{AODstructure.5}) shows that $\vartheta_\Sigma^C(x_i,r_i)\to 0$. Suppose that $\Sigma_\infty$ is a plane. Then (\ref{AODstructure.5}) shows that $\theta_\Sigma^P(x_i,r_i)\to 0$. We now claim that for any set $\Sigma$, any $x\in \Sigma$, and $r>0$, we have
\begin{equation}\label{AODstructure.6}
\vartheta_\Sigma^C(x,r)\sleq \theta_\Sigma^P(x,r).
\end{equation}
In light of (\ref{AODstructure.6}), we recall that while the infimum in the definition $\theta_\Sigma^P$ is always obtained (and hence could have been called a minimum), that the infimum of $\vartheta_\Sigma^C$ may not be obtained. However, if $P$ is any plane, $x\in P$, and $r>0$, then by choosing a KP cone $\cC$ whose nonflat points are very far away from $x$ and whose tangent plane at $x$ is $P$, we may make $\D{x,r}{\cC}P$ as small as we like.
 From (\ref{AODstructure.6}) and $\theta_\Sigma^P(x_i,r_i)\to 0$, we have that $\vartheta_\Sigma^C(x_i,r_i)\to 0$. Thus, $\ell = 0$ (see (\ref{AODstructure.1})), and we have that $\vartheta_\Sigma^C(x,r)\to 0$ uniformly on $K$ as $r\downarrow 0$. Thus, $\vartheta_\Sigma^C(x,r)\to 0$ uniformly on compact subsets as $r\downarrow 0$.

Suppose now that $x\in \Sigma$ is a nonflat point. Define
\begin{equation}\label{AODstructure.7}
\ell = \limsup_{r\downarrow 0} \theta_\Sigma^C(x,r).
\end{equation}
Let $r_i$ be a sequence such that $r_i\downarrow 0$ and $\theta_\Sigma^C(x,r_i)\to \ell.$ Identically to before, we extract a subsequence $r_i$ such that $\mu_{x,r_i}\to \mu_\infty$, and $\Sigma_\infty = \supp\mu_\infty$ is either a plane or a KP cone. However, by the nonflatness assumption and Corollary \ref{strongconnectedness}, we have that $\Sigma_\infty$ is not a plane. By Corollary \ref{tanmeasures}, we thus have that $\Sigma_\infty$ is a KP cone centered at $0$. Hence, we get that 
\begin{equation}\label{AODstructure.8}
\D{0,1}{\frac {\Sigma-x}{r_i}}{\Sigma_\infty}\to 0.
\end{equation}
Hence,
\begin{equation}\label{AODstructure.9}
\D{x,r_i}{\Sigma}{r_i\Sigma_\infty+x}\to 0.
\end{equation}
Because  $\Sigma_\infty$ is a KP cone centered at the origin (and hence $r_i\Sigma_\infty  + x= \Sigma_\infty+x$ is a KP cone centered at $x$), we have that $\ell = 0$.

\end{proof}

\begin{remark}
Let $\mathbf{KP}$ be the set of KP cones in $\RR^n$. We note that (\ref{AODstructure.6}) implies that the set of planes is contained in the closure of $\mathbf{ KP}$ (in the topology of Hausdorff distance restricted to compact balls). It is not hard to see that there is a $\delta_0$ such that if $\Sigma$ is any set with $\theta_\Sigma^P(x,r)\sleq \delta_0$, then we also have that $\theta_\Sigma^P(x,r)\sleq 2 \vartheta_\Sigma^C(x,r)$ (see Lemma \ref{KPflat1}). It follows that the closure of $\mathbf{KP}$ is $\mathbf{KP}\cup \{(n-1)-$planes$\}$.
\end{remark}

The following lemma gives quantitative information on the flatness of the support of $\mu$  at points near a nonflat point and scales which are sufficiently small.

\begin{corollary}\label{FlatSides} For any $\delta>0$, there exist $\epsilon = \epsilon(\delta)$, $\eta = \eta(\delta)$, and $A = A(\delta)$ with the following property. Let $\mu$ be a 3-asymptotically optimally doubling measure on $\RR^4$, and suppose that $0\in \Sigma = \supp \mu$ is a nonflat point. By Theorem \ref{AODstructure}, we have that there is an $r_0$ small enough such that
\leqn{Flatsides0}
\sup_{0<r\sleq r_0, x\in \Sigma\cap B(0,r_0)}\vartheta_\Sigma^C(x,r)< \epsilon\quad\mbox{ and }\quad\sup_{0<r\sleq r_0}\theta_\Sigma^C(0,r)<\eta.
\endeqn
For this $r_0$, it holds that
\begin{equation}\label{Flatsides1}
\theta_\Sigma^P(x,r) < \delta\quad\mbox{ for all } x\in B\left(0,\frac{2r_0}3\right)\cap\Sigma, r < \frac {|x|}A .
\end{equation}
\end{corollary}

\begin{proof}
Let all notation and supposition hold. Note that the hypotheses of Corollary \ref{FlatSidesSet} are satisfied by Theorem \ref{AODstructure}, and the conclusion of Corollary \ref{FlatSidesSet} gives us (\ref{Flatsides1}).
\end{proof}

\section{Nonflat Points of H\"older Asymptotically Optimally Doubling Measures}\label{NonflatPointsofHOAOD}

In this section, we begin our investigation of the nonflat points in the support of a H\"older asymptotically uniform measure (see Theorem \ref{HAL}). We find appropriate H\"older estimates on $\theta_\Sigma^P$, $\theta_\Sigma^C$, and $\vartheta_\Sigma^C$ in a neighborhood of a nonflat point in the support. In Section 5, we use these estimates to construct a parametrization by a KP cone. Stated precisely, in the next two sections we prove the following theorem.

\begin{theorem}\label{theorem} For any $\alpha>0$, there exists $\beta = \beta(\alpha)$ with the following property. If $\mu$ is an $(\alpha, 3)$-asymptotically uniform measure, then for any $x\in \supp\mu$ which is nonflat
there exist a KP cone centered at $0$, neighborhoods $U$ of $0$ and $U'$ of $x$ and a diffeomorphism $\vphi\in C^{1,\beta}(U\to U')$ such that $\vphi(\cC\cap U) = \supp(\mu)\cap U'$. Further, $\vphi$ has the property that $\vphi(0) = x$ and $D_0\vphi = \Id$.
\end{theorem}

To this end, we assume that $\alpha >0$ and
\leqn{4hypos}
\begin{array}{l}
\mbox{$\mu$ is a Radon measure on $\RR^{4}$ which is $(\alpha,3)-$asymptotically uniform,}\\
\Sigma = \supp \mu\mbox{ satisfies that $0\in \Sigma$ is a nonflat point,}\\
\left| \frac{\mu(B(x,r))}{\omega_{3} r^{3}}-1 \right|\sleq C_{0} r^\alpha \mbox{ for }x\in\Sigma\cap B(0,1), 0<r\sleq 1 \mbox{ (see Definition \ref{mureg}).}
\end{array}
\endeqn
The second and third conditions may be viewed simply as a translation and dilation to normalize the scales at which we work. In Section \ref{Control of the Moments}, we adopt methods from \cite{DKT} to gain control on polynomials which we will call the moments of $\mu$. In Section \ref{atsing}, we use the information about the moments to get quantitative bounds on $\theta_\Sigma^C(0,r)$. Finally in Section \ref{AwayFromOrigin}, we develop quantitative information at points near the origin and scales sufficiently small. We then construct a parametrization in Section \ref{Parametrization} for a set $\Sigma$ satisfying the estimates we demonstrate. In Section 4, the constant $C$ depends on $\alpha$ and $C_0$ and radius $r_0$ is chosen small enough depending on $\alpha$ and $C_0$, as well as 
$$\sup_{0<r\sleq r_0}\theta^C(0,r)\quad\mbox{and}\quad \sup_{0<r\sleq r_0, x\in \Sigma\cap B(0,r_0)} \vartheta^C(x,r)$$
being small enough (see Theorem \ref{AODstructure}).

\subsection{Control of the Moments}\label{Control of the Moments}

Let $\nu$ be a Radon measure on $\RR^n$. Define the \emph{first moment} of $\nu$ at a point $x\in \supp(\nu)$ and a scale $r$ to be the vector
\leqn{bdef}
b_{x,r}(\nu) = \frac{n+1}{2\omega_{n-1} r^{n+1}} \int_{B(x,r)} (r^2 - |y-x|^2)(y-x) \, d\mu(y).
\endeqn
Define the \emph{second moment} of $\nu$ at a point $x\in \supp(\nu)$ and a scale $r$ to be the quadratic form
\leqn{Qdef}
Q_{x,r}(\nu) (y) = \frac{n+1}{\omega_{n-1} r^{n+1}} \int_{B(x,r)} \langle y, z-x\rangle^2 \, d\mu(z).
\endeqn
Define also the \emph{trace of $Q_{x,r}(\nu)$} to be
\leqn{trdef}
\tr Q_{x,r}(\nu) = \frac{n+1}{\omega_{n-1} r^{n+1} }\int_{B(x,r)} |z-x|^2 \, d\mu(z).
\endeqn
In any orthonormal coordinates centered at the origin $(x_1, \ldots,x_n)$ (and $x$ and $r$ understood), we set
\leqn{qdef}
q_{ij}(\nu)=  \frac{n+1}{\omega_{n-1} r^{n+1}} \int_{B(x,r)} (z_i-x_i)(z_j-x_j) \, d\nu(z).
\endeqn
It follows that $Q_{x,r}(\nu) = \sum_{i,j=1}^n q_{ij}(\nu) x_ix_j$. Moreover, $\tr Q_{x,r}(\nu) = \sum_{i=1}^n q_{ii}(\nu)$, and so coincides with the usual notion of the trace of a quadratic polynomial. For the rest of Section 4, we set $Q_{x,r} = Q_{x,r}(\mu)$ and $b_{x,r} = b_{x,r}(\mu)$ (see (\ref{4hypos})).

First, we set some notation. Fix $0<\gamma<\theta<\alpha/2$ for the remainder of Section 4. We will sometimes work at the scale $\rho = r^{1+\gamma}$. We also denote $\twid Q_{x,r}(z) = Q_{x,r}(z)-|z|^2$. Further, we set the blow up of $\Sigma$ at scale $r$ to be $\Sigma_r = (1/r)\Sigma$ (and so $\Sigma_\rho = (1/\rho)\Sigma$). We now summarize some of the results of \cite{DKT} which  highlight the interactions between the moment at the scale $r$ and the geometry at a different scale.

\begin{theorem}\label{DKTbounds}
Recall hypotheses (\ref{4hypos}). 
\begin{enumerate}[(1)]
\item{\rm{\cite{DKT}}} For $0<r\sleq 1/2$, $|\tr Q_{0,r}-(n-1)|\sleq C r^\alpha.$
\item{\rm{\cite{DKT}}} For $0<r\sleq 1/2$. $|b_{0,r}|\sleq C r^{1+\theta}.$
\item{\rm{\cite{DKT}}} For $0<r\sleq 1/2$ and $z\in \Sigma\cap B(0,r/2)$, 
\begin{equation}
|2\langle b_{0,r}, z\rangle + \twid Q_{0,r}(z)|\sleq C \frac{|z|^3}{r} + C r^{2+\alpha}
\end{equation}
\item For $0<r\sleq 1/2$, $0<s\sleq r/4$, and $z\in \Sigma_s\cap B(0,2)$,
\begin{equation}
|\twid Q_{0,r}(z)|\sleq C \frac s r + C \frac {r^{2+\alpha}}{s^2} + C \frac {r^{1+\theta}}{s}.
\endeqn
\item { For $M\sgeq 4$, $0<r\sleq 1/2$, and $z\in \Sigma_{r/M}\cap B(0,2)$ }
\begin{equation}
|\twid Q_{0,r}(z)|\sleq CMr^\theta + \frac CM+ CM^2 r^\alpha.
\end{equation}
\item For $\tau\in[1 /{2}, 1]$, $r$ such that $0<r^\gamma\sleq 1/{8}$, and $z\in \Sigma_{\rho}\cap B(0,2)$, 
\begin{equation}\label{DKTineq}
\left|\twid Q_{0,\tau r}(z)\right| \sleq Cr^{\theta-\gamma} +C r^\gamma + Cr^{\alpha-2\gamma}.
\end{equation}
Letting $\beta_0 = \min(\theta-\gamma, \gamma, \alpha-2\gamma)$, (\ref{DKTineq}) gives that
\begin{equation}\label{DKTineq'}
\left|\twid Q_{0,\tau r}(z)\right| \sleq Cr^{\beta_0}.
\end{equation}
\end{enumerate}
\end{theorem}

\begin{remark}
 We note that in \cite{DKT}, two cases with respect to (2) were considered: the case where $b_{0,r}$ is small (satisfies Theorem \ref{DKTbounds}(2)), and the case where it is large (does not satisfy Theorem \ref{DKTbounds}(2)). However, contained in their analysis, they showed that the latter case automatically implies flatness. Thus, our nonflatness assumption implies Theorem \ref{DKTbounds}(2).
\end{remark}

\begin{proof}[Proof of (4), (5), and (6)]
We begin by proving (4). Set $x = sz$. Then $x\in \Sigma\cap B(0,2s)\subseteq \Sigma\cap B(0,r/2)$. We apply (3) to see that
\begin{equation}
\left|2\left\langle \frac{b_{0,r}}s, z\right\rangle + \twid Q_{0,r}(z)\right| = \frac1{s^2}\left|2\langle  {b_{0,r}}, x\rangle + \twid Q_{0,r}(x) \right|  \sleq \frac 1 {s^2}\left(C\frac{|x|^3}r + Cr^{2+\alpha}\right)\sleq C\frac s r + c \frac{r^{2+\alpha}}{s^2}.
\endeqn
By applying (2), we get
\begin{equation}
|\twid Q_{0,r}(z)| \sleq 2\frac {|b_{0,r}||z|}{s} + C \frac s r + C\frac {r^{2+\alpha}}{s^2}\sleq C \frac {r^{1+\theta}}{s} + C \frac s r + C\frac {r^{2+\alpha}}{s^2}
\endeqn

We now have that (5) follows from (4) by setting $s = r/M$ and checking that $0<s\sleq r/4$. Similarly, (6) follows from (4) by taking $s = \rho = r^{1+\gamma}$ and taking $\tau r$ as our radius, and checking that $0<s\sleq \tau r/4$.

\end{proof}

We now seek to understand the second moment $Q_{0,r}(x)$. Let $(x_1, x_2,x_3,x_4)$ be orthonormal coordinates centered at the origin. Since $Q_{0,r}$ is a quadratic polynomial, we can represent it as the matrix $Q_{0,r} = (q_{ij})$.  Note that if we compute the second moment $Q_{0,r}(\cH^3\res_{\cC}) = K$ of 3-dimensional Hausdorff  measure on the KP cone $\cC = \{x_4^2 = x_1^2 + x_2^2 + x_3^2\}$ (at any radius) we get
$$K = \left(\begin{array}{cccc}
\frac 32 & 0 & 0 & 0 \\
0 & \frac 12 & 0 & 0 \\
0 & 0 & \frac 12 & 0 \\
0 & 0 & 0 & \frac 12 \end{array}
\right).$$ Our first lemma proves that at small enough radii, $Q_{0,r}$ becomes close to $K$.

\begin{lemma}\label{QvsK} Let $\delta > 0$. There exists an $r_0>0$ such that for all $0<r\sleq r_0$, there exists an orthonormal basis $(x_1, x_2, x_3, x_4)$ for which $\max_{ij}|q_{ij}-K_{ij}| <\delta$ .
\end{lemma}

\begin{proof}
Let $\delta>0$ be given. Fix parameters $M\sgeq4$ and $\epsilon>0$ to be specified later. By Theorem \ref{AODstructure}, there exists $r_0>0$ such that for all $0<r\sleq r_0$,
\begin{equation}\label{QvsK.1}
\theta_\Sigma^C(0,r)< \frac\epsilon M.
\end{equation}
Fix $0<r\sleq r_0$. Let $\cC$ be a KP cone centered at the origin such that
\begin{equation}\label{QvsK.2}
\D{0,r}\Sigma\cC <\frac\epsilon M.
\end{equation}
Let $(x_1, x_2, x_3,x_4)$ be orthonormal coordinates such that
\begin{equation}\label{QvsK.2.1}
\cC = \{x_4^2 = x_1^2 + x_2^2 + x_3^2\}.
\end{equation}
By finding the vector from $x$ to $\cC$ normal to $\cC$, a quick computation shows that for any point $x\in\RR^4$,
\begin{equation}\label{QvsK.3}
d(x,\cC) = \frac{\big| |x_4|- |(x_1,x_2,x_3)|\big|}{\sqrt 2}.
\end{equation}
We manipulate (\ref{QvsK.3}) to get
\begin{equation}\label{QvsK.4}
2 d(x,\cC)^2 = x_1^2 + x_2^2 + x_3^2 + x_4^2 -2|x_4||(x_1,x_2,x_3)|.
\end{equation}
Set $C_1 =  5/{\omega_3}$, we compute
\begin{equation}\label{QvsK.5}
\begin{split}\frac{C_1}{r^{5}} \int_{B(0,r)} 2 d(x,\cC)^2 \,d\mu(x) & = \frac{C_1}{r^{5}} \int_{B(0,r)} x_1^2 + x_2^2+x_3^2+x_4^2 - 2|x_4||(x_1,x_2,x_3)| \, d\mu(x) \\
& =  \tr Q_{0,r} - 2 \frac {C_1}{r^{5}} \int_{B(0,r)}|x_4||(x_1,x_2,x_3)|\, d\mu(x).\end{split}
\end{equation}
We apply H\"older's inequality to find
\begin{equation}\label{QvsK.6}
\begin{split}
\frac {C_1}{r^5} \int_{B(0,r)}|x_4||(x_1,x_2,x_3)|\, d\mu & \sleq 
 \left(\frac {C_1}{r^5} \int_{B(0,r)}|x_4|^2\, d\mu\right)^{\frac12} \left(\frac {C_1}{r^5} \int_{B(0,r)}|(x_1,x_2,x_3)|^2\, d\mu\right)^{\frac12}\\
& = \sqrt{(q_{44})(q_{11} + q_{22} + q_{33})}
\end{split}
\end{equation}
Substituting (\ref{QvsK.6}) into (\ref{QvsK.5}) and manipulating, we get that
\begin{equation}\label{QvsK.7}
\frac{C_1}{r^5} \int_{B(0,r)} 2d(x,\cC)^2 \,d\mu \sgeq \tr Q_{0,r} - 2\sqrt{(q_{11} + q_{22}+q_{33})(q_{44})} =  \left(\sqrt q_{44} - \sqrt{q_{11}+q_{22}+q_{33}}\right)^2.\end{equation}
By applying the H\"older asymptotically uniform property and (\ref{QvsK.2}), we also have that
\begin{equation}\label{QvsK.8}
\begin{split}
\frac{C_1}{r^5}\int_{B(0,r)} 2 d(x,\cC)^2 \,d\mu& \sleq \frac{2C_1}{r^5} \int_{B(0,r)} \left(\frac {\epsilon r}M\right)^2 \,d\mu = 10\frac{\mu(B(0,r))}{\omega_3r^3} \frac{\epsilon^2}{M^2}\\
&\sleq 10 \left(1 + C_0r_0^\alpha\right)\frac{\epsilon^2}{M^2}\sleq {C}\frac{\epsilon^2}{M^2}.
\end{split}
\end{equation}
Combining (\ref{QvsK.7}) and (\ref{QvsK.8}), we have that
\begin{equation}\label{QvsK.9}
\left(\sqrt q_{44} - \sqrt{q_{11}+q_{22}+q_{33}}\right)^2 \sleq C\frac{\epsilon^2}{M^2}.
\end{equation}
Note that $0\sleq q_{ii}\sleq 3 + Cr^\alpha$ by Theorem \ref{DKTbounds}(1) and the definition of $q_{ij}$. We use this fact and (\ref{QvsK.9}) to get
\begin{equation}\label{QvsK.10}
\big| q_{44} - q_{11} - q_{22} - q_{33} \big| = \big| \sqrt{q_{44} }-\sqrt{ q_{11} + q_{22} + q_{33}} \big|\cdot\big| \sqrt{q_{44}}+\sqrt{ q_{11} + q_{22} + q_{33}} \big|   \sleq C \frac{\epsilon}{M} 
\end{equation}
By Theorem \ref{DKTbounds}(1), we also have that
\begin{equation}\label{QvsK.11}
|q_{11} + q_{22} + q_{33} + q_{44} -3| \sleq C r^{\alpha}.
\end{equation}
From (\ref{QvsK.10}) and (\ref{QvsK.11}), we get that
\begin{equation}\label{QvsK.12}
|q_{44} - \frac32| \sleq C\frac{\epsilon}{M}+ C r^\alpha \quad\mbox{and}\quad
|q_{11}+ q_{22}+q_{33} - \frac32| \sleq C\frac{\epsilon}{M}+ C r^\alpha.
\end{equation}
Set $\sigma = \sigma(r, M, \epsilon): = C \epsilon/M + Cr_0^\alpha$ for the constants above. Note $\sigma\to 0$ as $r, \epsilon \to 0$ (recall $M\sgeq 4$).

From Theorem \ref{DKTbounds}, we have that
\begin{equation}\label{QvsK.13}
|\twid Q_{0,r}(z)| \sleq C Mr^{\theta} + \frac CM + C M^2 r^{\alpha} \quad \mbox{ for }z\in \Sigma_{r/M} \cap B(0,2).
\end{equation}
We now extend this to information about the points in $\cC$. First, we note that because $\cC$ is a cone centered at the origin, by Lemma \ref{ConeG1} and (\ref{QvsK.2}), it follows that
\begin{equation}\label{QvsK.14}
\D{0,1}{\cC}{\Sigma_{r/M}}\sleq 2 \epsilon.
\end{equation}
Since $\twid Q_{0,r}$ is a quadratic form with $|\twid Q_{0,r}(x)| \sleq C |x|^2$, we get that for any $e\in \Ss^3$,  
\begin{equation}\label{QvsK.15}
|\partial_{e} \twid Q_{0,r}(x)| \sleq C |x| \sleq C\quad\mbox{ for }x\in B(0,2).
\end{equation}
Let $a\in \cC\cap B(0,1)$. By (\ref{QvsK.14}) there is a point $x\in \Sigma_{r/M}$ such that $|x-a|<2\epsilon $. Setting $e = \frac{x-a}{|x-a|}$, we integrate along the path from $x$ to $a$ along $e$, apply (\ref{QvsK.15}), and get that
\begin{equation}\label{QvsK.16}
|\twid Q_{0,r}(a)| \sleq |\twid Q_{0,r}(x)| + |\twid Q_{0,r}(a) - \twid Q_{0,r}(x)|\sleq |\twid Q_{0,r}(x)| + C\epsilon.
\end{equation}
Combining (\ref{QvsK.13}) and (\ref{QvsK.16}), we get that
\begin{equation}\label{QvsK.17}
|\twid Q_{0,r}(a)|\sleq  C Mr^{\theta} + \frac CM + C M^2 r^{\alpha}+ C\epsilon = : \eta(r, M, \epsilon)\quad\mbox{ for }a\in \cC\cap B(0,1).
\end{equation}
Note that by choosing $\epsilon$ small, $M$ large, and $r$ very small (depending on $M$), we can make $\eta = \eta(r, M, \epsilon)$ as small as we like.

We now plug in some special points of $\cC$ to extract information about $Q_{0,r}$. We continue to work in the same orthonormal coordinates $(x_1,x_2,x_3,x_4)$ (see (\ref{QvsK.2.1})). Let $z_i^\pm = (x_i\pm x_4)/{\sqrt{2}}$ for $i = 1,2,3$. Because $z^{\pm}_i\in \cC\cap B(0,1)$ , we may apply (\ref{QvsK.17}) to get
\begin{equation}\label{QvsK.18}
|\twid Q_{0,r}(z_i^\pm)| = \left|\frac12q_{ii} + \frac12q_{44} - 1 \pm q_{i4}\right|\sleq \eta
\end{equation}
(recall that by (\ref{qdef}), $q_{i4} = q_{4i}$). From (\ref{QvsK.18}) and (\ref{QvsK.12}), we get
\begin{equation}\label{QvsK.19}
\left|q_{ii}-\frac12\right|\sleq \left|\frac 12 q_{ii} + \frac12q_{44} -1 + q_{i4}\right| + 
\left|\frac 12 q_{ii} + \frac12q_{44} -1 - q_{i4}  \right | + \left | \frac 32 - q_{44}\right | \sleq 2\eta + \sigma.
\end{equation}
We also get from (\ref{QvsK.18}) that
\begin{equation}\label{QvsK.20}
|q_{i4}|\sleq \frac12 \left| q_{i4} + \frac 12 q_{ii} + \frac 12 q_{44} - 1\right| + \frac12 \left| q_{i4} - \frac 12 q_{ii} - \frac 12 q_{44} + 1\right|\sleq \eta.
\end{equation}
Let $y_{ij} = (x_i + x_j)/2 + x_4/\sqrt 2$ for $i,j = 1,2,3$. Because $y_{ij}\in \cC\cap B(0,1)$, we may apply (\ref{QvsK.17}) to get
\begin{equation}\label{QvsK.21}
|\twid Q_{0,r}(y_{ij})| = \left| \frac 14 q_{ii} + \frac 14 q_{jj} + \frac 12 q_{44} + \frac 12 q_{ij}+ \frac 1{\sqrt 2} q_{i4} + \frac 1{\sqrt 2} q_{j4} - 1\right|\sleq \eta
\end{equation}
(recall that by (\ref{qdef}), $q_{ij} = q_{ji}$). So from (\ref{QvsK.12}), (\ref{QvsK.19}), (\ref{QvsK.20}), and (\ref{QvsK.21}), we get
\begin{equation}\label{QvsK.22}
|q_{ij} | \sleq 2 \eta + \left|q_{ii}-\frac12\right| + \left|q_{jj} - \frac 12 \right| + 2\left|q_{44} - \frac 32\right|  + 2\sqrt 2|q_{i4}| + 2\sqrt 2|q_{j4}|\sleq C \eta + C \sigma.
\end{equation}
Hence, by  (\ref{QvsK.12}), (\ref{QvsK.19}), (\ref{QvsK.20}), and (\ref{QvsK.22}), we can make $\eta$ and $\sigma$ small enough such that $\max_{ij}|q_{ij} - K_{ij}|\sleq \delta$  in the orthonormal basis $(x_1, x_2, x_3, x_4)$.

\end{proof}

\begin{lemma}\label{diagonalbasis}
For any $\delta > 0$, there is an $r_0$ small enough such that the following hold. For $0<r\sleq r_0$, there is an orthonormal basis $(x_1, x_2, x_3, x_4)$ diagonalizing $Q_{0,r}$ as
\begin{equation}\label{diagonalbasis.1}
Q_{0,r} = \sum_{i=1}^4 \lambda_i x_i^2
\end{equation}
and
\begin{equation}\label{diagonalbasis.2}
|\lambda_i - \frac12|<\delta \mbox{ for } i =1,2,3\quad\mbox{ and }\quad|\lambda_4-\frac32| < \delta.
\end{equation}
\end{lemma}

\begin{proof}
Let $\delta>0$ be given. Let $\eta>0$ to be chosen small enough. Let $r_0$ be small enough that Lemma \ref{QvsK} is satisfied with $\eta$ in place of $\delta$. Fix $0<r\sleq r_0$, and let $(y_1, y_2, y_3, y_4)$ be the orthonormal basis given by Lemma \ref{QvsK}, i.e., such that
\begin{equation}\label{diagonalbasis.5}
\max_{ij}|q_{ij}-K_{ij}|\sleq \eta.
\end{equation}

 Note that the eigenvalues of $K$ are $1/2$ with multiplicity 3 and $3/2$ with multiplicity 1. Note that $Q_{0,r}$ has real eigenvalues because it is symmetric. Let $\lambda_1\sleq\lambda_2\sleq \lambda_3\sleq \lambda_4$ be the eigenvalues of $Q_{0,r}$.  By the theory of Gershgorin disks [see, for example \cite{Cha}], there exists $\eta$ small enough such that (\ref{diagonalbasis.2}) is satisfied. Further, because $Q_{0,r}$ is symmetric, there is an orthonormal basis $(x_1, x_2, x_3, x_4)$ diagonalizing $Q_{0,r}$ such that (\ref{diagonalbasis.1}) is satisfied.

\end{proof}

We now know a great deal about what $Q_{0,r}$ looks like, and so we also know a great deal about what $\twid Q_{0,r} = Q_{0,r}-|\cdot|^2$ looks like. We now seek to exploit this knowledge with a lemma about polynomials that look like $\twid Q_{0,r}$.

\begin{lemma}\label{Polynomial}
Let $(x_1,x_2,x_3,x_4)$ be orthonormal coordinates on $\RR^4$ and $P(x):\RR^4\to\RR$ be a polynomial of the form
\begin{equation}\label{Polynomial.1}
P(x) = \sum_{i = 1}^4 \eta_i x_i^2
\end{equation}
such that
\begin{equation}\label{Polynomial.2}
\begin{array}{c}
|\eta_i + \frac 12|< \frac 18\,\, \,\, \mbox{for }i\in\{1,2,3\}\\
|\eta_i - \frac 12|< \frac 18 \,\,\,\,\mbox{for }i = 4\\
\end{array}
\end{equation}
Let $\Sigma$ be a set such that
\begin{equation}\label{Polynomial.3}
 |P(x)|\sleq \epsilon \mbox{ for all }x\in \Sigma\cap B(0,1), 
\end{equation}
Then
\begin{equation}\label{Polynomial.4}
\hd{0,1}\Sigma{P\inv(0)} \sleq C\sqrt \epsilon.
\end{equation}

\end{lemma}
\begin{proof}
Let all notation and hypotheses hold. Let $x\in \Sigma\cap B(0,1)$. Write $x = (re, x_4)\in \RR^3\times \RR$, for $r\sgeq 0$ and $e = (e_1, e_2, e_3) \in \RR^3$,  $|e| = 1$. We consider two cases of (\ref{Polynomial.3}). First, we suppose that
\begin{equation}\label{Polynomial.5}
0\sleq P(x)\sleq \epsilon.
\end{equation}
By (\ref{Polynomial.2}), $\eta_i$ is negative for $i\in \{1,2,3\}$. Thus by (\ref{Polynomial.5}), there exists $\hat{r}\sgeq r$ such that $\hat{x}:= (\hat{r}e, x_4)$ satisfies
\begin{equation}\label{Polynomial.6}
P(\hat{x}) = 0.
\end{equation}
We next note that
\begin{equation}\label{Polynomial.7}
|x-\hat{x}| = |r-\hat{r}|.
\end{equation}
Let $\eta_e = \sum_{i=1}^3 \eta_i e_i^2$. Because $|e| = 1$ and (\ref{Polynomial.2}), we get that $|\eta_e + 1/2|\sleq 1/8$. We compute
\begin{equation}\label{Polynomial.8}
\epsilon\sgeq |P(x)-P(\hat{x})| = \left| \sum_{i=1}^3 \eta_i e_i^2 (r^2 - \hat{r}^2) \right| = |r^2 - \hat{r}^2| |\eta_e|.
\end{equation}
We recall that $\hat{r}\sgeq r\sgeq 0$. We note that if $a,b>0$, then $|a-b|\sleq \sqrt{|a^2 - b^2|}$. We use these observations and (\ref{Polynomial.8}) to compute
\begin{equation}\label{Polynomial.9}
|\hat{r} - r| \sleq \sqrt{|\hat{r}^2 - r^2|}\sleq \frac 1{\sqrt{|\eta_e|}}\sqrt \epsilon\sleq C \sqrt{\epsilon}.
\end{equation}

We now consider the other case of (\ref{Polynomial.3}). Suppose that
\begin{equation}\label{Polynomial.10}
-\epsilon\sleq P(x)\sleq 0.
\end{equation}
For ease, assume $x_4\sgeq 0$ (the case $x_4\sleq 0$ will be similar). Because $\eta_4>0$ by (\ref{Polynomial.2}), $x_4\sgeq 0$, and (\ref{Polynomial.10}), there exists $\hat{x_4}\sgeq x_4$ such that $\hat{x} :=(x_1, x_2, x_3, \hat{x_4})$ satisfies
\begin{equation}\label{Polynomial.11}
P(\hat{x}) = 0.
\end{equation}
Note that
\begin{equation}\label{Polynomial.12}
|x-\hat{x}| = |x_4- \hat{x_4}|.
\end{equation}
Note also that 
\begin{equation}\label{Polynomial.13}
\epsilon\sgeq |P(x)-P(\hat{x})| = |\eta_4 (x_4^2 - \hat{x_4}^2)| = |\eta_4| |x_4^2 - \hat{x_4}^2|.
\end{equation}
As before, we note that $\hat{x_4}\sgeq x_4\sgeq 0$. Hence, from (\ref{Polynomial.2}) and (\ref{Polynomial.13}), we get that
\begin{equation}\label{Polynomial.14}
|\hat{x_4}-x_4|\sleq \sqrt{|\hat{x_4}^2 - x_4^2|}\sleq \frac 1{\sqrt {\eta_4}} \sqrt \epsilon \sleq C\sqrt \epsilon.
\end{equation}

Hence, we've shown that for any $x\in \Sigma\cap B(0,1)$, there exists $\hat{x} \in P\inv (0)$ such that
$|x-\hat{x}|\sleq C\sqrt \epsilon.$
Hence,
\begin{equation}\label{Polynomial.16}
\hda{0,1}{\Sigma}{P\inv(0)}\sleq C\sqrt \epsilon.
\end{equation}
Because $P$ is a homogenous polynomial, $P\inv(0)$ is a cone centered at $0$. Hence, we apply Lemma \ref{ConeG1} to (\ref{Polynomial.16}) to get
\begin{equation}
\hd{0,1}{\Sigma}{P\inv(0)}\sleq C\sqrt \epsilon.
\end{equation}

\end{proof}

\subsection{The Tangent Cone at the Singularity}\label{atsing}

In this section we show that at a nonflat point there is a unique KP cone $\cC$ based at the origin to which the blowups converge at a H\"older rate. Specifically, we stive to show that there are an $r_0>0$ small enough, constant $C$, and exponent $\beta_1 = \beta_1(\alpha)$ such that for $0<r\sleq r_0$,
$$\D{0,r}\cC\Sigma\sleq C r^{\beta_1}.$$
To do this, we will apply Lemmas \ref{DKTbounds}(6), \ref{diagonalbasis}, and \ref{Polynomial}. Lemma \ref{diagonalbasis} allows us to diagonalize $Q_{0,r}$, which allows us to apply Lemma \ref{Polynomial} with the bounds from Lemma \ref{DKTbounds}(6). Taken together, these will give us the estimates on $\hd{0,r}\Sigma\cC$. To get estimates on $\hd{0,r}\cC\Sigma$, we will apply Theorem \ref{Top}, which will allow us to show that every point in $ \cC\cap B(0,r)$ has a point in $\Sigma$ which is no further than $\hd{0,2r}\Sigma\cC$, and this will finish the argument.

 Fix $\epsilon>0$ small. By Theorem \ref{AODstructure}, there is some radius $r_0$ small enough such that
\begin{equation}\label{atsing.0}
\vartheta_\Sigma^C(x,r)<\epsilon\quad\mbox{ for }x\in \Sigma\cap B(0,r_0), 0<r\sleq r_0.
\end{equation}
We work at a scale of $\rho = r^{1+\gamma_0}$ for some $0<\gamma_0<\alpha/2$. For ease, let $\rho_0 = r_0^{1+\gamma_0}$. Define $\cE(\rho)$ by
\begin{equation}\label{atsing.1}
\cE(\rho) = \twid Q_{0,r}\inv(0).
\end{equation}
We recall Lemma \ref{DKTbounds}(6), which tells us that that for $r_0<1/4$, $0<r\sleq r_0$,
\begin{equation}\label{atsing.2}
\twid Q_{0,\tau r}(z) \sleq C r^{\beta_0} \quad\mbox{ for all } z\in \Sigma_\rho \cap B(0,2) \mbox{ and } \tau\in\left[1/2,1\right].
\end{equation}
From Lemma \ref{diagonalbasis}, we know that for $r_0$ small enough, we can diagonalize $\twid Q_{0,r}$ as $\twid Q_{0,r} = \sum_{i=1}^4 (\lambda_i - 1) x_i^2$ with $|\lambda_i-1/2|<1/8$ for $i = 1,2,3$ and $\lambda_4 - 3/2|<1/8.$ This allows us to apply Lemma \ref{Polynomial} with the bound of (\ref{atsing.2}), which yields that
\begin{equation}\label{atsing.3'}
\hd{0,2}{\Sigma_\rho}{\cE((\tau r)^{1+\gamma_0})}\sleq C r\frexp{\beta_0} 2 \quad\mbox{ for all } \tau\in \left[1/2, 1\right].
\end{equation}
Manipulating the $\tau$ above,
\begin{equation}\label{atsing.3}
\hd{0,2}{\Sigma_\rho}{\cE(\tau \rho)}\sleq C r\frexp{\beta_0} 2 \quad\mbox{ for all } \tau\in \left[1/{2^{1+\gamma_0}}, 1\right].
\end{equation}
In particular, the above holds for all $\tau\in [1/2,1]$.

Let $\sigma>0$. By Theorem \ref{AODstructure}, for $\rho_0$ small enough and $0<\rho\sleq \rho_0$, there exists a KP cone $\cC(\rho)$ based at the origin such that
\begin{equation}\label{atsing.4}
\D{0,2}{\Sigma_\rho}{\cC(\rho)}<\sigma.
\end{equation}
From (\ref{atsing.3}) and (\ref{atsing.4}), it follows that
\begin{equation}\label{atsing.5}
\hd{0,2}{C(\rho)}{\cE(\tau\rho)}<\sigma + C r\frexp{\beta_0}2 \quad \mbox{ for all } 0<\rho\sleq \rho_0, \tau\in \left[1/2,1\right].
\end{equation}
Fix $\tau\in [1/2, 1]$. Let $a\in \cE(\tau\rho)$ such that$ 1/2\sleq |a|\sleq 1$. Let $\nu_a$ be the normal vector to $\cE(\tau\rho)$ at $a$, and let $\ell_a = \{a + t\nu_a\mid t\in \RR\}$ be the line going through $a$ parallel to $\nu_a$. 

Let $\delta>0$. Recall that $\cE(\tau\rho)$ is the zero set of a homogenous degree 2 polynomial. By (\ref{atsing.5}) and requiring that $\sigma$ and $\rho_0$ be small enough, we can not only guarantee that $\ell_a$ interesects $\cC(\rho)$, but also that at the point of intersection (nearest to $a$), the angle between $\ell_a$ and $\cC(\rho)$ is arbitrarily close to $\pi/2$. In particular, we can guarantee that it is greater than $\pi/ 4$. It follows from this observation, (\ref{atsing.0}), and (\ref{atsing.4}) that we may invoke Theorem \ref{Top} with $\nu_a$ at the point of intersection. Hence, for small enough $\rho_0$ and $\sigma$, there exists $z\in \Sigma_\rho \cap \ell_a$ such that $d(z, \cC(\rho))\sleq \delta$. Coupling this with (\ref{atsing.5}), we get that
\begin{equation}\label{atsing.6}
d(z, \cE(\tau\rho))\sleq \delta + 2\sigma + Cr\frexp{\beta_0}2.
\end{equation}
For $\delta$, $\sigma$, and $\rho_0$ small enough, the fact that $\nu_a$ is normal to $\cE(\tau\rho)$ implies that the nearest point on $\cE(\tau\rho)$ to $z$ is $a$. Hence,
\begin{equation}\label{atsing.7}
|z-a| = d(z, \cE(\tau\rho))\sleq 2 \hd{0,2}{\Sigma_\rho}{\cE(\tau\rho)}\sleq Cr\frexp{\beta_0}2.
\end{equation}
Recall that our only assumption on $a$ was that $1/2\sleq |a|\sleq 1$ and $a\in \cE(\tau\rho)$. Hence, (\ref{atsing.7}) tells us that
\begin{equation}\label{atsing.8}
\hd{0,1}{\cE(\tau\rho)\setminus B\left(0,1/2\right)}{\Sigma_\rho}\sleq Cr\frexp {\beta_0 }2.
\end{equation}
Hence, combining (\ref{atsing.3}) and (\ref{atsing.8}), we get that
\begin{equation}\label{atsing.9}
\hda{0,1}{\cE(\tau\rho)\setminus B\left(0,1/2\right)}{\cE(\tau'\rho)}\sleq Cr\frexp {\beta_0}2 \quad\mbox{ for } \tau, \tau' \in \left[1/2,1\right].
\end{equation}
From the fact that both $\cE(\tau\rho)$ and $\cE(\tau'\rho)$ are cones based at 0, it follows from (\ref{atsing.9}) that
\begin{equation}\label{atsing.10}
\hd{0,1}{\cE(\tau\rho)}{\cE(\tau'\rho)}\sleq Cr\frexp {\beta_0}2 \quad\mbox{ for } \tau, \tau' \in \left[1/2,1\right]
\end{equation}
(see Lemma \ref{ConeG1}). Because condition (\ref{atsing.10}) is symmetric, we get that
\begin{equation}\label{atsing.11}
\D{0,1}{\cE(\tau\rho)}{\cE(\tau'\rho)}\sleq Cr\frexp {\beta_0}2 \quad\mbox{ for } \tau,\tau' \in \left[1/2,1\right].
\end{equation}

We now exploit (\ref{atsing.11}) to get a rate of convergence for the one parameter family $\cE(\rho)$. Suppose that $0<\rho'<\rho\sleq \rho_0$. Write $\rho' = \tau\rho/{2^N}$ for $\tau\in[1/2,1]$, $N\in\ZZ_{\sgeq 0}$. Then we get from (\ref{atsing.11}) that
\begin{equation}\label{atsing.12}
\begin{split}
\D{0,1}{\cE(\rho)}{\cE(\rho')}&\sleq \sum_{j=0}^{N-1} \D{0,1}{\cE\left(\frac\rho{2^j}\right)}{\cE\left(\frac{\rho}{2^{j+1}}\right)} + \D{0,1}{\cE\left(\frac\rho{2^N}\right)}{\cE(\rho')}\\
&\sleq \sum_{j=0}^{N} C \left(\frac r{2^j}\right)\frexp{\beta_0}2\sleq C r\frexp{\beta_0}2\sum_{j=0}^\infty \left(\frac 1{2\frexp {\beta_0}2}\right)^j = C r\frexp {\beta_0}2.
\end{split}
\end{equation}

We now recall Lemma \ref{diagonalbasis} and that $\cE(\rho) = \twid Q_{0,r}\inv(0).$ Together, they imply that any convergent sequence $\cE(\rho_i)$ is converging to a KP cone based at the origin. Moreover, by (\ref{atsing.12}), the one parameter family is Cauchy and hence converges to a unique KP cone based at the origin. Call this KP cone $\cC$. We now use (\ref{atsing.12}) to compute that for $0<\rho\sleq \rho_0$,
\begin{equation}\label{atsing.13}
\D{0,1}{\cE(\rho)}{\cC}\sleq \lim_{\rho'\downarrow 0} \D{0,1}{\cE(\rho)}{\cE(\rho')} + \D{0,1}{\cE(\rho')}{\cC} \sleq C r\frexp{\beta_0}2.
\end{equation}
We now wish to show that
\begin{equation}\label{atsing.14}
\D{0,1}{\Sigma_\rho}{\cC}\sleq C r\frexp{\beta_0}2.
\end{equation}
First, we note that by (\ref{atsing.3}), Lemma \ref{ConeG1}, and (\ref{atsing.13}), we get that
\begin{equation}\label{atsing.15}
\hd{0,1}{\Sigma_\rho}{\cC}\sleq \hd{0,1}{\Sigma_\rho}{\cE(\rho)} + \hd{0,1}{\cE(\rho)}{\cC} \sleq C r\frexp{\beta_0}2.
\end{equation}
We now recall that by Theorem \ref{AODstructure}, we have that for any sequence $\rho_i\downarrow 0$ with $\Sigma_{\rho_i}$ convergent, $\Sigma_{\rho_i}\to \cC'$ for some KP cone $\cC'$. We note that (\ref{atsing.15}) implies that $\cC' \subseteq \cC$, and because each is a KP cone, $\cC' = \cC$. Hence, we derive that $\D{0,1}{\Sigma_{\rho_i}}{\cC}\to 0$ as $i\to \infty$ for any convergent sequence $\rho_i$. From this observation, we get that
\begin{equation}\label{atsing.16}
\D{0,1}{\Sigma_\rho}{\cC}\to 0 \quad\mbox{ as }\rho\downarrow 0.
\end{equation}
Let $0< \rho\sleq \rho_0$ and $a\in \cC\cap B(0,1)\setminus B(0,1/2)$. Similarly to how we proved (\ref{atsing.8}), we apply (\ref{atsing.0}), (\ref{atsing.15}), (\ref{atsing.16}) to apply Theorem \ref{Top} to the vector $\nu_a$ (which we recall is the unit normal to $\cC$ at $a$). This gives us that there is a point $z\in \Sigma_\rho$ with $z = a + t\nu_a$ and $t$ small. As before, we get that $|z-a| = d(z,\cC)\sleq |a| \hd{0,|a|}{\Sigma}\cC\sleq C r\frexp{\beta_0}{2}$. Hence, we get that
\begin{equation}\label{atsing.17}
\hd{0,1}{\cC\setminus B(0,1/2)}{\Sigma_\rho}\sleq C r\frexp{\beta_0}2.
\end{equation}
Let $a\in \cC\cap B(0,1)\setminus\{0\}$. Then we apply scale invariance of $\cC$ and (\ref{atsing.17}) to compute that
\begin{equation}
\hd{0,|a|}{\cC\setminus B(0,|a|/2)}{\Sigma_\rho} = \hd{0,1}{\cC\setminus B(0,1/2)}{\Sigma_{|a|\rho}}\sleq C r\frexp{\beta_0}2.
\end{equation}
Hence,
\begin{equation}
d(a, \Sigma_\rho)\sleq C |a| r\frexp{\beta_0}{2}\sleq Cr\frexp {\beta_0}2.
\end{equation}
From which we get
\begin{equation}\label{atsing.18}
\hd{0,1}\cC{\Sigma_\rho}\sleq Cr\frexp{\beta_0}2.
\end{equation}
Combining (\ref{atsing.15}) and (\ref{atsing.18}), 
\begin{equation}
\D{0,1}{\Sigma_\rho}\cC\sleq Cr\frexp{\beta_0}2.
\end{equation}
It is now convenient to switch back from $\rho = r^{1+\gamma_0}$ to $r$. Making the substitution to (\ref{atsing.17}), we get
\begin{equation}
\D{0,1}{\Sigma_r}\cC \sleq C r\frexp{\beta_0}{2(1+\gamma_0)}.
\end{equation}
Let $\beta_1 = {\beta_0}/{(2(1+\gamma_0))}$. Recall that $\beta_0$ and $\gamma_0$ depended only on $\alpha$, and hence so does $\beta_1$. Then we've shown the following.

\begin{lemma}\label{singularpart}
 Recall hypotheses (\ref{4hypos}) and (\ref{atsing.0}). There is a (unique) KP cone $\cC$ based at the origin, $r_0 = r_0(C_0, \alpha, \epsilon) >0$ small enough, a constant $C = C(C_0, \alpha)$, and $\beta_1 = \beta_1(\alpha)$ (defined above) such that for $0<r\sleq r_0$,
\begin{equation}
\D{0,1}{\Sigma_r}\cC\sleq C r^{\beta_1}.
\end{equation}
\end{lemma}

\subsection{H\"older Closeness Away from the Origin}\label{AwayFromOrigin}
In this section we investigate how the quantity $\theta_\Sigma^P(x,r)$ behaves for points $x$ near the origin and scales $r$ which are appropriately small in terms of $|x|$. Recall hypotheses (\ref{4hypos}). We begin by stating a result of \cite{DKT}.

\begin{theorem}\label{DKTprop} Suppose that $\mu$ is a Radon measure on $\RR^n$ which is $(\alpha,n-1)$-H\"older asymptotically optimally doubling with convergence constants  $C_K$ for each compact set $K\subseteq \Sigma$. Then there exist $\delta = \delta(n,\alpha)$ and $r_1 = r_1(C_K, \alpha, n,\delta)$ such that if $r_1'\sleq r_1$ and 
$$\sup_{x\in K, 0<r\sleq r_1'}\theta_\Sigma^P(x,r)\sleq \delta,$$
 then there exist $C_K' = C_K'(C_K, \delta, n,\alpha)$ (but not otherwise dependent on $\mu$) and $\beta = \beta(\alpha)$ such that for all $0<r\sleq r_1'$ and $x\in K$,
$$\theta_\Sigma^P(x,r)\sleq C_K'\left(\frac r {r_1'}\right)^{\beta}.$$
\end{theorem}

\begin{remark}
Although this result is slightly more quantitative than  \cite{DKT} Proposition 8.6, it is not difficult to obtain. The result is proven in Section 8 of \cite{DKT}, and the only place where a more complicated relationship between $\delta$ and $r_1$  can arise is in the proof of Lemma 8.2. However, one removes all ambiguity of the dependence in this proof by simply choosing $r_0$ small enough such that $9\epsilon (\gamma_1)\sleq 1/(8n+8)$ and $\delta$ small enough such that $9C\delta\sleq1/(8n+8)$.
\end{remark}

In light of Theorem \ref{DKTprop}, we recall Corollary \ref{FlatSides}, which tells us that there is an $A$ large enough such that for some $r_0>0$,
\begin{equation}
\theta_\Sigma^P(x,r)< \delta \quad\mbox{ for all }x\in \Sigma\cap B(0,r_0)\setminus \{0\}, 0<r\sleq \frac {|x|}A.
\end{equation}
Hence we can apply Theorem \ref{DKTprop} to get
\begin{equation}\theta_\Sigma^P\left(x,r\right)\sleq C\left(\frac r{r_1'}\right)^{\beta}\quad\mbox{ for all } x\in\Sigma\cap B(0,r_0)\setminus\{0\}, 0<r\sleq r_1'\sleq \frac {|x|}A. \end{equation}
Taking
\begin{equation}r_1'=\frac1A |x|\end{equation}
gives that
\begin{equation}\label{8.8}
\theta_\Sigma^P\left(x,r\right)\sleq C \left(\frac{r}{|x|}\right)^\beta.
\end{equation}

We now fix $0<\gamma_1<\beta_1$.
Let $0<r\sleq{|x|^{1+\gamma_1}}/{A}$, we get that ${1}/{|x|}\sleq C r\frexp1{1+\gamma_1}$. Combining this with (\ref{8.8}) gives us that
\begin{equation}\label{8.9}
\theta_\Sigma^P\left(x,r\right)\sleq C \left(\frac{r}{r\frexp{1}{1+\gamma_1}}\right)^\beta = Cr\frexp{\beta\gamma_1}{1+\gamma_1}.
\end{equation}
Setting $\beta_2 = \frac{\gamma_1\beta}{1+\gamma_1}$, (\ref{8.9}) says that
\begin{equation}\theta_\Sigma^P\left(x,r\right)\sleq C r^{\beta_2}.\end{equation}
Hence, we have proven the following.

\begin{lemma}\label{flatparts}
Recall hypotheses (\ref{4hypos}).  There is an $r_0>0$ and a constant $C$ small enough such that for all $0<r\sleq r_0$,
\begin{equation}
\theta_\Sigma^P\left(x,r\right)\sleq C r^{\beta_2}.
\end{equation}
\end{lemma}

\section{Parametrization}\label{Parametrization}

In this section, we use the geometric information we have gathered to construct a $C^{1,\beta}$ parametrization of a neighborhood of 0 by a KP cone. We work toward the Theorem \ref{parametrization}, which when paired with Lemmas \ref{singularpart} and \ref{flatparts}, proves Theorem \ref{theorem}.

\begin{theorem}\label{parametrization}
Let $\Sigma\subseteq\RR^4$ be a closed set containing $0$,  $0<\gamma<\beta_1$, and $0<\beta_2$. There exists $\beta = \beta(\gamma,\beta_1,\beta_2)$ with the following property. Let $\cC$ be a KP cone cnetered at 0, and assume the following estimates on $\Sigma$:
\begin{enumerate}
\item[(E0)] for $x\in B(0,r_0)\cap \Sigma$ and $0<r\sleq 2 r_0$, $\vartheta^C_\Sigma(x,r)<\epsilon$,
\item[(E1)] for $0<r\sleq 2 r_0$, $\D{0,r}\Sigma\cC\sleq \min(\sigma, C_1 r^{\beta_1}),$
\item[(E2)]  for $x\in \Sigma\cap B(0,r_0)\setminus\{0\}$, $0<r\sleq 16 \frac{|x|^{1+\gamma}}{A}$, $\theta_\Sigma^P(x,r) \sleq\min(\delta, C_2 r^{\beta_2})$.
\end{enumerate}
For $A>1$ large enough and $\sigma, \delta, r_0>0$ small enough, we have that $\Sigma$ admits a $C^{1,\beta}$ parametrization by $\cC$. That is, there exist neighborhoods $U$ of $0$ and $U'$ of $0$ and a diffeomorphism $\vphi\in C^{1,\beta}(U\to U')$ such that $\vphi(\cC\cap U) = \Sigma\cap U'$. Further, $\vphi$ has the property that $\vphi(0) = 0$ and $D_0\vphi = \Id$, and $U$ has the property that $U\cap \cC \supseteq B(0,r_0)\cap \cC$.
\end{theorem}

 Let $\Sigma\subseteq \RR^4$ be a closed set containing 0. We fix exponents $0<\gamma<\beta_1$, and $0<\beta_2$. The parameters $\epsilon, \sigma, \delta, 1/A$ and $r_0$ will be chosen small enough throughout this section. Let $(x_1, x_2, x_3, x_4)$ be orthonormal coordinates centered at the origin and $\cC = \{x_4^2 = x_1^2 + x_2^2 + x_3^2\}$. We assume the following estimates on $\Sigma$.

For $x\in\Sigma\cap B(0,r_0), \setminus \{0\}$ and $0<r\sleq {|x|^{1+\gamma}}/{A}$, let $P(x,r)$ be a plane such that
\begin{equation}\label{P(x,r)}
\D{x,r}{P(x,r)} \Sigma\sleq \min(\delta, C r^{\beta_2} )\quad\mbox{ and }x\in P(x,r).
\end{equation}
For each cross section $\cC_h = \cC\cap\{x_4 = h\}$, we note that $\cC_h$ is the 2 sphere of radius $h$ centered at $(0,h)$ inside of the plane $\{x_4 = h\}$. This means that nearest point projection in the cross section is defined for all $x\neq (0,h)$. Let $\tau(x): \RR^4\to \RR^3$ be orthogonal projection onto the first 3 coordinates. We thus define $\pi:\RR^4\setminus \{x: \tau(x) = 0\}\to \cC$ to be nearest point projection in the cross section.  One can check that
\begin{equation}\label{nearestpoint}
\pi(x) = \left(\frac {|x_4|}{|\tau(x)|}\tau(x), x_4\right).
\end{equation}
We define the vector field $\eta$ on $\RR^4 \setminus \RR x_4$  by 
\leqn{eta_a}
\eta_x = \frac1{|\tau(x)|}\left({-\tau(x)},0\right).
\endeqn
Note that for $x = a\in \cC$, this is the vector normal to the cross section $\cC_{a_4}$ at $a$ viewed in the plane $\{x_4 = a_4\}$. Hence
\begin{equation}
\pi(x)-x = \pm d(x,\cC_{x_4}) \eta_{\pi(x)}
\end{equation}
depending on whether $x$ is inside or outside of the sphere $\cC_{a_4}$. We also note that $\pi(x)$ is the nearest point in $\cC$ to $x$ along the line based at $x$ in the direction $\eta_x$. Thus, for $a\in \RR^4\setminus \RR x_4$, we define the half line
\begin{equation}\label{ell_a}
\ell_a = \{b\in \RR^4: b-a = t\eta_a \mbox{ for some }t\in \RR \mbox{ and } \tau(a)\cdot\tau(b)\sgeq0\}.
\end{equation}
Note that the half lines $\ell_a$ are the integral curves of $\eta$ and that for any $x\in \RR^4-\RR x_4$, $\pi(x)$ is the unique intersection of the integral curve containing $x$ with $\cC$. 
 For $a\in\cC$, note that $\eta_a$ is the normal vector to the cross section $\cC_{a_4}$, as opposed to the normal vector to the cone $\cC$, which is 
\begin{equation}\label{nu_a}
\nu_a = \frac{1}{|a|}\left(-\tau(a), a_4\right).
\end{equation}
Further, we note the following.

\begin{lemma}\label{pidistnaught}
Recall hypotheses (E0)-(E2).
For $x\in \RR^4\setminus \RR x_4$,
$|x-\pi(x)| = \sec(\pi/4)\, d(x,\cC).$
\end{lemma}

\begin{proof}
First, we note that because $\cC$ is a smooth manifold away from $0$ and $0$ is not the closest point to $x$, the distance $d(x,\cC)$ is the length of the vector based at $x$ in the $\nu_{\pi(x)}$ direction ending on $\cC$. Second, we note that the angle between $\nu_{\pi(x)}$ and $\eta_{\pi(x)}$ is $\frac\pi4$. Thus because $\cC$ is a cone, the vector based at $x$ pointing in the $\eta_x$ ending on $\cC$ has length $\sec(\frac\pi 4) d(x,\cC)$.
\end{proof}

\begin{lemma}\label{pidist}
Recall hypotheses (E0)-(E2).
For $x\in \Sigma\cap B(0,r_0)\setminus \{0\}$, 
\begin{equation}
|\pi(x)-x|\sleq \min(C |x|^{1 + \beta_1}, 2 \sigma |x|).
\end{equation}
\end{lemma}

\begin{proof}
Let $r = 2|x|$. Applying Lemma \ref{pidistnaught} and assumption (E1), we get that
\begin{equation}
d(x,\cC)\sleq \min(C|x|^{1+\beta_1}, 2\sigma |x|).
\end{equation}
We then apply Lemma \ref{pidistnaught} and conclude.
\end{proof}

\begin{lemma}\label{Angle1} 
Recall hypotheses (E0)-(E2) and (\ref{P(x,r)}). For $r_0$ small enough, depending on $C_2$, $\beta_2$ and $\gamma$, we have the following. Let $x_1, x_2\in \Sigma\cap B(x,r_0), 0<t_1, t_2<r_0$ such that
\begin{equation}\label{Angle1.1}
\frac14 t_1\sleq\frac12 t_2\sleq t_1 \sleq \frac{16}A |x_1|^{1+\gamma}\quad\mbox{ and }\quad|x_1-x_2|\sleq \frac{t_1}{16}.
\end{equation}
Then $\meang(P(x_1, t_1), P(x_2,t_2)) \sleq C t_1^{\beta_2}.$
\end{lemma}

\begin{proof}

By Lemma \ref{planes}, it will suffice to find a radius $r$ such that
\begin{equation}\label{angle.0}
\hd{x_1,r}{P(x_1,t_1)}{P(x_2, t_2)}\sleq C t_1^{\beta_2}.
\end{equation}
We take $r$ to be $t_1/8$. Let $x\in P_1\cap B(x_1, t_1/8)$. By (\ref{P(x,r)}) there exists $y\in \Sigma\cap B(x_1,t_1)$ with $|x-y|\sleq Ct_1^{1+\beta_2}$. Further, we may take $|x-y|\sleq \delta t_1\sleq t_1/8$. We estimate that
\begin{equation}
|y- x_2|\sleq |y-x| + |x-x_1| + |x_1-x_2|\sleq 5\frac{ t_1}{16} < \frac{t_1}2\sleq t_2.
\end{equation}
Hence, $y\in B(x_2,t_2)$. So by assumption there exists $z\in P(x_2,t_2)\cap B(x_2,t_2)$ such that
$$|y-z|\sleq C t_2^{1+\beta_2}\sleq C t_1^{1+\beta_2}.$$
Thus,
\begin{equation}\label{angle.1}
|x-z|\sleq |x-y|+|y-z| \sleq C t_1^{1+\beta_2}.
\end{equation}
Because for each $x\in P(x_1,t_1)\cap B(x_1, t_1/8)$ there exists $z\in P(x_2, t_2)$ such that (\ref{angle.1}) holds, we have that
\begin{equation}
\hda{x_1,\frac{ t_1}{8}}{P(x_1,t_1)}{P(x_2,t_2)}\sleq C t_1^{\beta_2}.
\end{equation}
Because $P_1$ is a cone based at $x_1,$ by Lemma \ref{ConeG1} we obtain that (\ref{angle.0}) holds for $r=t_1/8$, and we are done.
\end{proof}

\begin{corollary}
Recall hypotheses (E0)-(E2). For $x,y\in B(0, r_0)$, the following hold:
\begin{enumerate}[(1)]\label{angle}
\item $\meang(P(x,r),P(x,r'))\sleq Cr^{\beta_2}$ for $ 0<r'\sleq r \sleq 16{|x|^{1+\gamma}}/A$
\item $P(x) = \lim_{r\downarrow 0} P(x,r)$ exists
\item\label{xrangle} $\meang\left(P\left(x,r\right), P(x)\right) \sleq C r^{\beta_2}\quad\mbox{for}\quad 0< r\sleq {|x|^{1+\gamma}}/{A}$
\item\label{xyangle}$
\meang\left(P(x),P(y)\right) \sleq C |x-y|^{\beta_2} \quad\mbox{when }  |x-y|\sleq {|x|^{1+\gamma}}/{A}$
\item\label{tangentdistance} $d(y,P(x))\sleq C|x-y|^{1+\beta_2}\quad\mbox{when }|x-y|\sleq {|x|^{1+\gamma}}/{A}$
\end{enumerate}
\end{corollary}

\begin{proof}

Let $B\in[1/2,1)$ and write $r' = B r/{2^j}$ for $j\in \NN$. By Lemma \ref{planes}, Lemma \ref{Angle1}, and subadditivity of angles we have that
\begin{equation}\label{angle.3}
\begin{split}
\meang(P(x,r), P(x,r'))&\sleq \sum_{i=0}^{j-1} \meang\left({P\left(x,{r}/{2^i}\right), P\left(x,{r}/{2^{i+1}}\right)}\right) + \meang\left(P\left(x,{r}/{2^j}\right), P\left(x,{Br}/{2^j}\right)\right)\\
&\sleq C \sum_{i=0}^{j} \left({r}/{2^i}\right)^{\beta_2} \sleq C r^{\beta_2}\sum_{i=0}^\infty
\left({1}/{2^{\beta_2}}\right)^{i} = C r^{\beta_2},
\end{split}
\end{equation}
establishing (1). In particular, it follows that the one parameter family $P(x,r)$ is Cauchy as $r\downarrow 0$, establishing (2). With the help of (1) and subadditivity of angles, we compute
\begin{equation}\label{angle.4}
\begin{split}
\meang(P(x,r),P(x))&\sleq \lim_{r'\downarrow 0}\meang(P(x,r), P(x,r')) + \meang(P(x,r'),P(x))\\
&\sleq C r^{\beta_2} + \lim_{r'\downarrow 0} \meang(P(x,r'),P(x)) = C r^{\beta_2} ,
\end{split}
\end{equation}
establishing (3).

We now prove (4). Let $x\in B(0,r_0)\cap\Sigma\setminus\{0\}$ and $|x-y|\sleq {|x|^{1+\gamma}}/A$. Let $r = 16|x-y|$. Then we apply (3), Lemma \ref{Angle1}, and subadditivity of angles to compute
\begin{equation}
\begin{split}
\meang(P(x),P(y))&\sleq \meang(P(x),P(x,r)) + \meang(P(x,r),P(y,r)) + \meang(P(y,r),P(y))\\
&\sleq C r ^{\beta_2}= C |x-y| ^{\beta_2}.
\end{split}
\end{equation}
Finally, we prove (5). Under the same assumptions that $x\in B(0,r_0)\cap\Sigma\setminus\{0\}$ and $|x-y|\sleq {|x|^{1+\gamma}}/A$, we apply (\ref{P(x,r)}) and (3) to get that
\begin{equation}
\begin{split}
d(y,P(x))&\sleq d(y,P(x,2|x-y|)) + d(P(x)\cap B(x, 2|x-y|), P(x,2|x-y|)\cap B(x,2|x-y|))\\
&\sleq d(y,P(x,2|x-y|)) + 2|x-y|\D{x,2|x-y|}{P(x)}{ P(x,2|x-y|)}\\
&\sleq C|x-y|^{1+{\beta_2}}.
\end{split}
\end{equation}

\end{proof}

\begin{lemma}\label{ConeToPlane}
Recall hypotheses (E0)-(E2). Let $x\in \Sigma\cap B(0,r_0)$,  $R = 8 |x|^{1+\gamma}/A$, and $0<r\sleq 2R$. Then
\begin{equation}\label{ConeToPlane.1}
\hd{x,R}\cC{P(x,r)}\sleq C |x|^{\min(\beta_1-\gamma, \beta_2(1+\gamma))}.
\end{equation}
\end{lemma}

\begin{proof}
Let $a\in \cC\cap B(x,R).$ Note that $B(x,R)\subseteq B(0, 2|x|)$. Hence by (E1) there exists $y\in \Sigma\cap B(0,2|x|)$ such that
\begin{equation}\label{ConeToPlane.2}
|a-y|\sleq C|x|^{1+\beta_1} = R\left(C|x|^{\beta_1-\gamma}\right).
\end{equation}
Further, we have that
\begin{equation}\label{ConeToPlane.3}
|x-y|\sleq |x-a| + |a-y| \sleq R + R\left(C|x|^{\beta_1-\gamma}\right)\sleq R\left(1 + C|r_0|^{\beta_1-\gamma}\right)\sleq 2R,
\end{equation}
for $r_0$ small enough.
Hence, $y\in B(x,2R)\cap \Sigma$. Because $2R = 16|x|^{1+\gamma}/A$, (E2) tells us that there exists $p\in P(x,2R)$ such that
\begin{equation}\label{ConeToPlane.5}
|y-p|\sleq 2R \cdot CR^{\beta_2} = R\cdot C|x|^{\beta_2(1+\gamma)}.
\end{equation}
Putting (\ref{ConeToPlane.2}) and (\ref{ConeToPlane.5}) together, we get
\begin{equation}\label{ConeToPlane.6}
|a-p|\sleq |a-y| + |y-p|\sleq CR\left(|x|^{\beta_1-\gamma} + |x|^{\beta_2(1+\gamma)}\right).
\end{equation}
Because for each $a\in B(x,R)$ there exists a $p\in P(x,2R)$ such that (\ref{ConeToPlane.6}) is satisfied, we get that
\begin{equation}\label{ConeToPlane.7}
\hda{x,R}\cC{P(x,2R)}\sleq C|x|^{\min(\beta_1-\gamma,\beta_2(1+\gamma)}.
\end{equation}
Thus, by Lemma \ref{ConeG1}, (\ref{ConeToPlane.7}) tells us that
\begin{equation}\label{ConeToPlane.8}
\hd{x,R}\cC{P(x,2R)}\sleq C|x|^{\min(\beta_1-\gamma,\beta_2(1+\gamma))}.
\end{equation}
Corollary \ref{angle}, (\ref{distlessthanangle}), and (\ref{ConeToPlane.8}) tells us that
\begin{equation}\label{ConeToPlane.9}
\hd{x,R}\cC{P(x,r)}\sleq \hd{x,R}\cC{P(x,2R)} + \hd{x,R}{P(x,2R)}{P(x,r)}\sleq C|x|^{\min(\beta_1-\gamma,\beta_2(1+\gamma))}.
\end{equation}

\end{proof}

We now derive information about the angle $P(x)$ makes with the tangent plane $T_{\pi(x)}\cC$. We define
\begin{equation}\label{beta3}
\beta_3 = \min(\beta_1-\gamma, \beta_2(1+\gamma), \gamma).
\end{equation}

\begin{lemma}\label{tangentangle}
Recall hypotheses (E0)-(E2).  For $x\in \Sigma\cap B(0,r_0)$, $\meang\left(P(x), T_{\pi(x)} \cC\right)\sleq C|x|^{\beta_3}.$
\end{lemma}

\begin{proof}
Let $R = 8 {|x|^{1+\gamma}}/A$. By Lemma \ref{ConeToPlane}, we have that
\begin{equation}\label{tangentangle.1}
\hd{x,R}\cC{P(x,r)}\sleq C|x|^{\min(\beta_1-\gamma, \beta_2(1+\gamma)}.
\end{equation}
Letting $r\downarrow 0$, we get
\begin{equation}\label{tangentangle.2}
\hd{x,R}\cC{P(x)}\sleq C|x|^{\min(\beta_1-\gamma,\beta_2(1+\gamma))}.
\end{equation}
We apply Lemma \ref{pidist} to get
\begin{equation}\label{tangentangle.5}
|\pi(x)-x|\sleq C|x|^{1+\beta_1}\sleq 8\frac{|x|^{1+\gamma}}{A} \left(Cr_0^{\beta_1-\gamma}\right) = R \left(Cr_0^{\beta_1-\gamma}\right),
\end{equation}
and hence $|x|$ and $|\pi(x)|$ are within a factor of $2$ for $r_0$ small enough.
We use this observation and Lemma \ref{KPflat1} to see that
\begin{equation}\label{tangentangle.3}
\hd{\pi(x), \frac R2}{T_{\pi(x)}\cC}{\cC}\sleq C\frac R{|x|} = C|x|^\gamma.
\end{equation}
Let $p\in T_{\pi(x)}\cC\cap B(\pi(x), R/2)$. Then (\ref{tangentangle.3}) says there exists $c\in \cC\cap B(\pi(x), R/2)$ such that
\begin{equation}\label{tangentangle.4}
|p-c|\sleq C R|x|^\gamma.
\end{equation}
Next, we claim that for $r_0$ small enough, $c\in B(x,R)$. 
For $r_0$ small enough, (\ref{tangentangle.5}) tells us that
$|\pi(x)-x|\sleq R/2.$
Hence, $c\in B(\pi(x), R/2)\subseteq B(x, R)$. By (\ref{tangentangle.2}), there exists $q\in P(x)$ such that 
\begin{equation}\label{tangentangle.7}
|c-q|\sleq CR |x|^{\min(\beta_1-\gamma, \beta_2(1+\gamma))}.
\end{equation}
From (\ref{tangentangle.4}) and (\ref{tangentangle.7}), we get that
\begin{equation}\label{tangentangle.8}
|p-q|\sleq CR|x|^{\min(\beta_1-\gamma, \beta_2(1+\gamma))} + CR |x|^{\gamma}\sleq CR|x|^{\beta_3}
\end{equation}
(recall (\ref{beta3})).
Because for each $p\in T_{\pi(x)}\cC\cap B(\pi(x), R/2)$ there exists $q\in P(x)$ such that (\ref{tangentangle.8}) is satisfied, we have that
\begin{equation}\label{tangentangle.9}
\hda{\pi(x), \frac R2}{T_{\pi(x)}\cC}{P(x)}\sleq C|x|^{\beta_3}.
\end{equation}
Because $T_{\pi(x)}\cC$ is a cone through $\pi(x)$, Lemma \ref{ConeG1} and (\ref{tangentangle.9}) give us
\begin{equation}\label{tangentangle.10}
\hd{\pi(x), \frac R2}{T_{\pi(x)}\cC}{P(x)}\sleq C|x|^{\beta_3}.
\end{equation}
Lemma \ref{tangentangle} follows from Lemma \ref{planes}.

\end{proof}

Let $O = O(r_0,\delta)$ be the set
\begin{equation}
O = \left\{x\in\RR^n:|x_4|\sleq \frac{r_0}{\sqrt2}, \big||\tau(x)| - |x_4|\big|\sleq 2\delta |x_4| \right\}.
\end{equation}
We note that $\cC \cap B(0,r_0) = \cC\cap O$.  Lemma \ref{ConeToPlane} allows us to prove that $\pi|_{\Sigma\cap O}$ is a lower Lipschitz map surjective onto $\cC\cap O$  (see (\ref{nearestpoint})). 

\begin{lemma}\label{bilip}
Recall hypotheses (E0)-(E2).
For $r_0$ small enough, $\pi|_{\Sigma\cap O}$ is lower Lipschitz and $\pi(\Sigma\cap O) = \cC\cap O$.
\end{lemma}

\begin{proof}
First we prove surjectivity which is an application of Theorem \ref{Top}. Fix $a\in B(0,r_0)\cap \cC$, and consider the unit vector $\eta_a$. By (E0) and (E1), we have that if $\epsilon$ and $\sigma$ are small enough, we may apply Theorem \ref{Top}. It follows that there exist a $|t|\sleq \delta|a|$ and an $x\in \Sigma$ with $a  + t\eta_a = x.$
Next we claim that $x\in O$. It follows from the definition of $x$ that $x_4 = a_4$ and $\tau(x) = \tau(a) + \tau(t\eta_a) = \tau(a) + t\eta_a$. Hence,
\begin{equation}
\big||\tau(x)| - |x_4|\big| = \big||\tau(a)| + |t\eta_a| - |a_4|\big|\sleq \big|\tau(a)-|a_4|\big| + \delta|a| = 0 + \sqrt2 \delta|a_4| = \sqrt2\delta|x_4|.
\end{equation}
We finish by noting that $\pi(x) = a$ because $a\in \ell_x\cap \cC$. Hence, $\pi(\cC\cap O) = \Sigma\cap O$.

Next, we prove that $\pi|_{\Sigma\cap O}$ is lower Lipschitz for $x,y\in \Sigma\cap O$ sufficiently far apart. That is, we assume
\begin{equation}\label{bilip.1}
|x-y|\sgeq \frac{\max(|x|,|y|)^{1+\gamma}}{A}.
\end{equation}
By Lemma \ref{pidist}, we have that 
\begin{equation}\label{bilip.2}
|\pi(x)-x|\sleq C|x|^{1+\beta_1},\quad|\pi(y)-y|\sleq C|y|^{1+\beta_1}.
\end{equation}
By applying (\ref{bilip.1}) and (\ref{bilip.2}), we get
\begin{equation}\label{bilip.3}
\begin{split}
|\pi(x)-\pi(y)|&\sgeq |x-y| - |\pi(x)-x| - |\pi(y)-y|\sgeq |x-y| - C\left(|x|^{1+\beta_1} + |y|^{1+\beta_1}\right)\\
&\sgeq |x-y| -C\max(|x|,|y|)^{1+\beta_1}\sgeq |x-y| - C r_0^{\beta_1-\gamma}\cdot \max(|x|,|y|)^{1+\gamma}\\
&\sgeq |x-y| - Cr_0^{\beta_1-\gamma} |x-y|
\end{split}
\end{equation}
(recall that $\beta_1>\gamma)$. Hence, for $r_0$ small enough, we get that
\begin{equation}\label{bilip.4}
|\pi(x)-\pi(y)|\sgeq \frac 12|x-y| \quad\mbox{ for $x,y\in \Sigma\cap O$ such that }|x-y|\sgeq \frac{\max(|x|,|y|)^{1+\gamma}}{A}.
\end{equation}


To prove that $\pi|_{\Sigma\cap O}$ is lower Lipschitz on points which are very close together, we need two lemmas whose proofs appear in the appendix. The first one tells us how flatness of a set is perturbed by a $C^2$ diffeomorphism.
\begin{lemmaC2}
Suppose that $U,V\subseteq \RR^n$ are open sets, and $\psi\in C^2(U,V)$ is bijective and satisfies
\begin{equation}\label{C2.1'}
0<\lambda\sleq \frac{|\psi(x)-\psi(y)|}{|x-y|}\sleq \Lambda \qquad\mbox{ for all } x,y\in U
\end{equation}
and
\begin{equation}\label{C2.2'}
||D^2\psi||_\infty = \sup_{x\in U}||D_x^2\psi||<\infty.
\end{equation}
Let $\Gamma\subseteq \RR^n$, and $z\in \Gamma\cap U$, $B(z,r)\subseteq U$, $P$ be a plane through $z$, and set $\twid P = D_z\psi (P-z) + \psi(z)$, $\twid\Gamma = \psi(\Gamma)$. Then
\begin{equation}\label{C2.3'}
\hd{\psi(z),\lambda r}{\twid \Gamma}{\twid P}\sleq \frac{||D^2 \psi||_\infty}{2\lambda}r + \frac{\Lambda}{\lambda}\hd{z,r}{\Gamma}{P}.
\end{equation}

\end{lemmaC2}

\begin{lemmacoords}
For $a\in\cC\setminus\{0\}$, there exists a neighborhood $U\supseteq B(a, 2{|a|}/{A})$, $V\subseteq\RR^3$ open, $I$ an open interval with $0\in I$, and a smooth coordinate map $\psi^a: U\to V\times I$ such that $V\times \{0\} = \psi^a(\cC\cap U)$ and $\twid \pi = \psi^a\circ\pi\circ(\psi^a)\inv$ is orthogonal projection onto $\RR^3\times\{0\}$ (where $\pi$ is the same map of Section 5; see (\ref{nearestpoint})). Further, $\psi^a$ satisifes the estimates
\begin{equation}\label{coords.1'}
\frac12\sleq \frac{|\psi^a(x)-\psi^a(y)|}{|x-y|}\sleq 2\qquad \mbox{ for all }x,y\in U
\end{equation}
and
\begin{equation}\label{coords.2'}
||D^2\psi^a||_\infty = \sup_{x\in U} ||D^2_x\psi^a|| \sleq \frac C{|a|}
\end{equation}
for some $C$ independent of $a$.
\end{lemmacoords}

We continue the proof of Lemma \ref{bilip}. Let $x\in \Sigma \cap O$. Let $a = \pi(x)$, let $U$, $V$, $I$, $\psi^a$, and $\twid \pi$ as in Lemma \ref{coords}. Set $R =  {|x|^{1+\gamma}}/A$. Following the notation of Lemma \ref{C2}, let $\twid\Sigma = \psi^a(\Sigma\cap B(a, 2R))$. Following the notation of Lemma \ref{C2}, for  $z\in \Sigma\cap B(a,2R)$, $\twid z = \psi^a(z)\in \twid \Sigma$ and $0<r\sleq 2R$, we define $\twid P ( \twid z, r) = \twid{P(z,2r) } = D_z\psi^a(P(z, 2r)-z) + \twid z$. By Lemma \ref{C2} and (\ref{coords.1'}), 
\begin{equation}\label{bilip.6}
\hd{\twid z,r}{\twid \Sigma}{\twid P(\twid z, r)}\sleq ||D^2 \psi^a||_\infty\cdot 2r + 4 \hd{z, 2r}{\Sigma}{P(z,2r)}.
\end{equation}
for $0<r\sleq 2R$. By (\ref{coords.2'}) and (E2), (\ref{bilip.6}) gives
$\hd{\twid z,r}{\twid \Sigma}{\twid P(\twid z, r)}\sleq C |x|^{1+\gamma}/{|a|} + 4 \delta.$
Because $a = \pi(x)$, $|a|\sgeq |x|-|\pi(x)-x|\sgeq {|x|}/2$, and so
\begin{equation}\label{bilip.8}
\hd{\twid z,r}{\twid \Sigma}{\twid P(\twid z, r)}\sleq C|x|^{\gamma} + 4 \delta\sleq C |r_0|^{\gamma} + 4 \delta.
\end{equation}
Thus, we require that $r_0$ and $\delta$ be small enough that $C|r_0|^{\gamma}, 4 \delta \sleq 1/{32}.$
So (\ref{bilip.8}) gives us that
\begin{equation}\label{bilip.10}
\hd{\twid z,r}{\twid\Sigma}{\twid P(\twid z, r)}\sleq \frac 1{16}.
\end{equation}
Next we recall that Lemma \ref{coords} tells us that $V\times\{0\} = \psi^a(\cC\cap U)$. Similarly to before, we apply Lemmas \ref{ConeToPlane} and \ref{C2} to get that for $0<r\sleq 2R$, 
\begin{equation}\label{bilip.11}
\begin{split}
\hd{\twid z, R}{V\times\{0\}}{\twid P(\twid z, r)}&\sleq \frac{||D^2\psi^a||_\infty}{2\cdot1/2} + \frac{2}{1/2}\hd{z,r}\cC{P(z,r)} 
\sleq C|x|^{\gamma} + C|x|^{\min(\beta_1-\gamma, \beta_2(1+\gamma))}\\ &\sleq C|x|^{\beta_3}\sleq C r_0^{\beta_3}.
\end{split}
\end{equation}
Note that $e_4$ is the normal vector to the plane $\RR^3\times \{0\}$, and let $\twid\nu_{\twid z, r}$ be the normal vector (with positive 4th coordinate) to $\twid P(\twid z, r)$. Then Lemma \ref{planes} and (\ref{bilip.11}) guarantees that there is an $r_0$ small enough such that
\begin{equation}\label{bilip.12}
|e_4-\twid\nu_{\twid z,r}|\sleq \frac 1{16}.
\end{equation}

Let $y\in \Sigma\cap O$, $|x-y|\sleq \frac{|x|^{1+\gamma}}A$. Let $\twid x = \psi^a(x)$ and $\twid y = \psi^a(y)$. In particular, we note that by (\ref{coords.1'}),  $\rho := |\twid x - \twid y|\sleq2 |y-z|\sleq 2 R$, and so $\twid P(\twid x, \rho)$ is defined. Because $\twid \pi$ is orthogonal projection onto $\RR^3\times\{0\} = \langle e_4\rangle^\perp$, we get
\begin{equation}\label{bilip.13}
|\twid x - \twid y|^2 = |\twid\pi(\twid x) - \twid \pi(\twid y)|^2 + |\langle \twid x - \twid y, e_4\rangle|^2.
\end{equation}
We compute by (\ref{bilip.10}) and (\ref{bilip.12}) that
\begin{equation}\label{bilip.14}
|\langle \twid x - \twid y, e_4\rangle|\sleq |\langle \twid x - \twid y, \twid \nu_{\twid x, \rho}\rangle| +
 |\langle \twid x - \twid y, e_4-\twid \nu_{\twid x, \rho}\rangle|\sleq \frac 1{16} \rho+ \frac 1{16}|\twid x - \twid y| = \frac 18|\twid x - \twid y|.
\end{equation}
We apply (\ref{bilip.14}) to (\ref{bilip.13}) to get
$
|\twid\pi(\twid x) - \twid \pi(\twid y)|^2  = |\twid x - \twid y|^2 - |\langle \twid x - \twid y, e_4\rangle|^2\sgeq \frac{63}{64} |\twid x - \twid y|^2\sgeq \frac{63}{64\cdot 4}|x-y|^2.$
Hence, $\pi|_{\Sigma\cap O}$ is lower Lipschitz.

\end{proof}

Hence, we may define $\vphi:\cC\cap O\to \Sigma\cap O$ by $\vphi = \pi_{\Sigma\cap O}\inv$. Because $\pi_{\Sigma\cap O}$ is lower Lipschitz, $\vphi$ is upper Lipschitz. We use $\vphi$ as our parametrization of $\Sigma$ in a neighborhood of 0. We now begin the process of extending $\vphi$ to a $C^{1,\beta}$ map on a neighborhood of $0$. We will employ a modification of the Whitney Extension Theorem, which we state here explicitly for the reader's convenience. For $f\in C^k(\RR^m\to\RR^\ell)$, we let $D^k_a f$ be the $k$-linear map from $(\RR^m)^k$ of partial derivatives at $a$, where by convention we take $D^0_a f = f(a)$.

\begin{theorem}\label{Whitney}\emph{($C^{k,\beta}$ Whitney Extension Theorem)}\\ Let $\beta>0$, $k, \l, m\in \NN$, $A\subseteq \RR^m$ be closed, and for each $a\in A$ a polynomial $P_a:\RR^m\to\RR^{\l}$ such that $\deg P_a\sleq k.$
Define for $K\subseteq A$, $r>0$, $0\sleq i\sleq k$,
\leqn{rho_i}
\rho_i(K,r) = \sup\left\{\frac{||D^i_b P_b-D_b^i P_a||}{|a-b|^{k-i}} : a,b\in A, |a-b|\sleq r\right\}.
\endeqn
If for each compact $K\subseteq A$ and each $0\sleq i\sleq k$
\begin{equation}\label{Whitney1} \rho_i(K,r)\sleq Cr^\beta\end{equation}
then there exists $\vphi \in C^{k,\beta}_{loc}(\RR^m\to\RR^{\l})$ such that for all $a\in A$ and $0\sleq i \sleq k$,
$D^i \vphi(a) = D^i P_a(a).$
\end{theorem}

We first say a few words on this theorem. Although the extension in the theorem as stated is $C^{k,\beta}_{loc}(\RR^m\to\RR^\ell)$, we will only be interested in a $C^{1,\beta}$ extension of $\vphi$ to a neighborhood of 0. The theorem is presented this way  to be consistent with Federer.

We define the polynomials we will use for our analysis.  We first set some notation. For $a\in\cC\setminus\{0\}$, let  $r_a = a/|a|$ be the unit radial tangent vector. Let $\nu_a$ be the inward pointing unit normal vector to $\cC$ at $a$ as defined in (\ref{nu_a}). A vector $\theta_a$ at $a$ orthogonal to both $r_a$ and $\nu_a$ will be said to be of type $\theta$. The motivation is that a vector of type $\theta$ is tangent to the cross section $\cC_{a_4}$.  For $x\in \Sigma\cap O\setminus{0}$ and $0<r\sleq {|x|^{1+\gamma}}/A$, let $L(x,r) = P(x,r)-x$ and $L(x) = P(x) - x$ be the approximating planes recentered to be vector spaces. For $a\in \cC\cap 0\setminus 0$, let $\lambda_a$ be the unit normal vector to $L(\vphi(a))$. For $a,b\in\cC$, let $\phi_a$ be projection in the direction of $\eta_a$ onto $L(\vphi(a))$, and $\phi_{a,b}$ be projection in the direction of $\eta_a$ onto $L(\vphi(b))$. Note that $\phi_a = \phi_{a,a}$.  Recall that $\tau(x) = (x_1,x_2,x_3)$.

We define $M_a$ by
\begin{equation}\label{Mdef}
\begin{split}
M_a(r_a)  =\phi_a(r_a),\,\,\,
M_a(\nu_a) =\nu_a,\,\,\,
M_a(\theta_a) = \phi_a(\theta_a)\frac{|\tau(\vphi(a))|}{|\tau(a)|}\, ,
\end{split}
\end{equation}
where $\theta_a$ is any vector of type $\theta$. We also set
$R(a) = \frac{|\tau(\vphi(a))|}{|\tau(a)|},$
which yields the slightly cleaner expression
$M_a(\theta_a) = \phi_a(\theta_a) R(a)$
for any vector $\theta_a$ of type $\theta$.
Define also $M_0 = \Id$. We then define, following the terminology of Theorem \ref{Whitney}, a polynomial for each $a\in \Sigma\cap B(0,r_0)$,
$P_a(x) = \vphi(a) + M_a(x-a).$
Note that $P_a(a) = \vphi(a)$.

\begin{lemma}\label{vphidist}
Recall hypotheses (E0)-(E2).
$|\vphi(a)-a| = \sec(\pi/4)\, d(\vphi(a),\cC)\sleq |a|^{1+\beta_1}.$
\end{lemma}

\begin{proof}
Follows immediately from Lemmas \ref{pidistnaught} and \ref{pidist}.
\end{proof}

\begin{lemma}\label{vectors}
For $a,b\in \cC\cap O\setminus 0$ such that $ |a-b|\sleq \frac{|a|^{1+\gamma}}{A}$,
\begin{enumerate}[(1)]
\item a vector $v_a$ is of type $\theta$ if and only if $(v_a)_4 = 0$
\item $|r_a - r_b|\sleq C|a-b|\frexp{\gamma}{1+\gamma}$
\item $|\nu_a - \nu_b|\sleq C|a-b|\frexp{\gamma}{1+\gamma}$
\item $|\eta_a - \eta_b|\sleq C|a-b|\frexp{\gamma}{1+\gamma}$
\item If $\theta_a$ is a unity vector of type $\theta$ at $a$, then there exists a vector $\theta_b$
of type $\theta$ at $b$  such that $|\theta_a - \theta_b|\sleq C|a-b|\frexp{\gamma}{1+\gamma}$.
\end{enumerate}
\end{lemma}

\begin{proof}
We first prove (1). Let $u_a\in T_a\cC$. Then $u_a\perp \nu_a$. Note that $r_a + \nu_a = 2 (a_4)/|a|\neq 0$. We then have that
$
u_a \mbox{ is of type }\theta\Leftrightarrow u_a\cdot r_a = 0 \Leftrightarrow u_a\cdot ( r_a + \nu_a) = 0 \Leftrightarrow (u_a)_4 = 0.
$

We now prove (2)-(4). Recall that $r_a = a/|a|$. We observe the identity that for $w,x,y,z\in \RR$,
\begin{equation}\label{vectors.1}
xy - zw = \frac12 \left((x-z)(y+w) + (y-w)(x+z)\right).
\end{equation}
Hence, we get that
\begin{equation}\label{vectors.2}
\begin{split}
\left| \frac{a}{|a|} - \frac{b}{|b|} \right| &= \left| \frac{a|b| - |a| b}{|a||b|}\right| \sleq \frac 12\left(\frac{|a-b|(|a| + |b|) - \big| |a|-|b| \big|\cdot |a + b|}{|a||b|}\right)\sleq \frac{|a-b|(|a| + |b|)}{2|a||b|}.
\end{split}
\endeqn
Recall that $|a-b|\sleq{ |a|^{1+\gamma}}/{A}$. This gives that $|b|\sleq |b-a| + |a|\sleq |a|({1+A})/{A}$, and similarly $|a|\sleq |b|(1+A)/{A}$. Hence, $|b|$ and $|a|$ are within a constant factor of each other and so (\ref{vectors.2}) gives that
\leqn{vectors.3}
\left| \frac a{|a|} - \frac b{|b|}\right| \sleq C\frac {|a-b|}{|a|}.
\endeqn
Applying that $|a-b|\sleq {|a|^{1+\gamma}}/{A}$, we get that $1/{|a|}\sleq 1/({A|a-b|\frexp1{1+\gamma}})$. Hence, we get from (\ref{vectors.3}) that
\leqn{vectors.4}
\left| \frac a{|a|} - \frac b{|b|}\right| \sleq C\frac {|a-b|}{|a-b|\frexp{1}{1+\gamma}}= C|a-b|\frexp{\gamma}{1+\gamma}.
\endeqn
From (\ref{vectors.4}), claims (2), (3), and (4) follow.

We prove (5). Let $\theta_a$ be a unit vector of type $\theta$ in $T_a\cC-a$. Let $\pi_b$ be projection onto $T_b\cC-b$ in the $\eta_b$ direction. Define $\theta_b = \pi_b(\theta_a)$. By (1), $(\theta_a)_4 = 0$. Because $\pi_b$ is projection onto $T_b\cC-b$ in the $\eta_b$ direction, we get that $(\theta_b)_4 = 0$ and $\theta_b\in T_b\cC-b$. So by (1), $\theta_b$ is a vector of type $\theta$ at $b$. Further, we compute that
\leqn{vectors.5}
|\theta_a-\theta_b| = \sec(\meang(\nu_b, \eta_b)) d(\theta_a, T_b \cC-b)\sleq C\meang(T_a\cC, T_b\cC).
\endeqn
By (2) and (\ref{distlessthanangle}), we see that
$|\theta_a-\theta_b| \sleq C |a-b|\frexp\gamma{1+\gamma}.$

\end{proof}

\begin{lemma}\label{phi}
Recall hypotheses (E0)-(E2).
For $a,b\in\cC$ with $|a-b|\sleq |a|^{1+\gamma}/A$,
\begin{enumerate}[(1)]
\item $||\phi_{a,a}-\phi_{b,a}||\sleq C|a-b|\frexp{\gamma}{1+\gamma}$
\item $||\phi_{a,a}-\phi_{a,b}||\sleq C|a-b|^{\beta_2}$
\item $||\phi_a-\phi_b||\sleq C |a-b|^{\min(\beta_2, \frac{\gamma}{1+\gamma})}.$
\item $|R(a) - R(b)| \sleq C|a-b|\frexp{\gamma}{1+\gamma}$.
\end{enumerate}
\end{lemma}

\begin{proof}

First, we note that (1) follows from Lemma \ref{vectors}(2) and the definition of $\phi_{c,d}$\,.

We now prove (2). Let $v\in \RR^4$, $|v|=1$. We note because each of $\phi_{a,a}$ and $\phi_{a,b}$ is a projection in the $\eta_a$ direction, $\phi_{a,a}(v)-\phi_{a,b}(v)$ is a scalar multiple of $\eta_a$. Thus, we have that
\begin{equation}\label{phi.2}
||\phi_{a,a}(v)-\phi_{a,b}(v)|| = \sec(\meang(\eta_a, \lambda_b)) d(\phi_{a,a}(v), L(\vphi(b))),
\end{equation}
where we recall that $\lambda_b$ is the normal vector to $L(\vphi(b))$ (see Figure 2). 

\begin{figure}
\begin{center}
\includegraphics[width=300pt]{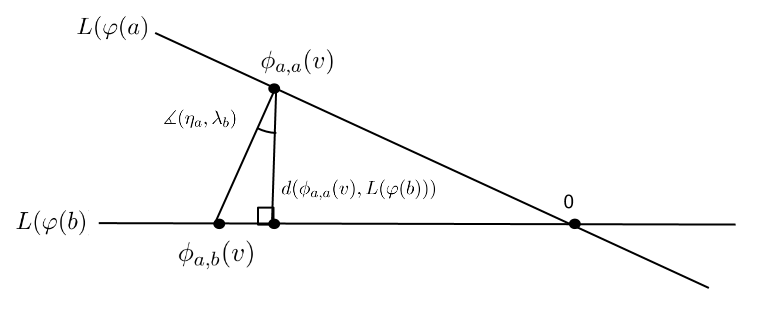}
\caption{Diagram of (\ref{phi.2})}
\end{center}
\end{figure}

Next, we note that
\begin{equation}\label{phi.3}
\meang(\eta_a,\lambda_b)\sleq \meang(\eta_a,\eta_b)+\meang(\eta_b,\nu_b) + \meang(\nu_b, \lambda_b)\sleq \frac\pi 3
\end{equation}
 for $r_0$ small enough by Lemma \ref{vectors}(2), $\meang(\eta_b,\nu_b) = \pi/4$, and Lemma \ref{tangentangle}. Thus, $\sec(\meang(\eta_a, \lambda_b))\sleq 2$. In addition, for $r_0$ small enough we have that $\meang(\nu_a, \lambda_a)\sleq5\pi/12$ (see (\ref{phi.3}) and Lemma \ref{tangentangle}), so $|\phi_{a,a}(v)|\sleq R$ for some $R$ (independent of $a$ and $b$). Hence, (\ref{phi.2}) and (\ref{phi.3}) gives us that
\begin{equation}\label{phi.4}
||\phi_{a,a}(v)-\phi_{a,b}(v)|| \sleq C \D{0,R}{L(\vphi(a))}{ L(\vphi(b))} \sleq C\D{0,1}{L(\vphi(a))}{ L(\vphi(b))}.
\end{equation}
Because $\vphi$ is Lipschitz, (\ref{distlessthanangle}), Corollary \ref{angle} and (\ref{phi.4}) tell us that
\begin{equation}\label{phi.5}
||\phi_{a,a}(v)-\phi_{a,b}(v)|| \sleq C |\vphi(a)-\vphi(b)|^{\beta_2}\sleq C|a-b|^{\beta_2}.
\end{equation}

To prove (3), we apply (1) and (2) to get
\begin{equation}
||\phi_a - \phi_b||  \sleq ||\phi_{a,a}-\phi_{a,b}|| + ||\phi_{a,b}-\phi_{b,b}||\sleq C|a-b|^{\min(\beta_2,\frac\gamma{1+\gamma})}.
\end{equation}

We now prove (4). Applying the definition of $R$, we get
\begin{equation}\label{phi.6}
|R(a) - R(b)| = \left|\frac{|\tau(\vphi(a))|}{|\tau(a)|}-\frac{|\tau(\vphi(b))|}{|\tau(b)|}\right| = 
\left|\frac{|\tau(\vphi(a))|\,|\tau(b)| - |\tau(\vphi(b))|\,|\tau(a)|}{|\tau(a)||\tau(b)|}\right|.
\end{equation}
We reobserve the identity that for $w,x,y,z\in \RR$,
\begin{equation}\label{phi.7}
xy - zw = \frac12 \left((x-z)(y+w) + (y-w)(x+z)\right).
\end{equation}
Applying (\ref{phi.7}) to (\ref{phi.6}) gives
\begin{equation}\label{phi.8}\begin{split}
|R(a)-R(b)| \sleq \frac1{2|\tau(a)||\tau(b)|}
&\Big(
\big| |\tau(\vphi(a))| - |\tau(\vphi(b))|\big|\big(|\tau(a)| + |\tau(b)|\big)
\\
\empty&+\big| |\tau(a)| - |\tau(b)| \big| \big(|\tau(\vphi(a))| + |\tau(\vphi(b))|\big)
\Big).\end{split}
\end{equation}
We note that because $|\tau(a)| = |a|/{\sqrt2}$ because $a\in\cC$, and the same holds for $b$. We then apply that $\vphi$ is Lipschitz (as well as $\vphi(0) = 0$), and the fact that $|\tau(x) - \tau(y)|\sleq |x-y|$ to (\ref{phi.8}) get that
\begin{equation}\label{phi.9}
|R(a)-R(b)|\sleq \frac C{|a||b|}\Big(|a-b| (|a|+|b|) \Big)
\end{equation}
Recall that we showed that $|a|$ and $|b|$ are within a constant factor of eachother (while proving (\ref{vectors.3}), so (\ref{phi.9}) gives
\begin{equation}\label{phi.10}
|R(a)-R(b)|\sleq C \frac{|a-b|}{|a|}\sleq C \frac{|a-b|}{|a-b|\frexp1{1+\gamma}} = C|a-b|\frexp\gamma{1+\gamma}.
\end{equation}

\end{proof}

We note that the argument used in (\ref{vectors.5}) and (\ref{phi.2})  will recur in many variations in the proofs to come. The reader who glazed over this point is encouraged to review it.

\begin{lemma}\label{Mnorm1}
Recall hypotheses (E0)-(E2). For $a\in C\cap O$, $||M_a - M_0|| \sleq C|a|^{\beta_3}$.
\end{lemma}

\begin{proof}
Note that at $a$ we can choose an orthonormal basis $\cB_a=\{r_a,\nu_a,\theta^1_a,\theta^2_a\}$ where $r_a$ and $\nu_a$ are the radial and normal vectors as before, and $\theta^i_a$ is a vector of type $\theta$. We thus check that each vector $v\in\cB_a$ will satisfy the bound
\begin{equation}\label{Mnorm1.1}
|M_a(v)-M_0(v)|\sleq C|a|^{\beta_3},
\end{equation}
from which the result follows. Recall that by definition $M_0 = \Id$.

First, we consider the easiest case, $\nu_a$. By definition $M_a(\nu_a) = \nu_a$, and so $M_a(\nu_a)-\nu_a = 0$.

 Next, consider $M_a(r_a)-r_a$. By definition, $M_a(r_a) = \phi_a(r_a)$. Because $\phi_a$ is a projection in the $\eta_a$ direction, we get that 
\begin{equation}\label{Mnorm1.2}
|\phi_a(r_a)-r_a| = \sec(\meang(\eta_a, \lambda_a)) d(r_a, L(\vphi(a)))\sleq C \D{0,1}{T_a\cC-a} {L(\vphi(a))}\sleq C|a|^{\beta_3}.
\end{equation}

Finally, we consider $\theta_a$ a vector of type $\theta$. Recall that $M_a(\theta_a) = \frac{|\tau(\vphi(a))|}{|\tau(a)|}\phi_a(\theta_a)$. We compute
\begin{equation}\label{Mnorm1.3}
\begin{split}
|M_a(\theta_a)-\theta_a|& =\left|\frac{|\tau(\vphi(a))|}{|\tau(a)|}\phi_a(\theta_a) - \theta_a\right|\sleq \left| \frac{|\tau(\vphi(a))|}{|\tau(a)|}-1\right||\phi_a(\theta_a)| + |\phi_a(\theta_a)-\theta_a|.
\end{split}
\end{equation}
From Lemma \ref{vphidist} we get that
\begin{equation}\label{Mnorm1.4}
\left| \frac{|\tau(\vphi(a))|}{|\tau(a)|}-1\right| = \frac{\big||\tau(\vphi(a))|-|\tau(a)|\big|}{|\tau(a)|}\sleq C\frac{|\vphi(a)-a|}{|a|}\sleq C{|a|}^{\beta_1}.
\end{equation}
Applying (\ref{Mnorm1.4}) to (\ref{Mnorm1.3}), we get
\begin{equation}
\begin{split}
|M_a(\theta_a)-\theta_a|&\sleq C|a|^{\beta_1} + \sec(\meang(\eta_a,\lambda_a))d(\theta_a,L(\vphi(a)))\\& \sleq C|a|^{\beta_1} +  C\D{0,1}{T_a\cC - a}{ L(\vphi(a))}\sleq C|a|^{\beta_3}.
\end{split}
\end{equation}

\end{proof}

\begin{lemma}\label{Mnorm2}
Recall hypotheses (E0)-(E2).
For $a,b\in \cC\cap O$, we have that $||M_a - M_b|| \sleq C|a-b|\frexp{\beta_3}{1+\gamma}.$
Or in the language of Theorem \ref{Whitney}, $\rho_1(\cC\cap O, r)\sleq C r\frexp{\beta_3}{1+\gamma}$
(recall (\ref{rho_i}), and note that $\cC\cap O$ is compact).
\end{lemma}

\begin{proof}
Recall that $\beta_3 = \min(\beta_1-\gamma, \beta_2(1+\gamma), \gamma)$. We break the proof into two scales. First, we assume that $|a-b|\sgeq \max(|a|, |b|)^{1+\gamma}/A$. Then Lemma \ref{Mnorm1} tells us that
\begin{equation}\label{Mnorm2.1}
||M_a-M_b|| \sleq ||M_a -M_0|| + ||M_b-M_0|| \sleq C\left(|a|^{\beta_3} + |b|^{\beta_3}\right)\sleq C|a-b|\frexp{\beta_3}{1+\gamma}.
\end{equation}

Let us now assume that $|a-b|\sleq |a|^{1+\gamma}/A$. As in Lemma \ref{Mnorm1}, we will consider the radial, normal, and type $\theta$ vectors separately. By Lemma \ref{vectors}(3), we have that $|\nu_a-\nu_b|\sleq  C|a-b|\frexp{\gamma}{1+\gamma}.$
We recall that $M_a(\nu_a) = \nu_a$. Thus, we compute
\begin{equation}
\begin{split}
|M_a(\nu_a) - M_b(\nu_a)|&\sleq |M_a(\nu_a)-M_b(\nu_b)| + |M_b(\nu_b-\nu_a)|\sleq |\nu_a-\nu_b| + ||M_b||\cdot|\nu_a-\nu_b|\\
&\sleq C|a-b|\frexp\gamma{1+\gamma}
\end{split}
\end{equation}
(note  $||M_b||\sleq ||M_b - M_0|| + 1\sleq  C$ by Lemma \ref{Mnorm1}).

We now show that
\begin{equation}\label{Mnorm2.2}
|M_a(r_a)-M_b(r_a)|\sleq C|a-b|\frexp{\beta_3}{1+\gamma}.
\end{equation}
 By Lemma \ref{vectors}(2), we have that $|r_a-r_b|\sleq C|a-b|\frexp\gamma{1+\gamma}.$
We thus compute
\begin{equation}\label{Mnorm2.4}
\begin{split}
|M_a(r_a) - M_b(r_a)|&\sleq |M_a(r_a) - M_b(r_b)| + |M_b(r_b-r_a)|\sleq |\phi_a(r_a)-\phi_b(r_b)| + ||M_b|| |r_b-r_a|\\
&\sleq |\phi_a(r_a) - \phi_b(r_b)| + C|b-a|\frexp\gamma{1+\gamma}.
\end{split}
\end{equation}
We now consider $|\phi_a(r_a) - \phi_b(r_b)|$. Applying  By Lemma \ref{vectors}(2) and Lemma \ref{phi}, we get
\begin{equation}\label{Mnorm2.5}\begin{split}
|\phi_a(r_a)-\phi_b(r_b)|&\sleq |\phi_a(r_a-r_b)| + |(\phi_a-\phi_b)(r_b)|\sleq ||\phi_a||\cdot|r_a-r_b| + ||\phi_a-\phi_b|| \cdot|r_b|\\
&\sleq C|a-b|^{\min(\beta_2, \frac{\gamma}{1+\gamma})}.\end{split}
\end{equation}
Coupling (\ref{Mnorm2.4}) and (\ref{Mnorm2.5}), we prove (\ref{Mnorm2.2}).

Finally, we consider vectors of type $\theta$. Let $\theta_a$ be a vector of type $\theta$ at $a$. By Lemma \ref{vectors}(5) there is a vector $\theta_b$ of type $\theta$ at $b$ such that
$|\theta_a-\theta_b|\sleq  C|a-b|\frexp\gamma{1+\gamma}.$
We now show that
\begin{equation}\label{Mnorm2.7}
|M_a(\theta_a)-M_b(\theta_a)|\sleq C|a-b|\frexp{\beta_3}{1+\gamma}.
\end{equation}
We compute that
\begin{equation}\label{Mnorm2.8}
\begin{split}
|M_a(\theta_a) - M_b(\theta_a)|&\sleq |M_a(\theta_a) - M_b(\theta_b)| + |M_b(\theta_b-\theta_a)|\sleq |M_a(\theta_a)-M_b(\theta_b)| + ||M_b||\cdot |\theta_b-\theta_a|\\
&\sleq \left|R(a) \phi_a(\theta_a) -R(b) \phi_b(\theta_b) \right| + C|b-a|\frexp\gamma{1+\gamma}.
\end{split}
\end{equation}
We now consider 
$\left|R(a) \phi_a(\theta_a) -R(b) \phi_b(\theta_b) \right|.$
We reobserve the identity
\begin{equation}\label{Mnorm2.9}
xy - zw = \frac12 \left((x-z)(y+w) + (y-w)(x+z)\right).
\end{equation}
Applying (\ref{Mnorm2.9}) to $\left|R(a) \phi_a(\theta_a) -R(b) \phi_b(\theta_b) \right|$ gives that
\begin{equation}\label{Mnorm2.10}
|R(a)\phi_a(\theta_a) - R(b)\phi_b(\theta_b)|\sleq |R(a)-R(b)|(|\phi_a(\theta_a)| + |\phi_b(\theta_b)|) + |\phi_a(\theta_a) - \phi_b(\theta_b)| (R(a) + R(b)).
\end{equation}
We then apply that $|R(a)|$, $|R(b)|$, $||\phi_a||$ and $||\phi_b||$ are all bounded, plus Lemma \ref{phi}(4) to (\ref{Mnorm2.10}) to get
\begin{equation}\label{Mnorm2.11}
|R(a)\phi_a(\theta_a) - R(b)\phi_b(\theta_b)|\sleq C |\phi_a(\theta_a) - \phi_b(\theta_b)| + C|a-b|\frexp{\gamma}{1+\gamma}.
\end{equation}
Thus to establish (\ref{Mnorm2.7}) (and finish the proof), we establish
\begin{equation}\label{Mnorm2.12}
|\phi_a(\theta_a) - \phi_b(\theta_b)|\sleq C |a-b|\frexp{\beta_3}{1+\gamma}.
\end{equation}
Applying $|\theta_a-\theta_b|\sleq  C|a-b|\frexp\gamma{1+\gamma}$
 and Lemma \ref{phi}, we compute that
\begin{equation}\label{Mnorm2.13}\begin{split}
|\phi_a(\theta_a)-\phi_b(\theta_b)|&\sleq |\phi_a(\theta_a-\theta_b)|+ |(\phi_a-\phi_b)(\theta_b)|\sleq ||\phi_a||\cdot|\theta_a-\theta_b| + ||\phi_a-\phi_b||\cdot |\theta_b|\\
&\sleq C |a-b|^{\min(\beta_2, \frac{\gamma}{1+\gamma})}\end{split}
\end{equation}
and conclude that (\ref{Mnorm2.7}) holds.
\end{proof}

\begin{lemma}\label{poly}
Recall hypotheses (E0)-(E2).
For $a,b\in\cC \cap O$, we have 
\leqn{poly.1}
\frac{|P_b(b)-P_a(b)|}{|b-a|}\sleq C|b-a|\frexp{\beta_3}{1+\gamma}.
\endeqn
Or, in the notation of Theorem \ref{Whitney}, we have that
$\rho_0(\cC\cap O, r)\sleq C r\frexp{\beta_3}{1+\gamma}$
(recall (\ref{rho_i}), and note that $\cC\cap O$ is compact).
\end{lemma}

\begin{proof}
Applying the definition of the polynomials $P_a$ and $P_b$, we get
\begin{equation}\label{poly.2}
P_b(b)-P_a(b) = \vphi(b) + M_b(b-b) - \vphi(a)-M_a(b-a) = \vphi(b)-\vphi(a)-M_a(b-a).
\end{equation}
To begin, we consider the case where $|a-b|\sgeq \max(|a|,|b|)^{1+\gamma}/A$. We compute that
\begin{equation}\label{poly.3}
\begin{split}
|\vphi(b)-\vphi(a)-M_a(b-a)|&\sleq |b-a-M_a(b-a)| + |\vphi(b)-b| + |\vphi(a)-a|.
\end{split}
\end{equation}
By Lemmas \ref{vphidist} and \ref{Mnorm1}, we get that
\begin{equation}\label{poly.4}
\begin{split}
|b-a-M_a(b-a)| + |\vphi(b)-b| + |\vphi(a) - a|& \sleq||M_a-\Id||\cdot|b-a| + C|b|^{1+\beta_1} + C|a|^{1+\beta_1}\\
&\sleq |a|^{\beta_3}|b-a| + C|a-b|\frexp{1+\beta_1}{1+\gamma}\sleq C|b-a|\frexp{\beta_3}{1+\gamma}.
\end{split}
\end{equation}

Assume now that $a,b\in\cC\cap O$ and  $|a-b|\sleq |a|^{1+\gamma}/A$. First, we define some ways that $a$ and $b$ may differ from each other. Let $\pi_a$ be projection in the $\eta_a$ direction onto $T_a\cC-a$.
\begin{itemize}
  \item[(i)] We say that $a$ and $b$ are \emph{radially separated} if $b-a$ is a radial vector.
  \item[(ii)] We say that $a$ and $b$ are \emph{$\theta$ separated} if $\pi_a(b-a)$ is a vector of type $\theta$.
\end{itemize}
\begin{figure}[htb]
\begin{center}
\includegraphics[width=150pt]{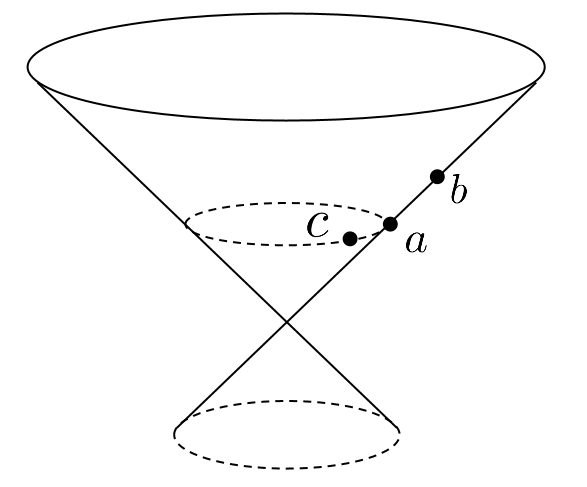}
\caption{$a$ and $b$ are radially separated, $a$ and $c$ are $\theta$ separated}
\end{center}
\end{figure}
Note that if $b-a = c r_a,$ then $\pi_a(b-a) = b-a$, and so conditions (i) and (ii) have more symmetry than may initially appear. We develop an alternative characterization of condition (ii), that
\begin{equation}\label{poly.5}
\mbox{$a$ and $b$ are $\theta$ separated}\quad\Leftrightarrow \quad a_4 = b_4.
\end{equation}
To see this, we first note that for all vectors $v\in T_a\cC-a$, that $v\cdot\nu_a = 0$. Also, $\nu_a + r_a = 2{a_4}/{|a|}\neq 0$ because $a\in \cC\setminus \{0\}$. Also, we note that because $\pi_a$ is projection in the $\eta_a$ direction and $(\eta_a)_4 = 0$, $\pi_a$ does not change the 4th coordinate. Thus, we have that
\begin{equation}\begin{split}
a_4 = b_4&\Leftrightarrow \pi_a(b-a)_4 = 0\Leftrightarrow \pi_a(b-a)\cdot (\nu_a + r_a) = 0 \Leftrightarrow
\pi_a(b-a)\cdot r_a = 0\\
 &\Leftrightarrow\,\mbox{$a$ and $b$ are $\theta$ separated},
\end{split}\end{equation}
establishing (\ref{poly.5}).

First, assume that $a$ and $b$ are radially separated. Then because $b-a$ is a mulitple of $r_a$, we apply the definition of $M_a$ to get that
\begin{equation}\label{poly.6}
P_b(b)-P_a(b) = \vphi(b) - \vphi(a) - M_a(b-a) = \vphi(b)-\vphi(a) - \phi_a(b-a).
\end{equation}
Note that when $a$ and $b$ are radially separated, $a/|a| = b/|b|$. So $r_a = r_b$  and $\eta_a = \eta_b$. We now use this to claim that
\begin{equation}\label{poly.7}
\vphi(a) + \phi_a(b-a)\in \ell_b.
\end{equation}
To see this, we expand
\begin{equation}\label{poly.8}
\vphi(a) + \phi_a(b-a) = \big(\vphi(a)-a\big) + \big(\phi_a(b-a) - (b-a)\big) + b.
\end{equation}
Because $\vphi$ is the inverse of $\pi$ which was projection in the $\eta_a$ direction, $\vphi(a)-a$ is a multiple of $\eta_a$, which in this case satisfies $\eta_a= \eta_b$. Because $\phi_a$ is a projection in the $\eta_a$ direction, we have that $\phi_a(b-a)-(b-a)$ is a scalar multiple of $\eta_a$, and $\eta_a = \eta_b$. Thus, from (\ref{poly.8}), we have that for some $s\in \RR$, 
\begin{equation}\label{poly.9}
\vphi(a) + \phi_a(b-a) = b + s\eta_b\in \ell_b
\end{equation}
(recall (\ref{ell_a})). Because $\vphi(b) \in \ell_b$, we have that $\vphi(b) - \vphi(a) - \phi_a(b-a)$ is a scalar multiple of $\eta_b$. Recall that $\lambda_a$ is the normal vector to $P(\vphi(a))$. Thus, because $\vphi(a) + \phi_a(b-a)\in P(\vphi(a))$ (because $\phi_a$ is projection into $L(\vphi(a))$), we have that
\begin{equation}\label{poly.10}
|\vphi(b) - \vphi(a) - \phi_a(b-a)| = \sec(\meang(\eta_b,\lambda_a)) d(\vphi(b), P(\vphi(a))\sleq C |b-a|^{1+\beta_2}
\end{equation}
by Corollary \ref{angle} and $\vphi$ being Lipschitz (Lemma \ref{bilip}). Thus, we have established (a stronger inequality than) (\ref{poly.1}) for $a$ and $b$ radially separated.

\begin{figure}
\begin{center}
\includegraphics[width = 250pt]{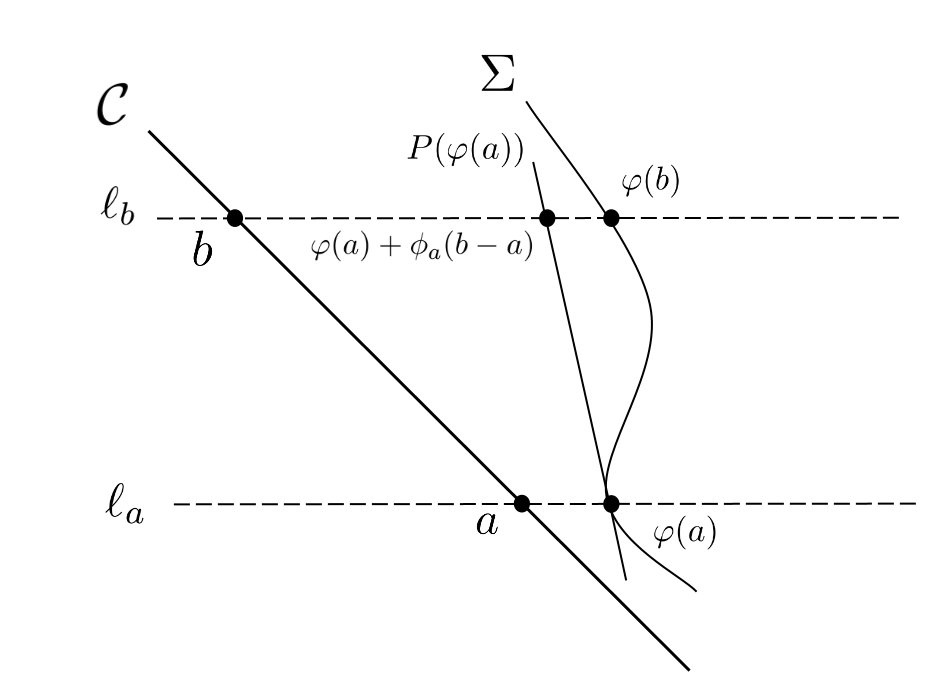}
\caption{The argument for radially separated points}
\end{center}
\end{figure}

Next, we assume that $a$ and $b$ are $\theta$ separated. In this case, many of the simplifications that we were able to make in the radial case will not hold true, but will be true up to $O\big(|b-a|^{1+\frac{\beta_3}{1+\gamma}}\big)$. Thus, while the core ideas remain the same, several more estimates must be applied. We begin by applying the assumption that $a$ and $b$ are $\theta$ separated to get
\begin{equation}\label{poly.11}
\begin{split}
|\vphi(b)-\vphi(a) -M_a(b-&a)|\sleq |\vphi(b)-\vphi(a) - M_a(\pi_a(b-a))| + |M_a(\pi_a(b-a)) - M_a(b-a)|\\
\sleq&|\vphi(b) - \vphi(a) - R(a)\phi_a(\pi_a(b-a))| + ||M_a||\cdot |\pi_a(b-a) - (b-a)|.
\end{split}
\end{equation}
By Lemma \ref{KPflat1}(2) and that $\pi_a$ is projection onto $T_a\cC$ in the $\eta_a$ direction, we have that
\begin{equation}\label{poly.12}
|\pi_a(b-a)-(b-a)|\sleq \sec(\meang(\eta_a,\nu_a)) d(b, T_a\cC)\sleq C\frac {|b-a|^2}{|a|}\sleq
 C |b-a|^{1+\frac{\gamma}{1+\gamma}}.
\end{equation}
Thus applying (\ref{poly.12}) to (\ref{poly.11}), we get that
\begin{equation}\label{poly.13}
|\vphi(b)-\vphi(a)-M_a(b-a)|\sleq |\vphi(b)-\vphi(a) - R(a)\phi_a(\pi_a(b-a))| + C|b-a|^{1+\frac\gamma{1+\gamma}}.
\end{equation}
Thus, we strive to establish
\begin{equation}\label{poly.14}
|\vphi(b)-\vphi(a)-R(a)\phi_a(\pi_a(b-a))|\sleq C |b-a|^{1+\frac{\beta_3}{1+\gamma}},
\end{equation}
which by (\ref{poly.13}) will establish (\ref{poly.1}) for $a$ and $b$ which are $\theta$ separated.

To begin proving (\ref{poly.14}), we note that because both $\phi_a$ and $\pi_a$ are projections in the $\eta_a$ direction, that $\phi_a(\pi_a(b-a)) = \phi_a(b-a)$. We use this and the linearity of $\phi_a$ to get
\begin{equation}\label{poly.15}
\vphi(b)-\vphi(a) - R(a)\phi_a(\pi_a(b-a)) = \vphi(b)-\vphi(a) - \phi_a(R(a)(b-a)).
\end{equation}
Next, we apply Lemma \ref{phi}(1) to get
\begin{equation}\label{poly.16}
\begin{split}
|\vphi(b) - \vphi(a) - \phi_a(R(a)(b-a))|&\sleq |\vphi(b)-\vphi(a) - \phi_{b,a}(R(a)(b-a))| + |(\phi_{b,a}-\phi_a)(R(a)(b-a))|\\
&\sleq |\vphi(b)-\vphi(a)-\phi_{b,a}(R(a)(b-a))| + C|b-a|^{1+\frac{\gamma}{1+\gamma}}.
\end{split}
\end{equation}

Similarly to the radial case, we now claim that
\begin{equation}\label{poly.17}
\vphi(a) + \phi_{b,a}(R(a)(b-a))\in \ell_b.
\end{equation}
Recall that by (\ref{poly.5}), $a_4  = b_4$. Note that $a + |\tau(a)|\eta_a = (0,a_4)$, where $0$ here represents $0\in \RR^3$. Thus, we have that $\ell_a\ni (0,a_4) = (0,b_4)\in \ell_b$. For $x,y,z\in \RR^4$ not colinear, let $\Delta x,y,z$ be the triangle with corners $x,y,$ and $z$. Consider $\Delta (0,a_4),a,b$. Let $z$ be the point so that $\Delta(0,a_4),\vphi(a),z$ is similar to $\Delta(0,a_4),a,b$ (see Fig. 2). Then we have that $z\in \ell_b$. Further, because the length of side $(0,a_4) , a$ is $|\tau(a)|$ and the length of side $(0,a_4), \vphi(a)$ is $|\tau(\vphi(a))|$, we get that $z = \vphi(a) + \frac{|\tau(\vphi(a))|}{|\tau(a)|}(b-a) = \vphi(a) + R(a)(b-a)$. Thus, $\vphi(a) + R(a)(b-a)\in \ell_b$. Further, since $\phi_{b,a}$ is a projection in the $\eta_b$ direction, we have that (\ref{poly.17}) holds.

\begin{figure}[htb]
\begin{center}
\includegraphics[width=250pt]{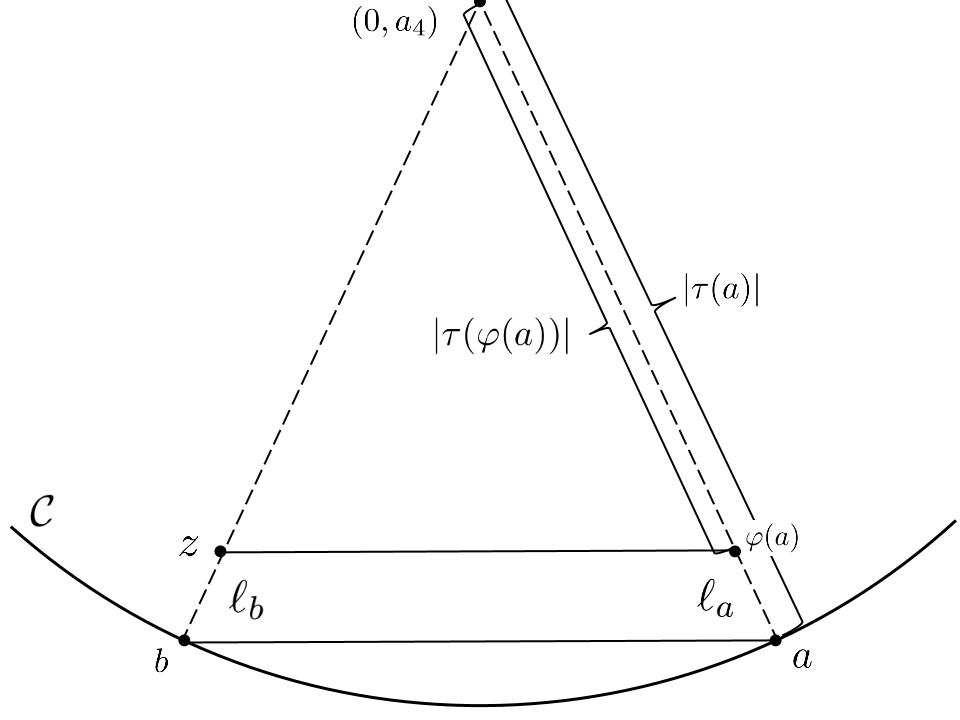}
\caption{
The point $z$ in the argument above; $z = \vphi(a) + \frac{|\tau(\vphi(a))|}{|\tau(a)|}(b-a)$}
\end{center}
\end{figure}

Note $\vphi(b)\in\ell_b$. By (\ref{poly.17}), we have that $\vphi(b)-\vphi(a) - \phi_{b,a}(R(a)(b-a))$ is a scalar multiple of $\eta_b$ and $\vphi(a)+\phi_{b,a}(R(a)(b-a))\in P(\vphi(a))$. Thus, as in the radial case, we get that
\begin{equation}\label{poly.18}
|\vphi(b)-\vphi(a) - \phi_{b,a}(R(a)(b-a))| = \sec(\meang(\eta_b,\lambda_a)) d(\vphi(b),P(\vphi(a)))\sleq C|b-a|^{1+\beta_2}.
\end{equation}
Thus, we have established (\ref{poly.14}), and thus established (a slightly stronger version of) (\ref{poly.1}) for $\theta$ separated points.

\begin{figure}[htb]
\begin{center}
\includegraphics[width=350pt]{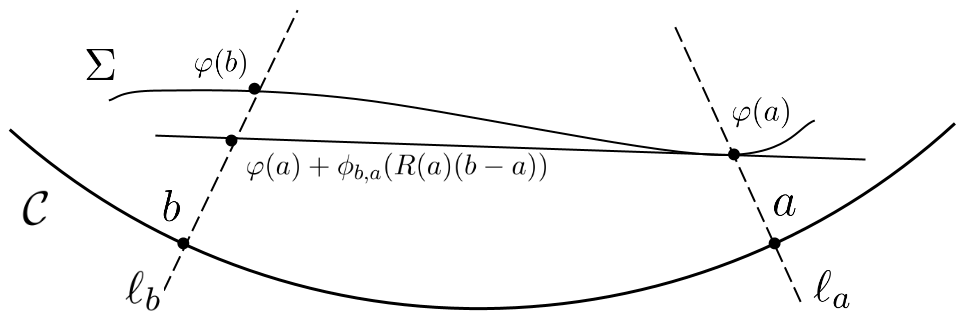}
\caption{The argument for $\theta$ separated points.}
\end{center}
\end{figure}

Finally, we consider general points $a,b\in \cC\cap O$ with $|a-b|\sleq |a|^{1+\gamma}/A$. We define $c = (\tau(a),b_4)$. Note that $a$ and $c$ are radially separated, $b$ and $c$ are $\theta$ separated, and $|a-c|,|c-b|\sleq |a-b|$. We compute
$P_b(b) - P_a(b) = P_b(b)-P_c(b) +P_c(b)-P_a(b).$
We expand to get
\begin{equation}\label{poly.20}
\begin{split}
P_c(b)-P_a(b)& = \vphi(c) + M_c(b-c) - \vphi(a) - M_a(b-a)\\
& = \vphi(c)-\vphi(a) - M_a(c-a) + M_a(c-a) - M_a(b-a) + M_c(b-c)\\
& = P_c(c)-P_a(c) + M_a(c-b) - M_c(c-b).
\end{split}
\end{equation}
Thus, by (\ref{poly.20}), we get that
\begin{equation}\label{poly.21}
|P_b(b)-P_a(b)|\sleq |P_b(b)-P_c(b)| + |P_c(c)-P_a(c)| + ||M_a-M_c||\cdot |c-b|.
\end{equation}
We now use that (\ref{poly.1}) holds for radially and $\theta$ separated points, Lemma \ref{Mnorm2}, and (\ref{poly.21}) get that
\begin{equation}\label{poly.22}
|P_b(b)-P_a(b)|\sleq C|b-c|^{1+\frac{\beta_3}{1+\gamma}} + C |c-a|^{1+\frac{\beta_3}{1+\gamma}} + C|a-c|\frexp{\beta_3}{1+\gamma}|c-b|.
\end{equation}
We now apply that $|b-c|,|c-a|\sleq |a-b|$ to get that $|P_b(b)-P_a(b)|\sleq C|a-b|^{1+\frac{\beta_3}{1+\gamma}}$ for all points $a,b\in \cC\cap O$.

\end{proof}

We now have gathered all of the estimates necessary to prove Theorem \ref{parametrization}.

\begin{proof}[Proof of Theorem \ref{parametrization}]
Let $\beta = {\beta_3}/({1+\gamma})$. First, note that $\vphi(a) = P_a(a)$. By Lemmas \ref{Mnorm2} and \ref{poly}, we have that Theorem \ref{Whitney} says that $\vphi$ extends to a $C^{1,\beta}$ map on $\RR^4$ (which we also call $\vphi$) such that $\vphi(a) = P_a(a)$ and $D_a\vphi = D_a P_a$ for all $a\in \cC\cap O$. Because $\vphi$ is a bijection from $\cC\cap O$ to $\Sigma\cap O$ (see Lemma \ref{bilip}), and $D_aP_a = M_a$ which is always of full rank (see definition of $M_a$), we have that there is some open set $U$ containing $\cC\cap O$ such that $\vphi$ is a diffeomorphism on $U$. Taking $U' = \vphi(U)$ completes the proof.

\end{proof}

\begin{remark}
We finish by reminding ourselves that Theorem \ref{theorem} is a local theorem. For example, Figure 7 (which is of course a dimension short) shows a set which at every point is smoothly parametrized by a KP cone or a plane.
\begin{figure}
\begin{center}
\includegraphics[width = 200pt]{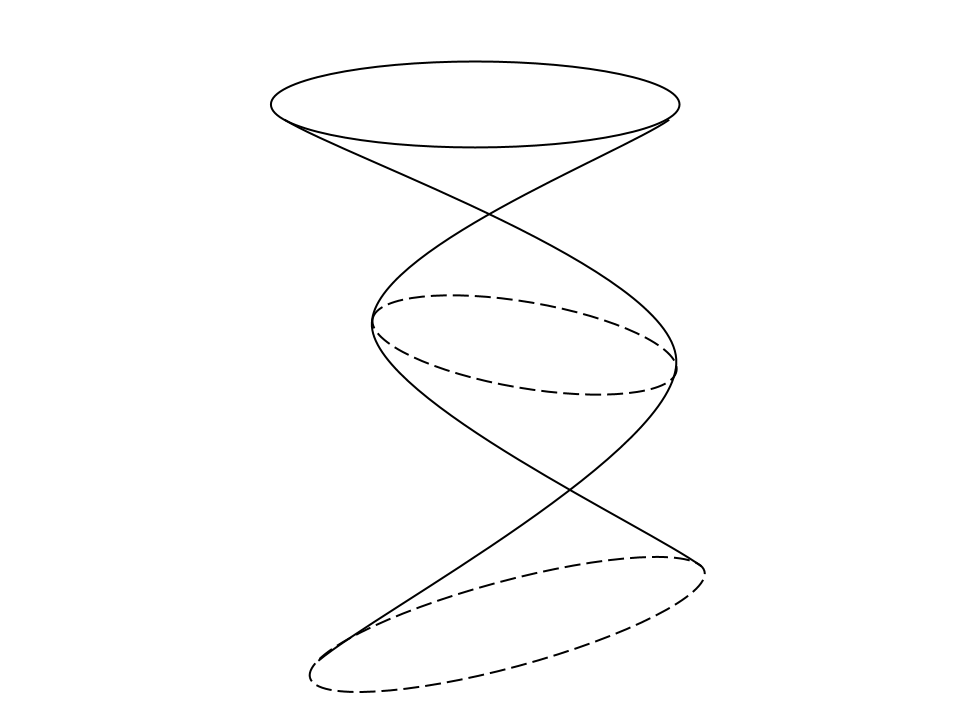}
\caption{A set locally smoothly parametrized by planes and KP cones.}
\end{center}
\end{figure}
\end{remark}

\appendix
\section{Appendix}

\begin{lemma}\label{C2}
Suppose that $U,V\subseteq \RR^n$ are open sets, $\Gamma\subseteq \RR^n$, $0<m<n$, and $\psi\in C^2(U,V)$ is bijective and satisfies
\begin{equation}\label{C2.1}
0<\lambda\sleq \frac{|\psi(x)-\psi(y)|}{|x-y|}\sleq \Lambda\mbox{ for all } x,y\in U
\quad
\mbox{ and }\quad
||D^2\psi^a||_\infty = \sup_{x\in U}||D_x^2\psi||<\infty.
\end{equation}
Let $z\in \Gamma\cap U$, $B(z,r)\subseteq U$, $P$ be a plane of dimension $m$ through $z$, and set $\twid P = D_z\psi (P-z) + \psi(z)$, $\twid\Gamma = \psi(\Gamma)$. Then
\begin{equation}\label{C2.3}
\hd{\psi(z),\lambda r}{\twid \Gamma}{\twid P}\sleq \frac{||D^2 \psi||_\infty}{2\lambda}r + \frac{\Lambda}{\lambda}\hd{z,r}{\Gamma}{P}.
\end{equation}
\end{lemma}

\begin{proof}
Without loss of generality, take $z = \psi(z) = 0$. Let $P\subseteq \RR^n$ be an $m$-plane through $0$ and set $d = \hd{0,r}{\Gamma}P.$
Note that $\lambda\sleq \frac{|D_0\psi (v) |}{|v|} \sleq \Lambda$ for all $v\in \RR^n \setminus\{0\}$ by (\ref{C2.1}). Further, note that we have
\begin{equation}\label{C2.5}
B(0,\lambda r)\subseteq \psi(B(0,r)).
\end{equation}
Let $y\in \twid \Gamma\cap B(0,\lambda r)$.  Then by (\ref{C2.5}) and bijectivity, we have that there exists $x\in B(0,r)\cap \Gamma$ such that $y=\psi(x)$. By $d = \hd{0,r}{\Gamma}P$, we get that there exists $p\in P$ such that $|p-x|\sleq rd.$ Let $\twid p = D_0\psi(p)\in \twid P$. We compute
\begin{equation}\label{C2.7}
\begin{split}
|\twid p-y| & = |D_0\psi(p) - \psi(x)|\sleq |D_0\psi (p) - \psi(p)| + |\psi(p) - \psi(x)|.
\end{split}
\end{equation}
By (\ref{C2.1}) and Taylor expansion, we get that
\begin{equation}\label{C2.8}
|\twid p-y|\sleq \frac{||D^2\psi||_\infty}{2}|p-0|^2 + \Lambda |p-x|\sleq \frac{||D^2\psi||_\infty}{2}r^2 + \Lambda  rd.
\end{equation}
Because for every $y\in B(0,\lambda r)\cap \twid \Gamma$, there exists $\twid p\in \twid P$ satisfying (\ref{C2.8}), we get that
\begin{equation}\label{C2.9}
\hda{0,\lambda r}{\twid \Gamma}{\twid P}\sleq  \frac{||D^2\psi||_\infty}{2\lambda}r + \frac{\Lambda}{\lambda} d .
\end{equation}
By Lemma \ref{ConeG1}, we get that (\ref{C2.9}) tells us (\ref{C2.3}).
\end{proof}

\begin{lemma}\label{coords}
For $a\in\cC\setminus\{0\}$, and $A>0$ large enough, there exists a neighborhood $U\supseteq B(a, 2|a|/{A})$, $V\subseteq\RR^3$ open, $I\ni 0$ an open interval, and a smooth coordinate map $\psi^a: U\to V\times I$ such that $V\times \{0\} = \psi^a(\cC\cap U)$ and $\twid \pi = \psi^a\circ\pi\circ(\psi^a)\inv$ is orthogonal projection onto $\RR^3\times\{0\}$ (where $\pi$ is the same map defined in Section 5; see (\ref{nearestpoint})). Further, $\psi^a$ satisifes the estimates
\begin{equation}\label{coords.1}
\frac12\sleq \frac{|\psi^a(x)-\psi^a(y)|}{|x-y|}\sleq 2\qquad \mbox{ for all }x,y\in U
\end{equation}
and
\begin{equation}\label{coords.2}
||D^2\psi^a||_\infty = \sup_{x\in U} ||D^2_x\psi^a|| \sleq \frac C{|a|}
\end{equation}
for some $C$ independent of $a$.
\end{lemma}

\begin{proof}
First, we fix an $a\in \cC$, $|a| = 1$. We define $\psi^a$ by defining its inverse. Choose orthonormal coordinates $(z_1, z_2, z_3)$ on $T_a\cC$ centered at $a$. Let $p$ be orthogonal projection from $\cC$ onto $T_a\cC$, and take $U'\supseteq B(a, 8/A)$ to be an open set where $p\inv$ is defined. Identifying $T_a\cC$ with $\RR^3$ under the $z$ coordinates, let $V = U'\cap T_a\cC$. Let $I = (-8/A,8/A)$. Define for $(z,t)\in V\times I$
\begin{equation}\label{coords.3}
(\psi^a)\inv(z,t) = p\inv(z) + t\eta_{p\inv(z)}.
\end{equation}

Assume that $A\sgeq 16$, so that $8/A\sleq 1/2$. Note that $\psi^a$ is bijective onto $U  = (\psi^a)\inv(V\times I)$, because $\eta$ is a smooth vector field (away from the $x_4$ axis), and the point $(z,t)$ is the flow after time $t$ of the point $p\inv z$ along the integral curves of $\eta$. Further, the same comments imply that it is smooth. Then we note that because $p\inv(z)\in \cC$, $V\times\{0\} = \psi^a(\cC\cap U)$. Further, $\twid\pi (z,t) = \psi^a(\pi( p\inv(z) + t\eta_{p\inv(z)})) = \psi^a(p\inv(z)) = z$, and so $\twid\pi$ is orthogonal projection onto $V\times \{0\}$.  

We now show that (\ref{coords.1}) holds for $\psi^a$ as long as $U'$ is chosen small enough and $1/A$ is chosen small enough. Continuing the identifcation of $T_a\cC$ with $\RR^3$, we set $e_i$ to be the coordinate vector of $z_i$, and note that $D_a\psi^a$ is the map
\begin{equation}\label{coords.4}
D_a\psi^a(e_i) = e_i,\quad D_a\psi^a \eta_a = e_4\qquad\mbox{ for } i =1,2,3.
\end{equation}
Because the $z_i$ are orthonormal and $\meang(\eta_a, T_a\cC) = \pi/4$, we get that 
\begin{equation}\label{coords.5}
\langle e_i, e_j\rangle = \delta_{ij}, \quad |\langle e_i, \eta_a\rangle|\sleq 1/{\sqrt 2}.
\end{equation}
From (\ref{coords.4}) and (\ref{coords.5}), as well as recalling that $|\nu_a| = 1$, we get that
\begin{equation}\label{coords.6}
\frac 1{\sqrt2} \sleq \frac{|D_a\psi^a v|}{|v|}\sleq \sqrt 2\quad\mbox{ for } v\in \RR^4\setminus\{ 0\}.
\end{equation}
Because $\psi^a$ is smooth, it follows from (\ref{coords.6}) that for $U'$ small enough,
\begin{equation}\label{coords.7}
\frac 12\sleq \frac{|\psi^a(x) - \psi^a(y)|}{|x-y|} \sleq 2 \quad\mbox{ for }x,y\in U'.
\end{equation}
Thus if $A$ is large enough, $B(a, 8/A)\subseteq U'$. By definition of $U$ and (\ref{coords.7}), we have that $B(a, 2/A)\subseteq U$. Because $\psi$ is $C^2$, by potentially restricting to a compactly contained open set, we may assume
\begin{equation}\label{coords.8}
C = ||D^2\psi^a||_\infty< \infty.
\end{equation}

Let $b\in \cC$, $|b| = 1$. Then there is a rotation $O\in O(4)$ taking $b$ to $a$ and fixing $\cC$. Define $\psi^b = O\inv\circ \psi^a\circ O$. For $b\in \cC\setminus \{0\}$, we define $\psi^b = |b|\psi^{\frac b{|b|}}(\cdot/{|b|})$. Note that $\psi^{\frac b{|b|}}$ satisfies $(\ref{coords.7})$ and $(\ref{coords.8})$, we have that $\psi^b$ satisfies $(\ref {coords.1})$ and $(\ref{coords.2})$.

\end{proof}

\end{document}